\documentclass[a4paper, reqno, 10pt]{amsart}
\usepackage{amsfonts}
\usepackage{graphicx}
\usepackage{amscd}
\usepackage{color}
\usepackage{amsmath}
\usepackage{amsfonts,amssymb,amsthm,mathrsfs,xypic,cite,leftidx,comment,enumitem}
\xyoption{curve}
\usepackage{fancybox}
\usepackage{tikz}
\usepackage{tikz-cd}
\usetikzlibrary{arrows,decorations.pathmorphing,backgrounds,positioning,fit,petri}
\usetikzlibrary{arrows,patterns,matrix}
\usetikzlibrary{decorations.markings}
\usepackage[all,cmtip]{xy}
\usepackage{caption}
\usepackage{subcaption}
\usepackage{leftidx}
\newcommand{\Gproj}{{\rm Gp}}

\newcommand{\gldim}{{\rm gl.dim}}
\newcommand{\gdim}{{\rm gl.dim}}
\newcommand{\Gp}{{\rm Gp}}

\newcommand{\Hom}{\operatorname{Hom}}
\newcommand{\End}{\operatorname{End}}

\newcommand{\Ker}{\operatorname{Ker}}
\newcommand{\cok}{\operatorname{Coker}}
\newcommand{\Ima}{\operatorname{Im}}
\newcommand{\rad}{\operatorname{rad}}
\newcommand{\Tor}{\operatorname{Tor}}
\newcommand{\Ext}{\operatorname{Ext}}

\newcommand{\add}{\operatorname{add}}
\newcommand{\op}{\operatorname{op}}

\newcommand{\rep}{{\rm rep}}

\newcommand{\Mon}{{\rm Mon}}
\newcommand{\modd}{\textendash\operatorname{mod}}
\newcommand{\dmod}{\operatorname{mod}\textendash}

\newcommand{\Refl}{\operatorname{Refl}}

\newcommand{\lxr}{\longrightarrow}
\newcommand{\iso}{\cong}

\newcommand{\modu}{\mbox{-}\operatorname{mod}}
\newcommand{\im}{\operatorname{Im}}

\DeclareMathOperator{\pd}{proj.dim}%
\DeclareMathOperator{\id}{\operatorname{Id}}
\DeclareMathOperator{\domdim}{domdim.}%
\DeclareMathOperator{\Dom}{Dom}%
\DeclareMathOperator{\findim}{fin.dim.}%

\newcommand{\calL}{{\mathcal L}}

\newtheorem{thm}{Theorem}[section]

\newtheorem{cor}[thm]{Corollary}

\newtheorem{lemdef}[thm]{Lemma-Definition}

\newtheorem{lem}[thm]{Lemma}
\newtheorem{exam}[thm]{Example}

\newtheorem{prop}[thm]{Proposition}
\newtheorem{rem}[thm]{Remark}
\newtheorem{defn}[thm]{Definition}

\title[Weakly Gorensteinness of tensor algebras and Morita algebras]{Weakly Gorensteinness  of \\ tensor algebras and Morita algebras}
\author{Nan Gao, \ \ Pu Zhang$^*$, \ \ Shijie Zhu}
\thanks{$^*$ Corresponding author}
\thanks{nangao@shu.edu.cn  \ \ \ \ pzhang$\symbol{64}$sjtu.edu.cn \ \ \ \ shijiezhu@ntu.edu.cn}
\thanks{Supported by  National Natural Science Foundation of China $($Grant No. 12131015, 12201321, 12271333$)$.}

\address{Nan Gao, Department of Mathematics, Shanghai University, Shanghai 200444, PR China}
\address{Pu Zhang, School of Mathematical Sciences, Shanghai Jiao Tong University, Shanghai 200240, PR China}
\address{Shijie Zhu, School of Mathematics And Statistics, Nantong University, Nantong 226019, PR China}

\begin{document}

\maketitle

\begin{abstract} \ An algebra $A$ is left weakly Gorenstein if any semi-Gorenstein-projective left $A$-modules
is Gorenstein-projective. The weakly Gorensteinness of two kinds of algebras are answered. Using the method of the monomorphism category,
it is proved that the tensor algebra $A\otimes B$ with ${\rm gl.dim} B< \infty$ is left weakly Gorenstein if and only if so is $A$.
For a class of Morita algebras $\Lambda=\begin{pmatrix}\begin{smallmatrix}
A & N \\
M & B \\
\end{smallmatrix}\end{pmatrix}$, the (semi-)Gorenstein-projective left $\Lambda$-modules are computed and described; and then it is proved that  $\Lambda$ is left weakly Gorenstein if and only if so are $A$ and $B$.
As an application, the upper triangular matrix algebra $T_n(A)$ is left weakly Gorenstein if and only if so is $A$.

\vskip5pt


\vskip5pt

\end{abstract}

\section{\bf Introduction}

It is a fundamental problem to judge whether an algebra is left weakly Gorenstein, which is hard in general due to the lack of explicit characterizations of
(semi-)Gorenstein-projective modules. Meanwhile, the problem has its own flavor and significance, for their wide applications, for example in the representation theory,
relative homological algebra, and the singularity theory.
See e.g. \cite{Buch}, \cite{AR1}, \cite{EJ2}, \cite{V}, \cite{Bel}, \cite{Chr}, \cite{Chen}, \cite{RZ2}.

\vskip5pt

The aim of this paper is to answer this problem for two classes of algebras:
one is the tensor algebra $\Lambda = A\otimes_k B$ with ${\rm gl.dim} B< \infty$, where $A$ and $B$ are finite-dimensional algebras over a field $k$, and another is the Morita algebra $\Lambda = \Lambda_{(\phi, \psi)}=\begin{pmatrix}\begin{smallmatrix}
A & N \\
M & B \\
\end{smallmatrix}\end{pmatrix}$ with specified datum $(A, B, M, N, \phi, \psi)$.
Homological theoretic aspect of tensor algebras has got importance since
H. Cartan and S. Eilenberg's work \cite{CE56}. The Morita algebras, formulated by H. Bass \cite{B}, originating from
the Morita equivalences \cite{Mor}, are widely used in many branches of mathematics.

\vskip5pt

We apply different methods for the two classes of algebras. Since there are no descriptions for all the modules over a tensor algebra,
we adopt conceptual and constructive arguments via the so-called
monomorphism category. This monomorphism categories have been developed from C. M. Ringel and M. Schmidmeier's work \cite{RS}, and D. Simson's work \cite {S} on submodule categories.
While for the Morita algebras $\Lambda$,  we heavily  use the diagram description of $\Lambda$-modules given by E. L. Green \cite{Gr} (see also \cite{GrP}),  homological computations, and analysis on exactness.

\subsection{Gorenstein-projective modules  and left weakly Gorenstein algebras}

A module $M$ over a ring $R$ is  {\it semi-Gorenstein-projective} if $\Ext^i_R(M, R)=0$ for $i\ge 1$.
Consider the $R$-dual $M^* = \Hom_R(M,R)$ and the canonical $R$-homomorphism $\phi_M: M\longrightarrow M^{**}$, given by $\phi_M(m)(f) = f(m)$ for $m\in M$ and $f\in M^*$.
If  $\phi_M$ is injective (bijective, respectively), then $M$ is called {\it torsionless} ({\it reflexive}, respectively).
For a finitely generated left module $M$ over a two-sided noetherian ring $R$, M. Auslander [A] has introduced the following conditions:

\vskip5pt

(G1) \  $M$ is semi-Gorenstein-projective;

(G2)  \ $M^*$ is a semi-Gorenstein-projective right $R$-module; and

(G3) \ $M$ is reflexive.

\vskip5pt

Let $\mathcal A$ be an abelian category with enough projective objects. {\it A complete projective resolution} in $\mathcal A$ is an acyclic complex \ $P^\bullet: \ \cdots\longrightarrow P^{-1}\longrightarrow P^0\stackrel{d_0}\longrightarrow P^1\longrightarrow \cdots $ of projective objects of $\mathcal A$, such that $\Hom_\mathcal A(P^\bullet, P)$ is again acyclic for any projective object $P$ of $\mathcal A$.
An object $G$ is {\it Gorenstein-projective}, if there is a complete projective resolution $P^\bullet$ such that $G\cong {\rm Ker} d_0.$ See E. E. Enochs and O. M. G. Jenda  \cite{EJ1}.

\vskip5pt

Throughout this paper, unless otherwise specified, we consider finitely generated modules over Artin algebras.
Let $A\mbox{-}{\rm mod}$ (${\rm mod}A$, respectively) be the category of finitely generated left (right, respectively) modules over Artin algebra $A$, and $\Gp(A)$ the full subcategory
of $A\mbox{-}{\rm mod}$ consisting of Goresnstein-projective objects in $A\mbox{-}{\rm mod}$, which are called {\it the Goresnstein-projective modules}. Let \ $^\perp A$ denote
the full subcategory of $A\mbox{-}{\rm mod}$ consisting of semi-Goresnstein-projective modules. All the tensor product will be over field $k$, if not  otherwise specified.

\vskip5pt

An important feature is that a finitely generated module $M$ over an Artin algebra $A$ is Gorenstein-projective if and only if $M$ satisfies
the conditions (G1), (G2), and (G3) (see L. W. Christensen [Chr, Theorem 4.2.6]; also M. Auslander and M. Bridger [AB, Proposition 3.8]).
It was a longstanding problem  whether the
conditions (G1), (G2), (G3) are independent (see L. L. Avramov and A. Martsinkovsky \cite [p.398]{AM}).
This problem has been completely solved recently: D. A. Jorgensen and L. M. \c Sega [JS] present the modules satisfying (G1) and (G3), but not (G2),
and also modules satisfying (G2) and (G3), but not (G1); modules satisfying (G1) and (G2), but not torsionless,
have been given by C. M. Ringel and P. Zhang \cite {RZ2}. It is also natural to study
semi-Gorensten-projective modules, torsionless modules, reflexive modules, and Gorenstein-projective modules, separately.

\vskip5pt

Following \cite{RZ2}, an algebra $A$ is {\it left weakly Gorenstein}, if \ $^\perp A = {\rm Gp}(A)$, i.e., any semi-Gorentein-projective left $A$-module is Gorenstein-projective.
Similarly, one has the notion of {\it a right weakly Gorenstein algebra}. It is an open problem whether a left weakly Gorenstein algebra is right weakly Gorenstein (\cite{Mar2}; \cite[Subsection 9.3]{RZ2}).
A Gorenstein algebra $A$ (i.e., ${\rm inj. dim} _AA < \infty$ and ${\rm inj. dim} A_A < \infty$) is left and right weakly Gorenstein (see \cite[Corollary 11.5.3]{EJ2}),
but the converse is not true (see e.g. \cite [Subsection 1.5]{GLZ}).
There are indeed many algebras which are not left weakly Gorenstein (see \cite {JS}, \cite {Mar1}, \cite {RZ2}).
For more information on left weakly Gorenstein algebras we refer to \cite[Theorems 1.2 - 1.4]{RZ2} and references there.

\subsection{\bf Torsionless and reflexive modules on tensor algebras}

Using the compatibility of the Cartan-Eilenberg isomorphism with the canonical homomorphism $\phi$ (see Proposition \ref{lem:phi_tensor}) and Proposition \ref{tensorker}, one can
prove
\begin{thm} \ {\rm (Theorem \ref{thm:main_tensor})} \ Let $A$ and $B$ be finite-dimensional $k$-algebras, $X\in A$-{\rm mod} and $Y\in B$-{\rm mod}. Then

\vskip5pt

$(1)$ \ If $X \ne 0 \ne Y$, then $X\otimes Y$ is a torsionless $(A\otimes B)$-module if and only if  $X$ is a torsionless $A$-module and $Y$ is a torsionless $B$-module.

\vskip5pt

$(2)$ \ If $X \ne 0 \ne Y$, then $X\otimes Y$ is a reflexive $(A\otimes B)$-module if and only if  $X$ is a reflexive $A$-module and $Y$ is a reflexive $B$-module.
\end{thm}

\subsection{Weakly Gorensteinness  of tensor algebras} \ Many $(A\otimes B)$-modules are not of the form $X\otimes Y$, where $X$ is an $A$-module and $Y$ is a $B$-module; and in general,
there are no descriptions for all the $(A\otimes B)$-modules. This is one of the main difficulty in dealing with the tensor algebras.

\vskip5pt

If $B = kQ/I$, where $Q$ is an acyclic quiver, and $I$ is an admissible ideal of the path algebra $kQ$,
then any $(A\otimes B)$-module can be described as a representation of bound quiver $(Q, I)$ over the algebra $A$.
From this viewpoint, the category $\Gp(A\otimes B)$ of the Gorenstein-projective $(A\otimes B)$-modules can described via the monomorphism category $\Mon(A, Q, I)$, when $I$ is an admissible ideal of $kQ$ generated by monomial relations (see \cite[Theorem 5.1]{LZ1}, \cite[Theorem 4.1]{LZ2}). This monomorphism category is defined via the combinatorics of the bound quiver $(Q, I)$, but also admits homological interpretation (see \cite[Theorem 2.6]{ZX}).
With this homological interpretation, the monomorphism category $\Mon(B, \ ^\perp A)$ can be also defined for any finite-dimensional $k$-algebra $B$ (which is not necessarily of the form $kQ/I$) in \cite [Definitions 3.1, 3.6] {HLXZ}.
For more applications of monomorphism categories we refer to \cite{KLM}, \cite{HZ}, \cite{GKKP}.

\vskip5pt

Using conceptual and constructive arguments via the monomorphism category $\Mon(B, \ ^\perp A)$,  we can judge the weakly Gorensteinness of the tensor algebra $A\otimes B$ with $\gldim B <\infty$.

\begin{thm} \ {\rm (Theorem \ref{thm:weak_G})} \ Let $A$ and $B$ be finite-dimensional $k$-algebras with $\gldim B <\infty$, and $\Lambda=A\otimes B$. Then

\vskip5pt

$(1)$ \ $\Lambda$ is left weakly Gorenstein if and only if $A$ is left weakly Gorenstein.

\vskip5pt

$(2)$ \ If $A$ is left weakly Gorenstein,  then $\leftidx^\perp \Lambda = \Mon(B,\leftidx^\perp A).$

\vskip5pt

$(3)$ \ If $B$ is not semisimple, then the converse of $(2)$ is also true.
\end{thm}

Note that the ``if" part of \ref{thm:weak_G}(1) does not need the condition $\gldim B <\infty$. Also, if $B$ is semi-simple, then $\leftidx^\perp \Lambda = \Mon(B,\leftidx^\perp A)$ holds true for an arbitrary finite-dimensional $k$-algebra. Thus, in this case, the converse of Theorem \ref{thm:weak_G}$(2)$ is not true. See Remark \ref{remweakG}.

\vskip5pt

Theorem \ref{thm:weak_G}(1) is  proved first for $B = kQ$ in \cite[Theorem 1.3(2)]{Z},
and  then for $B = kQ/I$ in \cite[Theorem C]{LuoZhu}, where $Q$ is an acyclic quiver, and $I$ is an admissible ideal of  $kQ$ generated by monomial
relations.

\subsection{Weakly Gorensteinness of Morita algebras}

We consider the Morita ring $\Lambda = \Lambda_{(\phi, \psi)}=\begin{pmatrix}\begin{smallmatrix}
A & _{A}N_{B} \\
_{B}M_{A} & B \\
\end{smallmatrix}\end{pmatrix}$ which is an Artin algebra (see subsection 5.1).
Any left $\Lambda$-module can be identified with a quadruple $(X, Y, f, g)$,
where $X\in A\mbox{-}{\rm mod}$ and $Y\in B\mbox{-}{\rm mod}$, $f\in {\rm Hom}_{B}(M\otimes_{A}X, Y)$ and $g\in {\rm Hom}_{A}(N\otimes_{B}Y, X)$ (see \cite[Theorem 1.5]{Gr}).
We will write it as  $\begin{pmatrix}\begin{smallmatrix}X\\Y\end{smallmatrix}\end{pmatrix}_{f, g}$, in this way the $\Lambda$-action coincides with the multiplication of matrices.
For more information on $\Lambda$-modules we refer to  \cite{Gr, KT, GrP, GP}.

\vskip5pt

Notice that any  ring can be written as a Morita ring. Rather than working in such a very broad setting, one should usually impose some specific conditions on the datum $(A, B, M, N, \phi, \psi)$ in practice.
Of our particular interest, we will focus on the Morita rings satisfying
$M\otimes_{A}N=0=N\otimes_{B}M$, and hence $\phi = 0 = \psi$. This condition looks quite restrictive
at first glance, but in fact it contains many algebras we are interested, for example, the triangular matrix algebras.

\vskip5pt

Through careful homological computations and analysis on exactness,  we obtain the following description of the semi-Gorenstein-projective $\Lambda\mbox{-}$modules.

\begin{thm} \ {\rm (Theorem \ref{semi})} \ Let $\Lambda=\begin{pmatrix}\begin{smallmatrix}
A & N \\
M & B \\
\end{smallmatrix}\end{pmatrix}$ be a Morita ring with $M\otimes_{A}N=0=N\otimes_{B}M$. Assume that \ $_BM, \ _AN$, \ $M_{A}$ and $N_{B}$ are projective modules. Then $\begin{pmatrix}\begin{smallmatrix}X\\Y\end{smallmatrix}\end{pmatrix}_{f, g}\in \ ^\perp \Lambda$  if and only if the following conditions are satisfied$:$

\vskip 5pt

{\rm (1)} \  ${\rm Hom}_A(g, A):{\rm Hom}_{A}(X, A)\lxr {\rm Hom}_{A}(N\otimes_{B}Y, A)$ is an epimorphism$;$

\vskip 5pt

{\rm (2)} \ ${\rm Ext}_{A}^{i}(g, A): {\rm Ext}_{A}^{i}(X, A)\longrightarrow {\rm Ext}_{A}^{i}(N\otimes_{B}Y, A)$ is an isomorphism for $i\geq 1;$

\vskip 5pt

{\rm (3)} \  ${\rm Hom}_B(f, B):{\rm Hom}_{B}(Y, B)\lxr {\rm Hom}_B(M\otimes_AX, B)$ is an epimorphism$;$

\vskip 5pt

{\rm (4)} \   ${\rm Ext}_{B}^{i}(f, B):{\rm Ext}_{B}^{i}(Y, B)\longrightarrow {\rm Ext}_{B}^{i}(M\otimes_{A}X, B)$ is an isomorphism for $i\geq 1$.

\end{thm}

\vskip5pt

A left $\Lambda$-module $\begin{pmatrix}\begin{smallmatrix}X \\ Y\end{smallmatrix}\end{pmatrix}_{f, g}$ is said to be {\it monic}, if
 $f:M\otimes_{A}X\lxr Y$ and $g:N\otimes_{B}Y\lxr X$ are injective maps.
 The following result, which is a consequence of Q. Q. Guo and C. C. Xi \cite[Proposition 3.14]{GX24},  gives the description of Gorenstein-projective modules over the Morita rings in Theorem 1.3.

\begin{prop} \ {\rm (Corollary \ref{gpmorita})} \ Let $\Lambda=\begin{pmatrix}\begin{smallmatrix}
A & N \\
M & B \\
\end{smallmatrix}\end{pmatrix}$ be a Morita ring with $M\otimes_{A}N=0=N\otimes_{B}M$. Assume that \ $_BM, \ _AN$, \ $M_{A}$ and $N_{B}$ are projective modules.
Then $\begin{pmatrix}\begin{smallmatrix}X\\Y\end{smallmatrix}\end{pmatrix}_{f, g}\in \Gp(\Lambda)$ if and only if it is a monic $\Lambda\mbox{-}$module with $\cok f\in \Gp(B)$ and $\cok g\in \Gp(A)$.
\end{prop}

\vskip5pt

The following observation gives a featured property of left weakly Gorenstein algebras.

\begin{thm}\ {\rm (Theorem \ref{wgmono})} \ Let $A$ be a left weakly Gorenstein algebra. If $\phi: X\longrightarrow Y$ is a left $A$-homomorphism such that $\phi^*=\Hom_A(\phi,A)$ is an epimorphism and
$\Ext^i_A(\phi,A)$ is an isomorphism for $i\geq 1$, then $\phi$ is a monomorphism and $\cok\phi$ is a Gorenstein-projective $A$-module.
\end{thm}

Theorem 1.3, Proposition 1.4, and Theorem 1.5  enable us to judge the weakly Gorensteinness of a class of Morita algebras:
the following result gives a satisfied answer to the question: when are the Morita algebras in Theorem 1.3 left weakly Gorenstein?

\vskip5pt

\begin{thm} \ {\rm (Theorem \ref{weakGmorita})} \ Let $\Lambda=\begin{pmatrix}\begin{smallmatrix}
A & N \\
M & B \\
\end{smallmatrix}\end{pmatrix}$ be a Morita ring with $M\otimes_{A}N=0=N\otimes_{B}M$. Assume that \ $_BM, \ _AN$, \ $M_{A}$ and $N_{B}$ are projective modules. Then $\Lambda$ is left weakly-Gorenstein if and only if  $A$ and $B$ are left weakly-Gorenstein.
\end{thm}

Let $T_n(A)$ ($n\ge 2$) be the triangular matrix algebra
$\left(\begin{smallmatrix}
A&A&\cdots&A&A\\
0&A&\cdots&A&A\\
\vdots&\vdots&\ddots&\vdots&\vdots\\
0&0&\cdots&A&A\\ 0&0&\cdots&0&A
\end{smallmatrix}\right)$. As an application one has the following results.

\vskip10pt

\begin{cor} \ {\rm (Corollary \ref{Tn(A)})} \  $(1)$ \ \ A left $T_n(A)$-module $\left(\begin{smallmatrix}
X_n\\
\vdots\\
X_1
\end{smallmatrix}\right)_{(\phi_i)}$ is semi-Gorenstein-projective if and only if
\ $\phi^*_i: \ \Hom_A(X_{i+1}, A) \longrightarrow \Hom_A(X_i, A)$
are epimorphisms for $1\le i\le n-1$, and $X_i\in \ ^\perp A$ for
all $i$.

\vskip5pt

$(2)$ \ $\left(\begin{smallmatrix}
X_n\\
\vdots\\
X_1
\end{smallmatrix}\right)_{(\phi_i)}$ is a  Gorenstein-projective left $T_n(A)$-module if
and only if $X_i\in \Gp(A)$ for all
$i$, $\phi_i: X_i \longrightarrow X_{i+1}$ are monomorphisms, and
$\cok\phi_i\in \Gp(A)$ for $1\le i\le
n-1$.

\vskip5pt

$(3)$ \  $T_n(A)$ is a left weakly Gorenstein algebra if and only if so is $A$.\end{cor}

\subsection{The organization} The paper is organized as follows. Section 2 deals with torsionless or reflexive modules
of the form $X\otimes Y$ over tensor algebras $A\otimes B$, and proves Theorem 1.1 and Proposition \ref{tensorsgp}.
In Section 3 we develop properties of the monomorphism category and then prove Theorem 1.2.
Section 4  is devoted to the proof of Theorem 1.3.
In Section 5  the weakly Gorensteinness of the Morita algebras in Theorem 1.3 is studied,  Theorems 1.4 and 1.5 are proved, and Examples 5.8 - 5.11 are given to illustrate the application of the main results.

\vskip5pt

The appendix provides a direct and elementary proof for the Cartan-Eilenberg isomorphism {\rm (\cite[p.209, p.205]{CE56})},
based on the  K\"unneth formula for the homology of the tensor product of complexes.

\section{\bf Modules over tensor algebras}

\subsection{Some facts}  We need the following facts. For a module $M$, let $\add M$ be the class of modules which are direct summands of a finite direct sum of $M$.

\begin{lem} \label{torless} \  Let $A$ be a finite-dimensional $k$-algebra and $M\in A\mbox{-}{\rm mod}$.  Then the following are equivalent$:$

$(1)$ \ $M$ is torsionless.

$(2)$ \ $M$ is a submodule of some projective $A$-module $P$.

$(3)$ \ Any left $(\add A)$-approximation of $M$ is a monomorphism.
\end{lem}

 Following \cite{RZ2}, an exact sequence  $0\longrightarrow X\stackrel{f}\longrightarrow Y\longrightarrow Z\longrightarrow 0$ in $A$-mod
is an {\it approximation sequence}, if $f$ is a left $(\add A)$-approximation.
The following fact shows a connection between torsionless and reflexive modules.

\begin{lem}\label{lem:refl_tors}{\rm (\cite[2.4(a)]{RZ2})} \ Let \ $0\longrightarrow X\stackrel{f}\longrightarrow P\longrightarrow M\longrightarrow 0$ be an approximation sequence. Then $X$ is reflexive if and only if $M$ is torsionless.
\end{lem}

\begin{lem}\label{lem:GP_app} \
If $0\longrightarrow G\stackrel{f}\longrightarrow P\longrightarrow G'\longrightarrow 0$ is an exact sequence such that
 $G$ is a Gorenstein-projective $A$-module and $f$ is a left $(\add A)$-approximation of $G$, then $G'$ is also Gorenstein-projective.
\end{lem}
\begin{proof} \ For convenience we include a proof. Since $f$ is a left approximation and $\Ext_A^i(G,A)=0$ for $i\ge 1$, it follows that $\Ext^i_A(G',A)=0$ for $i\ge 1$.
Therefore
one has the exact sequence $0\longrightarrow G'^* \longrightarrow P^* \longrightarrow G^*\longrightarrow 0$. Since $\Ext_A^i(G^*, A)=0$ for $i\ge 1$,
it follows that $\Ext^i_A(G'^*, A)=0$ for $i\ge 1$. Thus one has a commutative diagram
with exact rows
$$\xymatrix{
0\ar[r]& G \ar[r]^f\ar[d]^-{\phi_{G}}_-\cong & P\ar[r]\ar[d]^-{\phi_{P}}_-\cong & G'\ar[r]\ar[d]^-{\phi_{G'}}&0\\
0\ar[r]& G^{**} \ar[r]& P^{**}\ar[r]& G'^{**}\ar[r]&0.
}
$$
Hence $\phi_{G'}$ is an isomorphism, i.e., $G'$ is reflexive.
Therefore $G'$ satisfies (G1), (G2), (G3),  so $G'$ is Gorenstein-projective.
\end{proof}

\subsection{The Cartan-Eilenberg isomorphism}

The Cartan-Eilenberg isomorphism is a powerful tool in studying modules over tensor algebras.

\begin{thm}\label{CEiso}{\rm (\cite[Theorem 3.1, p.209, p.205]{CE56})} \ Let $A$ and $B$ be finite-dimensional algebras over a field $k$,   $L,M\in A$-{\rm mod},  and $U,V\in B$-{\rm mod}. Then there is an isomorphism of $k$-spaces
$$\Ext^n_{
A\otimes B}(L \otimes U, M\otimes V )\cong
\bigoplus_{p+q=n}
(\Ext^p_
A(L, M) \otimes \Ext^q_
B(U, V )), \ \forall \ n \geq 0.$$
\end{thm}

A direct and elementary proof of Theorem \ref{CEiso} will be included in the appendix.
For later use we need to write out the concrete isomorphism for the case $n = 0$. It is surprising that the inverse of this isomorphism can not be written out explicitly.

\begin{cor}\label{homiso}{\rm(\cite[Theorem 3.1, p.209, p.205]{CE56})}
Let $A,B$ be finite-dimensional $k$-algebras, $L,M\in A$-{\rm mod} and $U,V\in B$-{\rm mod}. Then there is an isomorphism
$$\chi_{_{L, U, M, V}}: \ \ \Hom_A(L, M) \otimes \Hom_B(U, V) \longrightarrow \Hom_{A\otimes B}(L \otimes U, M\otimes V)$$
given by \ $g\otimes h \mapsto ``l\otimes u \mapsto g(l)\otimes h(u)".$
\end{cor}

\vskip5pt

The following result will be used in the proof of Theorem \ref{thm:weak_G}, one of the main results of this paper, and it is also of independent interest.

 \begin{prop}\label{lem:app_tensor} \ Let $A$ and $B$ be finite-dimensional $k$-algebras and $\Lambda=A\otimes B$. Let $X$ and $Y$ be finitely generated modules over $A$ and $B$, respectively.
If $f: X\longrightarrow  P$ is a left $(\add A)$-approximation and $g:Y\longrightarrow Q$ is a left $(\add B)$-approximation, then
$f\otimes g: X\otimes Y\longrightarrow P\otimes Q$
is a left $(\add \Lambda)$-approximation.
\end{prop}
\begin{proof}  \ It suffices to show that for any direct summand $T$ of $_\Lambda\Lambda$, any homomorphism $\varphi:X\otimes Y\longrightarrow T$ factors through $f\otimes g$.
Denote by $\iota:T\longrightarrow \Lambda$ the inclusion and $\pi:\Lambda\longrightarrow T$ its retraction. The  composition $\iota\circ\varphi\in \Hom_\Lambda(X\otimes Y, A\otimes B)$.
According to the Cartan-Eilenberg isomorphism, $\chi_{_{X, Y, A, B}}^{-1}(\iota\circ\varphi)\in\Hom_A(X,A)\otimes\Hom(Y,B)$.
Thus, there exist $s_i:X\longrightarrow A$ and $t_i:Y\longrightarrow B$ such that $\chi_{_{X, Y, A, B}}^{-1}(\iota\circ\varphi)= \sum\limits s_i \otimes t_i$.
Since $f$ is a left $(\add A)$-approximation and $g$ is a left $(\add B)$-approximation, there are homomorphisms $u_i:P\longrightarrow A$ and $v_i:Q\longrightarrow B$ such that each $s_i=u_if$ and $t_i=v_ig$ for each $i$. Denote by $\psi=\chi_{_{P, Q, A, B}}(\sum u_i\otimes v_i)\in\Hom_\Lambda(P\otimes Q, A\otimes B)$. Then for any $x\otimes y\in X\otimes Y$, on one hand,
$$
(\iota\circ\varphi)(x\otimes y)=\chi_{_{X, Y, A,B}}(\sum s_i \otimes t_i)(x\otimes y)=\sum s_i(x) \otimes t_i(y).
$$
On the other hand,
$$
\psi\circ(f\otimes g)(x\otimes y)=\chi_{_{P, Q, A, B}}(\sum u_i\otimes v_i)(f(x)\otimes g(y))=\sum u_if(x)\otimes v_ig(y)=\sum s_i(x)\otimes t_i(y).
$$
So $\iota\circ \varphi=\psi\circ(f\otimes g)$. Therefore, $\varphi=\pi\circ \iota \circ\varphi = \pi\circ \psi\circ(f\otimes g)$ factors through $f\otimes g$.
\end{proof}

\subsection{Torsionless modules and reflexive modules over tensor algebras}

The following fact shows the compatibility of the canonical homomorphism $\phi$ with the Cartan-Eilenberg isomorphism.

\begin{prop}\label{lem:phi_tensor}
Let $A$ and $B$ be finite-dimensional $k$-algebras,  $X\in A$-{\rm mod} and $Y\in B$-{\rm mod}. Then
$$
\phi_{X\otimes Y}=\Hom(\chi^{-1}_{_{X, Y, A, B}}, A\otimes B)\circ \chi_{_{X^*, Y^*, A, B}}\circ (\phi_X\otimes \phi_Y).$$
That is, the composition of the isomorphisms
\begin{align*}X\otimes Y & \stackrel{\phi_X\otimes \phi_Y}\cong X^{**}\otimes Y^{**}  = \Hom_A(\Hom_A(X,A), A)\otimes \Hom_B(\Hom_B(Y,B),B)
\\& \stackrel{\chi_{_{X^*, Y^*, A, B}}}\longrightarrow \Hom_{A\otimes B}(\Hom_A(X,A)\otimes \Hom_B(Y,B), A\otimes B)
\\& \stackrel{(\chi^{-1}_{_{X, Y, A, B}}, A\otimes B)}\longrightarrow \Hom_{A\otimes B}(\Hom_{A\otimes B}(X\otimes Y, A\otimes B), A\otimes B) = (X\otimes Y)^{**}\end{align*}
is precisely the canonical homomorphism
$\phi_{X\otimes Y}: X\otimes Y\longrightarrow (X\otimes Y)^{**}$,
where
$$\chi_{_{X^*,Y^*,  A, B}}: \Hom_A(X^*, A)\otimes \Hom_B(Y^*,B)\longrightarrow \Hom_{A\otimes B}(X^*\otimes Y^*, A\otimes B)$$
is explicitly given in  {\rm Corollary \ref{homiso}},  and
$$\chi^{-1}_{_{X, Y, A, B}}: \Hom_{A\otimes B}(X\otimes Y, A\otimes B)\longrightarrow \Hom_A(X,A)\otimes \Hom_B(Y,B).$$
\end{prop}
\begin{proof} \ We evaluate both sides on $x\otimes y\in X\otimes Y$ and $\alpha\in (X\otimes Y)^*$.

\vskip5pt

One the left hand side one has
$$\phi_{X\otimes Y}(x\otimes y)(\alpha) = \alpha(x\otimes y), \ \ \forall \ x\otimes y\in X\otimes Y, \ \ \forall \ \alpha\in (X\otimes Y)^*.$$

On the right hand side, one has
\begin{align*}& \{[(\chi^{-1}_{_{X, Y, A, B}}, A\otimes B)\circ \chi_{_{X^*, Y^*, A, B}}\circ (\phi_X\otimes \phi_Y)](x\otimes y)\}(\alpha)
\\& =\{[(\chi^{-1}_{_{X, Y, A, B}}, A\otimes B)\circ \chi_{_{X^*, Y^*, A, B}}](\phi_X(x)\otimes \phi_Y(y))\}(\alpha)
\\ & = \chi_{_{X^*, Y^*, A, B}}(\phi_X(x)\otimes \phi_Y(y))(\chi^{-1}_{_{X, Y, A, B}}(\alpha)) \end{align*}
where $\chi^{-1}_{_{X, Y, A, B}}(\alpha)\in X^*\otimes Y^*$, and $\chi_{_{X^*, Y^*, A, B}}(\phi_X(x)\otimes \phi_Y(y))\in (X^*\otimes Y^*)^*.$
To see it is precisely the canonical homomorphism   $\phi_{X\otimes Y}$, assume that $\chi^{-1}_{_{X, Y, A, B}}(\alpha) = \sum\limits_if_i\otimes g_i$.
Then by the definition of $\chi_{_{X^*, Y^*, A, B}}$ (see Corollary \ref{homiso}) and the definition of $\phi_X(x)$ one has
\begin{align*}& \chi_{_{X^*, Y^*, A, B}}(\phi_X(x)\otimes \phi_Y(y))(\chi^{-1}_{_{X, Y, A, B}}(\alpha))
\\ & = \sum\limits_i\chi_{_{X^*, Y^*, A, B}}(\phi_X(x)\otimes \phi_Y(y))(f_i\otimes g_i)
\\ & = \sum\limits_i \phi_X(x)(f_i)\otimes \phi_Y(y)(g_i)
\\ & = \sum\limits_i f_i(x)\otimes g_i(y).\end{align*}
By the Cartan-Eilenberg isomorphism one has $\chi_{_{X, Y, A, B}}\circ \chi^{-1}_{_{X, Y, A, B}} = {\rm Id}_{(X\otimes Y)^*}$. It follows that
$$\alpha = \chi_{_{X, Y, A, B}}\circ \chi^{-1}_{_{X, Y, A, B}}(\alpha)= \sum\limits_i\chi_{_{X, Y, A, B}}(f_i\otimes g_i).$$
Thus, by the definition of $\chi_{_{X, Y, A, B}}$ one has
$$\alpha(x\otimes y) = \sum\limits_if_i(x)\otimes g_i(y).$$
This finishes the proof.
\end{proof}

\begin{prop}\label{tensorker} \ Let $A$ and $B$ be finite-dimensional $k$-algebras,
$f:X\longrightarrow X'$ an $A$-homomorphism, and $g: Y\longrightarrow Y'$ a $B$-homomorphism.
Then $${\rm Ker} (f\otimes g)={\rm Ker} f\otimes Y+X\otimes {\rm Ker} g.$$

In particular, if both $X$ and $Y$ are non-zero, then $f\otimes g$ is a monomorphism if and only if both $f$ and $g$ are monomorphisms.
\end{prop}

\begin{proof} \ Denote by $\sigma_1: {\rm Ker} f \longrightarrow X$ the canonical embedding. Consider
the exact sequence $0\longrightarrow {\rm Ker} g \stackrel {\sigma_2} \longrightarrow Y \stackrel {\tilde{g}} \longrightarrow \Ima g \longrightarrow  0.$  By the exactness of $-\otimes_k-$, there is a commutative diagram with exact rows and columns:
$$
\xymatrix{&0\ar[d]&& 0\ar[d]&& 0\ar[d]&\\
0\ar[r]&{\rm Ker} f\otimes {\rm Ker} g \ar[rr]^{1_{{\rm Ker} f}\otimes \sigma_2}\ar[d]_{\sigma_1\otimes 1_{{\rm Ker}g}}&&{\rm Ker} f \otimes Y\ar[rr]^{1_{{\rm Ker} f}\otimes \tilde{g}}\ar[d]^{\sigma_1\otimes 1_{Y}}&&
{\rm Ker} f\otimes \Ima g\ar[d]^{\sigma_1\otimes 1_{\Ima \tilde{g}}}\ar[r] & 0 \\
0\ar[r]&X\otimes {\rm Ker}g \ar[rr]^{1_X\otimes {\sigma_2}}\ar[d]_{f\otimes 1_{{\rm Ker}g}}&& X \otimes Y\ar[rr]^{1_X\otimes \tilde{g}}\ar[d]^{f\otimes 1_Y}& & X\otimes \Ima g \ar[d]^{f\otimes 1_{\Ima \tilde{g}}} \ar[r] & 0 \\
0\ar[r]& X'\otimes {\rm Ker}g \ar[rr]^{1_{X'}\otimes {\sigma_2}}&& X' \otimes Y \ar[rr]^{1_{X'}\otimes \tilde{g}}&& X'\otimes \Ima g \ar[r] & 0}
$$
By diagram chasing one sees that ${\rm Ker} (f\otimes g)={\rm Ker} f\otimes Y+X\otimes {\rm Ker} g.$ In fact, for any
$\sum\limits_i x_i\otimes y_i\in {\rm Ker} (f\otimes g)$, then there is an element $\sum\limits_j k_j\otimes g(\tilde{y}_j)\in {\rm Ker} f\otimes \Ima g$ such that
\begin{align*}(1_X\otimes \tilde{g} ) (\sum\limits_i x_i\otimes y_i)&  = (\sigma_1\otimes 1_{\Ima \tilde{g}}) (\sum\limits_j k_j\otimes g(\tilde{y}_j)) = (\sigma_1\otimes 1_{\Ima \tilde{g}}) \circ (1_{{\rm Ker} f}\otimes \tilde{g}) (\sum\limits_j k_j\otimes \tilde{y}_j) \\ & =  (1_X\otimes \tilde{g})\circ (\sigma_1\otimes 1_Y) (\sum\limits_j k_j\otimes \tilde{y}_j) = (1_X\otimes \tilde{g}) (\sum\limits_j k_j\otimes \tilde{y}_j). \end{align*}
Thus $\sum\limits_i x_i\otimes y_i - \sum\limits_j k_j\otimes \tilde{y}_j $ is in $\Ker (1_X\otimes \tilde{g}) = \Ima (1_X\otimes \sigma_2) = X\otimes \Ker g,$ and hence
$\sum\limits_i x_i\otimes y_i\in {\rm Ker} f\otimes Y+X\otimes {\rm Ker} g $. \end{proof}

%
%

\begin{thm}\label{thm:main_tensor}
Let $A$ and $B$ be finite-dimensional $k$-algebras, $0\ne X\in A$-{\rm mod} and $0\ne Y\in B$-{\rm mod}. Then
\vskip5pt

$(1)$ \ $X\otimes Y$ is a torsionless $(A\otimes B)$-module if and only if  $X$ is a torsionless $A$-module and $Y$ is a torsionless $B$-module.

\vskip5pt

$(2)$ \ $X\otimes Y$ is a reflexive $(A\otimes B)$-module if and only if  $X$ is a reflexive $A$-module and $Y$ is a reflexive $B$-module.

\end{thm}
\begin{proof}
According to Proposition \ref{lem:phi_tensor},
$\phi_{X\otimes Y}=\Hom(\chi^{-1}_{_{X, Y, A, B}}, A\otimes B)\circ \chi_{_{X^*, Y^*, A, B}}\circ (\phi_X\otimes \phi_Y)$,
where $\Hom(\chi^{-1}_{_{X, Y, A, B}}, A\otimes B)$ and $\chi_{_{X^*, Y^*, A, B}}$ are isomorphisms. Thus, $\phi_{X\otimes Y}$ is injective (bijective, respectively) if and only if $(\phi_X\otimes \phi_Y)$ is injective
(bijective, respectively), which is, by Proposition \ref{tensorker}, equivalent to that both  $\phi_X$ and $\phi_Y$ are injective (bijective, respectively).
\end{proof}

\begin{rem}
$(1)$ \  The assumption $X$ and $Y$ are non-zero is used only for the ``only if'' part of {\rm Theorem \ref{thm:main_tensor}}.

\vskip5pt

$(2)$ \ Not all the torsionless or reflexive modules over $A\otimes B$ has the form $X\otimes Y$. For example, consider the algebras $A=k(1\longrightarrow 2)$, $B=k(1\stackrel{x}\longrightarrow 2\stackrel{y}\longrightarrow 3)/\langle yx\rangle$. Then $A\otimes B$ is isomorphic to the following bound quiver algebra with relations
$\beta\alpha, \ \beta'\alpha', \ \delta\alpha-\alpha'\gamma, \ \epsilon\beta-\beta'\delta$.
$$
\xymatrix{ 1\ar[r]^\alpha\ar[d]_-\gamma &2\ar[d]^\delta\ar[r]^\beta&3\ar[d]^\epsilon\\
4\ar[r]^{\alpha'} &5\ar[r]^{\beta'}&6}
$$
Since there is an approximation sequence
$$
0\longrightarrow \rad P_2\longrightarrow P_2\oplus P_4 \longrightarrow \rad P_1\longrightarrow 0,
$$
where $\rad P_1$ is a torsionless $(A\otimes B)$-module, it follows from {\rm Lemma  \ref {lem:refl_tors}} that $\rad P_2$ is a reflexive $(A\otimes B)$-module$.$ Since
both $\rad P_1$ and $\rad P_2$ are indecomposable $(A\otimes B)$-modules of dimension $3$, they can not be of the form $X\otimes Y$.
\end{rem}

\subsection{Semi-Gorenstein-projective Modules over tensor algebras}

The strong Nakayama conjecture asserts that for a finite-dimensional algebra $A$, a finitely generated left $A$-module $M=0$ if $\Ext^i_A(M,A)=0$ for all $i\geq 0$.
It is known that the strong Nakayama conjecture holds for some classes of algebras. However, it is still open in general.

\vskip5pt

\begin{prop} \label{tensorsgp} \  Let $A$ and $B$ be finite-dimensional $k$-algebras,  $X\in A$-{\rm mod} and $Y\in B$-{\rm mod}.
If both $X$ and $Y$ are semi-Gorenstein-projective, then so is $X\otimes Y$.

\vskip5pt

Conversely, assume that $\findim A^{\op}<\infty$ and $\findim B^{\op}<\infty$.
If $X\otimes Y$ is semi-Gorenstein-projective and $X \ne 0 \ne Y$, then both $X$ and $Y$ are semi-Gorenstein-projective.
\end{prop}
\begin{proof} The first assertion directly follows from the Cartan-Elienberg isomorphism.

\vskip5pt

Conversely, assume that $\findim A^{\op}<\infty$,  $\findim B^{\op}<\infty$,
$X \ne 0 \ne Y$,  and that $X\otimes Y$ is semi-Gorenstein-projective. By the Cartan-Eilenberg isomorphism
   $$\Ext^m_{
A\otimes B}(X \otimes Y, A\otimes B)\cong
\bigoplus_{p+q=m}
(\Ext^p_
A(X, A) \otimes \Ext^q_
B(Y, B))=0,  \ \forall \ m\geq 1$$
one has $\Ext^p_A(X, A) \otimes \Ext^q_B(Y, B)=0$ for any integers $p+q\geq 1$. If $\Ext^p_A(X, A)\neq 0$ for some $p>0$, then $\Ext^q_B(Y, B)=0$ for all $q\geq 0$.
Under the assumption $\findim B^{\op}<\infty$, it is known that the strong Nakayama conjecture holds.
Therefore $Y=0$, which is a contradiction. Thus $X$ is semi-Gorenstein-projective. Similarly, $Y$ is semi-Gorenstein-projective.
\end{proof}

\section{\bf Weakly Gorensteinness of tensor algebras}

Throughout this section, $A$ and $B$ are finite-dimensional algebras over a field $k$, and $\Lambda=A\otimes B$, where $\otimes = \otimes_k$.
The aim of this section is to show that if ${\rm gl.dim}B <\infty$  then $\Lambda$ is left weakly Gorenstein if and only if so is $A$.
The method used will be the monomorphism category $\Mon(B, \ \leftidx^\perp A).$

\vskip5pt

\begin{lemdef} {\rm (\cite[Lemma 3.2]{HLXZ})} \label{defmonic} \ $(1)$ \ A left $\Lambda$-module $M$ is called {\it a monic representation of $B$ over $A$}, provided that one of the following equivalent conditions holds$:$

\vskip5pt

${\rm (i)}$ \ $\Tor_i^\Lambda (A\otimes Y, M)=0$ for all $i\geq 1$ and each right module $Y\in {\rm mod} B$.

\vskip5pt

${\rm (ii)}$ \ $\Ext^i_\Lambda(M, D(A)\otimes X)=0$ for $i\geq 1$ and for each left module $X\in B\mbox{-}{\rm mod}$, where $D = \Hom_k(-, k)$.

\vskip5pt

${\rm (iii)}$ \ As a left $B$-module, $M$ is a projective module.

\vskip5pt

Denote by $\Mon(B,A)$ the full subcategory of $\Lambda$-{\rm mod} consisting of monic representation of $B$ over $A$, which is called the monomorphism category of $B$ over $A$.

\vskip5pt

${\rm (2)}$ \ For a full additive subcategory $\calL$ of $A$-$\modd$, put
$$
\Mon(B,\calL)=\{M\in\Mon(B,A)\mid (A\otimes Y)\otimes_\Lambda M\in\calL, \ \forall \ Y\in\dmod B\}.
$$
A module in $\Mon(B,\calL)$ is called a monic representation of $B$ over $\calL$, and $\Mon(B,\calL)$ is called the monomorphism category of $B$ over $\calL$.
\end{lemdef}

Note that $\Mon(B, A\mbox{-}{\rm mod}) = \Mon(B, A).$ When $B = kQ/I$, where $Q$ is an acyclic quiver and $I$ is an admissible ideal generated by monomial relations,
then the monomorphism categories $\Mon(B,A)$ and $\Mon(B,\calL)$  coincide with the ones defined in \cite{LZ1}, \cite{LZ2}, \cite{ZX}, via the combinatorial datum of $Q$,
where the terminology ``monomorphism" has an explicit combinatorial interpretation.

\vskip5pt

It turns out that if $\gldim B<\infty$ then  Gorenstein-projective modules over $A\otimes B$ could be characterized by the monomorphism category:

\begin{lem} {\rm (\cite[Theorem 4.5]{HLXZ})} \label{thm:GP_mon} \
Let $A$ and $B$ be finite-dimensional $k$-algebras and $\Lambda=A\otimes B$. If $B$ is a Gorenstein algebra, then $\Gp(\Lambda)=\Mon(B,\Gp(A))$ if and only if $\gldim B<\infty$.
\end{lem}

\begin{lem}\label{lem:syzygy_GP}
Let $X$ be a semi-Gorenstein-projective $A$-module. If there is some non-negative integer $d$ such that $\Omega^d(X)$ is Gorenstein-projective, then so is $X$.
\end{lem}
\begin{proof} \ By induction it suffices to prove that if $\Omega(X)$ is Gorenstein-projective then so is $X$. Assume that $\Omega(X)$ is Gorenstein-projective
and $f:\Omega(X)\longrightarrow P$ is a minimal left $\add A$-approximation of $\Omega(X)$. By Lemma \ref{lem:GP_app}, $\cok f$ is Gorenstein-projective.
Let $0\longrightarrow \Omega(X)\stackrel{\sigma}\longrightarrow P(X)\longrightarrow X\longrightarrow 0$ be an exact sequence with $P(X)$ the projective cover of $X$.
Since by assumption $X$ is semi-Gorenstein-projective, it follows that the map $\sigma$
is also a left $(\add A)$-approximation of $\Omega(X)$. Thus $X\cong \cok f\oplus Q$ for some projective summand $Q$ of $P(X)$. Hence $X$ is Gorenstein-projective.
\end{proof}

\begin{lem}\label{lem:syzygy_in_mon} \ Let $A$ and $B$ be finite-dimensional $k$-algebras with $\gldim B=d<\infty$ and $\Lambda=A\otimes B$. Then $\Omega^t(M)\in\Mon(B,A)$ for every $\Lambda$-module $M$ and for $t\geq d$.
\end{lem}
\begin{proof}
Since  $\Mon(B,A)$ is resolving (\cite[Corollary 3.4]{HLXZ}), it suffices to show  $\Omega^{d}(M)\in\Mon(B,A)$,
By definition, one needs to show  $\Ext^i_\Lambda(\Omega^{d}(M), D(A)\otimes Y)=0$ for $i\geq 1$ and $Y\in B\mbox{-}{\rm mod}$. Since $\gldim B=d$,  it follows that
${\rm inj.dim} (D(A)\otimes Y)\leq d$ as a left $\Lambda$-module. Hence, by dimension shift $\Ext^i_\Lambda(\Omega^{d}(M), D(A)\otimes Y)=\Ext^{i+d}_\Lambda(M, D(A)\otimes Y)=0$, for all $i\geq 1$.
\end{proof}

\vskip5pt

The following result will play a key role in the proof of the main result Theorem \ref{thm:weak_G}.

\begin{lem}{\rm (\cite[Theorem 1.1]{Z})} \label{lem:SGP_cap_mon}
Let $A$ and $B$ be finite-dimensional $k$-algebras with $\gldim B<\infty$, and $\Lambda=A\otimes B$. Then
$\leftidx^\perp\Lambda\cap\Mon(B,A)=\Mon(B,\leftidx^\perp A)$.
\end{lem}
\begin{proof} \ This result has been proved in \cite[Theorem 1.1]{Z}. For the convenience of the reader we include here another proof, by using the following isomorphism
$$\Ext^i_\Lambda((A\otimes Y)\otimes_\Lambda M, U) \cong \Ext^i_\Lambda(M, U\otimes D(Y)), \ \forall \ i\ge 0 \eqno(*)$$
where $M\in \Mon(B,  A)$, $Y$ is an arbitrary right $B$-module, and $U$ is an arbitrary left $A$-module. For the proof of this isomorphism we refer to \cite[Lemma 3.5]{HLXZ}.

\vskip5pt

Now, let $M\in\Mon(B,\leftidx^\perp A)$. By definition, $M\in\Mon(B,A)$ and $(A\otimes Y)\otimes_\Lambda M\in\leftidx^\perp A$ for any $Y\in\dmod B$.
In particular,  $(A\otimes D(B))\otimes_\Lambda M\in\leftidx^\perp A$.
Hence by the isomorphism $(*)$ one has
$$\Ext^i_\Lambda(M, \Lambda) = \Ext^i_\Lambda(M, A\otimes B)\cong \Ext^i_\Lambda((A\otimes D(B))\otimes_\Lambda M,A)=0, \ \ \forall \ i\ge 1.$$ So $\Mon(B,\leftidx^\perp A)\subseteq \leftidx^\perp \Lambda\cap\Mon(B,A)$.

\vskip5pt

Conversely, let $M\in \leftidx^\perp \Lambda\cap\Mon(B, A)$. It suffices to show that for any  $Y\in\dmod B$, $(A\otimes Y)\otimes_\Lambda M\in\leftidx^\perp A$.
In fact, since $\gldim B<\infty$, it follows that $\pd_\Lambda (A\otimes D(Y))<\infty$ as a left $\Lambda$-module. Therefore by the isomorphism $(*)$  one has $\Ext^i_A((A\otimes Y)\otimes_\Lambda M,A)\cong \Ext^i_\Lambda(M, A\otimes D(Y))=0$ for  $i\ge 1$.
\end{proof}

\begin{thm}\label{thm:weak_G} \ Let $A$ and $B$ be finite-dimensional $k$-algebras with $\gldim B <\infty$, and $\Lambda=A\otimes B$. Then

\vskip5pt

$(1)$ \ $\Lambda$ is left weakly Gorenstein if and only if $A$ is left weakly Gorenstein.

\vskip5pt

$(2)$ \ If $A$ is left weakly Gorenstein,  then $\leftidx^\perp \Lambda = \Mon(B,\leftidx^\perp A).$

\vskip5pt

$(3)$ \ If $B$ is not semisimple, then the converse of $(2)$ is also true, and hence  $\Lambda$ is left weakly Gorenstein.
\end{thm}
\begin{proof} $(1)$ \ Assume $A$ is left weakly Gorenstein. For any semi-Gorenstein-projective left $\Lambda$-module $M$, we need to show that $M$ is Gorenstein-projective. Put $d = \gldim B<\infty$.
By Lemma \ref{lem:syzygy_in_mon}, $\Omega^d(M)\in\Mon(B,A)$. On the other hand, $\Omega^d(M)$ is semi-Gorenstein-projective as well. Hence $\Omega^d(M)\in \leftidx^\perp\Lambda\cap\Mon(B,A)=\Mon(B,\leftidx^\perp A)$ by Lemma \ref{lem:SGP_cap_mon}. Since $A$ is left weakly Gorenstein, $\Mon(B,\leftidx^\perp A)=\Mon(B,\Gp(A))$. By Lemma \ref{thm:GP_mon}, $\Mon(B,\Gp(A))= \Gp(\Lambda)$, which shows that  $\Omega^d(M)$ is Gorenstein-projective. Therefore, by Lemma \ref{lem:syzygy_GP},  $M$ is Gorenstein-projective as well.

\vskip5pt

Conversely, assume $\Lambda$ is left weakly Gorenstein. Suppose otherwise that $A$ is not left weakly Gorenstein. Then by \cite[Theorem 1.2]{RZ2} there is a semi-Gorenstein-projective $A$-module $X$ which is not torsionless.
By  Theorem \ref{thm:main_tensor} and Proposition \ref{tensorsgp},  $X\otimes B$ is a semi-Gorenstein-projective $\Lambda$-module which is not torsionless.
This contradicts the assumption that $\Lambda$ is left weakly Gorenstein.

\vskip5pt

$(2)$ \ If $A$ is a left weakly Gorenstein algebra, then by the assertion (1) and Lemma \ref{thm:GP_mon} one has
$\leftidx^\perp\Lambda = \Gp(\Lambda) = \Mon(B, \Gp(A)) = \Mon(B, \leftidx^\perp A).$

\vskip5pt

$(3)$ \ Assume that $B$ is not semisimple and  $\leftidx^\perp\Lambda = \Mon(B,\leftidx^\perp A)$.
It remains to prove that  $A$ is left weakly Gorenstein.
Suppose otherwise that $A$ is not left weakly Gorenstein. By Lemma \ref{lem:SGP_cap_mon} one has $\Mon(B,\leftidx^\perp A)=\leftidx^\perp\Lambda\cap\Mon(B,A)\subseteq\leftidx^\perp\Lambda$.
Thus, in order to get a contradiction,
it suffices to show that there is  a semi-Gorenstein-projective $\Lambda$-module $M$ which is not a monic representation of $B$ over $A$,
i.e., $M\in \ ^\perp \Lambda$ but $M\notin \Mon(B,  A)$. Indeed, such a $\Lambda$-module $M$ can be constructed in the following  way.

\vskip5pt

Since $A$ is not left weakly Gorenstein, again by \cite[Theorem 1.2]{RZ2}, there is a semi-Gorenstein-projective $A$-module $X$ which is not torsionless.
Take a minimal left $(\add A)$-approximation $f:X\longrightarrow P$ of $X$. Then $f$ is not a monomorphism (cf. Lemma \ref{torless}).

\vskip5pt

\vskip5pt

Since $0<\gldim B<\infty$, one can take a $B$-module $Y$ with $\pd Y=1$, say, with a projective resolution $0\longrightarrow Q_1\stackrel{d_1}\longrightarrow  Q_0\stackrel{d_0}\longrightarrow  Y\longrightarrow  0.$
Consider a pushout diagram

$$\xymatrix{
0\ar[r]& X\otimes Q_1 \ar[r]^{{\rm Id}_X\otimes d_1}\ar[d]_{f\otimes {\rm Id}_{Q_1}}& X\otimes Q_0\ar[r]^{{\rm Id}_X\otimes d_0}\ar[d]& X\otimes Y\ar[r]\ar@{=}[d]&0\\
0\ar[r]& P\otimes Q_1 \ar[r]& M\ar[r]& X\otimes Y\ar[r]&0.
}
$$
By Proposition \ref{lem:app_tensor},  $f\otimes {\rm Id}_{Q_1}$ is a left $(\add\Lambda)$-approximation of $X\otimes Q_1$. Since $f$ is not a monomorphism,
$f\otimes {\rm Id}_{Q_1}$ is not a monomorphism.

\vskip5pt

{\bf Claim 1:} $M\in\leftidx^\perp \Lambda$.

\vskip5pt

In fact, applying $\Hom_\Lambda(-,\Lambda)$ to the second row of the above diagram one gets an exact sequence
$$
0=\Ext^i_\Lambda(P\otimes Q_1, \Lambda)\longrightarrow\Ext^{i+1}_\Lambda(X\otimes Y, \Lambda)\longrightarrow \Ext^{i+1}_\Lambda(M,\Lambda)\longrightarrow \Ext^{i+1}_\Lambda(P\otimes Q_1, \Lambda)=0
$$
for all $i\geq 1$. Therefore by the Cartan-Eilenberg isomorphism one has the isomorphism for all $i\geq 1$, $$\Ext^{i+1}_\Lambda(M,\Lambda)\cong \Ext^{i+1}_\Lambda(X\otimes Y, \Lambda)\cong \Hom_A(X,A)\otimes \Ext^{i+1}_B(Y,B)=0.$$
It remains to show that $\Ext^1_\Lambda(M,\Lambda)=0$. Applying $\Hom_\Lambda(-,\Lambda)$ to the above pushout diagram yields a commutative diagram with exact rows:
$$\xymatrix{ \cdots\ar[r]&\Hom_\Lambda(P\otimes Q_1,\Lambda)\ar[r]^{\delta}\ar@{->>}[d]_{\Hom(f\otimes {\rm Id}_{Q_1},\Lambda)}& \Ext^1_\Lambda(X\otimes Y,\Lambda)\ar[r]\ar@{=}[d]&\Ext^1_\Lambda(M,\Lambda)\ar[r]\ar[d]&0 \\
\cdots\ar[r]&\Hom_\Lambda(X\otimes Q_1,\Lambda)\ar[r]\ar[r]^\partial &\Ext_\Lambda^1(X\otimes Y,\Lambda)\ar[r]&\Ext^1_\Lambda(X\otimes Q_0,\Lambda)\ar[r]&\cdots}$$
Since $X$ is a semi-Gorenstein-projective $A$-module, according to the Cartan-Eilenberg isomorphism one has  $\Ext^1_\Lambda(X\otimes Q_0,\Lambda)=0$. Hence the connecting morphism $\partial$ is an epimorphism.
Since $f\otimes {\rm Id}_{Q_1}$ is a left $\add \Lambda$-approximation, $\Hom(f\otimes {\rm Id}_{Q_1},\Lambda)$ is an epimorphism. Thus, the connecting morphism $\delta$ is epic as well. Therefore $\Ext^1_\Lambda(M,\Lambda)=0$.

\vskip5pt

{\bf Claim 2:} $M\not\in \Mon(B,A)$.
\vskip5pt

In fact, from the pushout diagram we obtain an exact sequence of $\Lambda$-modules:
$$
0\longrightarrow X\otimes Q_1 \stackrel{\left(\begin{smallmatrix}{\rm Id}_{X}\otimes d_1 \\f\otimes {\rm Id}_{Q_1}
\end{smallmatrix}\right)}\longrightarrow (X\otimes Q_0)\oplus (P\otimes Q_1)  \longrightarrow M \longrightarrow 0.
$$
Applying $\Hom_\Lambda(-, D(A)\otimes Q_1)$  and the Cartan-Eilenberg isomorphism, one gets a commutative diagram with exact rows (where ${\rm E} = {\rm Ext}$)
$$
\xymatrix@C=0.6cm{((X\otimes Q_0)\oplus (P\otimes Q_1), D(A)\otimes Q_1) \ar[r]\ar[d]_\wr& (X\otimes Q_1, D(A)\otimes Q_1)\ar[r]\ar[d]_\wr& {\rm E}^1(M,D(A)\otimes Q_1)\to 0\\
{\begin{smallmatrix}\Hom_A(X,D(A))\otimes \Hom_B(Q_0,Q_1)\\ \oplus\\ \Hom_A(P, D(A))\otimes \Hom_B(Q_1, Q_1)\end{smallmatrix}}\ar[r]^-\theta & \Hom_A(X,D(A))\otimes \Hom_B(Q_1,Q_1)}
$$
where the map $\theta$ is given by  $$\left(\Hom_A({\rm Id}_X,D(A))\otimes\Hom_B(d_1, Q_1), \ \Hom_A(f,D(A))\otimes{\rm Id}_{\Hom_B(Q_1,Q_1)}\right)$$ with  image
$$
\Ima \theta= \Hom_A(X,D(A))\otimes \Ima\Hom_B(d_1,Q_1)+\Ima\Hom_A(f,D(A))\otimes\Hom_B(Q_1, Q_1)
$$
Since $\pd Y=1$, it follows that $\Hom_B(d_1, Q_1): \Hom_B(Q_0, Q_1) \longrightarrow \Hom_B(Q_1,Q_1)$ is not an epimorphism.
Since $f:X\longrightarrow P$ is not a monomorphism, $D(f): D(P) \longrightarrow D(X)$ is not an epimorphism, i.e., $\Hom_A(f, D(A)): \Hom_A(P, D(A)) \longrightarrow \Hom_A(X, D(A))$ is not an epimorphism.
Thus,  $\Ima\Hom_B(d_1, Q_1)$ is a proper $k$-subspace of $\Hom_B(Q_1,Q_1)$, and $\Ima\Hom_A(f,D(A))$ is a proper $k$-subspace of  $\Hom_A(X,D(A))$. Therefore, $\Ima\theta$ is a proper subspace of
$\Hom_A(X,D(A))\otimes\Hom_B(Q_1,Q_1)$. Thus, $\theta$ is not an epimorphism, which implies $\Ext^1_\Lambda(M,D(A)\otimes Q_1)\neq 0$.
By definition, $M\not\in\Mon(B,A)$. This completes the proof.
\end{proof}

\begin{rem} \label{remweakG}  ${\rm (1)}$ \ From the proof one can see that the ``if" part of {\rm Theorem \ref{thm:weak_G}(1)} does not need the condition $\gldim B <\infty$, i.e.,
$A\otimes B$ is left weakly Gorenstein always implies that $A$ and $B$ are left weakly Gorenstein. In fact, we do not know whether the converse is true in general.

\vskip5pt

${\rm (2)}$ \ If $\gldim B=0$, i.e., $B$ is semisimple, then the conclusion of {\rm Theorem \ref{thm:weak_G}(3)} is no longer true. In fact, by Definition {\rm \ref{defmonic}} one sees that if $B$ is semi-simple then $\Lambda\mbox{-}{\rm mod} =\Mon(B,A)$, and that the equation
$$\leftidx^\perp\Lambda = \Mon(B,\leftidx^\perp A) = \{M\in \Lambda\mbox{-}{\rm mod} \ | \ _AM\in \ ^\perp A\}$$ holds for any finite-dimensional $k$-algebra $A$
$($see also {\rm Lemma  \ref{lem:SGP_cap_mon}}$)$.
\end{rem}

\section{\bf Semi-Gorenstein-projective modules over Morita rings}

\subsection{Modules over Morita rings} \ Let $A$ and $B$ be rings, $_{A}N_{B}$ an $A\mbox{-}B$-bimodule, $_{B}M_{A}$ a $B\mbox{-}A$-bimodule, $\phi:M\otimes_{A}N\lxr B$ a $B\mbox{-}B$-bimodule homomorphism, and $\psi:N\otimes_{B}M\lxr A$ an $A\mbox{-}A$-bimodule homomorphism. The Morita ring
$$\Lambda = \Lambda_{(\phi, \psi)}=\begin{pmatrix}\begin{smallmatrix}
A & _{A}N_{B} \\
_{B}M_{A} & B \\
\end{smallmatrix}\end{pmatrix}$$
is defined (\cite{B}) with componentwise addition and multiplication by
$$\begin{pmatrix}\begin{smallmatrix}
a & n \\
m & b \\
\end{smallmatrix}\end{pmatrix}\begin{pmatrix}\begin{smallmatrix}
a' & n' \\
m' & b' \\
\end{smallmatrix}\end{pmatrix}=\begin{pmatrix}\begin{smallmatrix}
aa'+\psi(n\otimes{m'}) & an'+nb' \\
ma'+bm' & bb'+\phi(m\otimes{n'}) \\
\end{smallmatrix}\end{pmatrix}$$
for all $m,m'\in M$ and $n,n'\in N$.
One always assumes that
$$\phi(m\otimes n)m'=m\psi(n\otimes{m'}) \ \ \mbox{and}  \ \ n\phi(m\otimes {n'})=\psi(n\otimes m)n', \ \ \forall \ m, m'\in M, \ \ \ \forall \ n, n'\in N$$
which is precisely  the sufficient and necessary condition such that $\Lambda_{(\phi, \psi)}$ is an associative ring.

\vskip 5pt

We will assume $M\otimes_{A}N=0=N\otimes_{B}M$.
Also, we assume that $\Lambda$ is an Artin algebra, i.e., $A$ and $B$ are Artin $R$-algebras with  a commutative artinian ring $R$, $M$ and $N$ are finitely generated $R$-modules,
such that $R$ acts centrally both on $M$ and $N$ (\cite[Proposition 2.2]{GrP}).
Recall from \cite[Theorem 1.5]{Gr} that a left $\Lambda$-module can be identified with a quadruple $(X, Y, f, g)$,
where $X\in A\modu,\ Y\in B\modu,\ f\in {\rm Hom}_{B}(M\otimes_{A}X, Y)$ and $g\in {\rm Hom}_{A}(N\otimes_{B}Y, X)$. We will write a left $\Lambda$-module as
$\begin{pmatrix}\begin{smallmatrix}X \\ Y \end{smallmatrix}\end{pmatrix}_{f, g}.$
In this way the $\Lambda$-action will coincide with the multiplication of matrices.

\vskip5pt

A homomorphism of left $\Lambda$-module is a pair
$$\begin{pmatrix}\begin{smallmatrix}a \\ b\end{smallmatrix}\end{pmatrix}: \ \begin{pmatrix}\begin{smallmatrix}X \\ Y \end{smallmatrix}\end{pmatrix}_{f, g}\lxr \begin{pmatrix}\begin{smallmatrix}X' \\ Y' \end{smallmatrix}\end{pmatrix}_{f', g'}$$
where $a: X\lxr X'$ is an $A$-homomorphism and $b: Y\lxr Y'$ is a $B$-homomorphism, such that the following diagrams commute:
\[
\xymatrix@C=0.5cm{
M\otimes_{A}X\ar@{->}[rr]^-{1\otimes a }\ar@{->}[d]_-{f} && M\otimes_{A}X^{\prime}\ar@{->}[d]^{f'}\\
Y\ar@{->}[rr]^{b}     &&Y^{\prime}\\
} \ \ \ \ \ \ \ \ \ \ \ \ \ \ \ \ \ \xymatrix@C=0.5cm{
N\otimes_{B}Y\ar@{->}[d]_-{g}\ar@{->}[rr]^-{1\otimes b} && N\otimes_{B}Y^{\prime}\ar@{->}[d]^-{g^{\prime}}\\
X\ar@{->}[rr]^{a}     && X^{\prime}.\\
}
\]

\vskip 5pt

Since by assumption $M\otimes_A N = 0 = N\otimes _B M$, there is an exact sequence $M\otimes_{A}\cok g\stackrel{\tilde{f}}{\lxr} Y\stackrel{\pi_{f}}{\lxr} \cok f\lxr 0$ of $B$-modules
and an exact sequence $N\otimes_{B}\cok f\stackrel{\tilde{g}}{\lxr} X\stackrel{\pi_{g}}{\lxr} \cok g\lxr 0$ of $A$-modules.

\vskip 5pt

\begin{lem}\label{projectives}  {\rm (\cite[Proposition 3.2]{GrP})} \ Let  $\Lambda=\begin{pmatrix}\begin{smallmatrix}
A &N \\
M & B \\
\end{smallmatrix}\end{pmatrix}$ be a Morita ring which is an Artin algebra, with $M\otimes_{A}N=0=N\otimes_{B}M$.
Then

\vskip 5pt

{\rm (1)} \  The indecomposable projective left $\Lambda$-modules are precisely
$$\begin{pmatrix}\begin{smallmatrix}P \\ M\otimes_{A}P\end{smallmatrix}\end{pmatrix}_{\id_{M\otimes_{A}P}, 0}\ \ \ \
\mbox{and}\ \ \ \ \begin{pmatrix}\begin{smallmatrix}N\otimes_{B}Q \\ Q\end{smallmatrix}\end{pmatrix}_{0, \id_{N\otimes_{B}Q}}$$
where $P$ and $Q$ range over the indecomposable projective left $A$-modules and $B$-modules, respectively.

\vskip 5pt

{\rm (2)} \ The indecomposable injective left $\Lambda$-modules are precisely
$$\begin{pmatrix}\begin{smallmatrix} I \\ {\rm Hom}_{A}(N, I)\end{smallmatrix}\end{pmatrix}_{0, \epsilon'_{I}}
\ \ \ \ \mbox{and}\ \ \ \
\begin{pmatrix}\begin{smallmatrix}  {\rm Hom}_{B}(M, J) \\ J\end{smallmatrix}\end{pmatrix}_{\epsilon_{J}, 0}$$
where $I$ and $J$ range over indecomposable injective left $A$-modules and $B$-modules, respectively, and $\epsilon'_{I}: N\otimes_{B}{\rm Hom}_{A}(N, I)\lxr I$ and $\epsilon_{J}: M\otimes_{A}{\rm Hom}_{B}(M, J)\lxr J$ are the evaluation maps.

\vskip 5pt

{\rm (3)} \ Let $$\begin{pmatrix}\begin{smallmatrix}X_1\\Y_1\end{smallmatrix}\end{pmatrix}_{f_1, g_1}
\xrightarrow{\begin{pmatrix}\begin{smallmatrix}a\\b\end{smallmatrix}\end{pmatrix}} \begin{pmatrix}\begin{smallmatrix}X_2\\Y_2\end{smallmatrix}\end{pmatrix}_{f_2, g_2}
\xrightarrow{\begin{pmatrix}\begin{smallmatrix}c\\d\end{smallmatrix}\end{pmatrix}}
\begin{pmatrix}\begin{smallmatrix}X_3\\Y_3\end{smallmatrix}\end{pmatrix}_{f_3, g_3}$$

\vskip5pt

\noindent be a sequence of $\Lambda$-homomorphisms. Then it is exact in $\Lambda\modu$ if and only if $X_{1}\stackrel{a}{\lxr} X_{2}\stackrel{c}{\lxr}X_{3}$ is exact in $A\mbox{-}{\rm mod}$ and $Y_{1}\stackrel{b}{\lxr} Y_{2}\stackrel{d}{\lxr} Y_{3}$
is exact in  $B\mbox{-}{\rm mod}$.
\end{lem}

We will start from left $\Lambda$-modules. However right  $\Lambda$-modules are also needed.
A right  $\Lambda$-module can be identified with a quadruple  $(X, Y, f, g)$,
where $X\in {\rm mod} A$ and $Y\in {\rm mod} B$,
$f\in {\rm Hom}_{A}(Y\otimes_{B}M, X)$ and $g\in {\rm Hom}_{B}(X\otimes_AN, Y)$.
We will write a right  $\Lambda$-module as
$(X, Y)_{f, g}$.
The indecomposable projective right $\Lambda$-modules are precisely
$$(P, P\otimes_{A}N)_{0, \id_{P\otimes_{A}N}} \ \ \ \
\mbox{and}\ \ \ \ (Q\otimes_{B}M, Q)_{\id_{Q\otimes_{B}M}, 0}$$
where $P$ and $Q$ range over the indecomposable projective right $A$-modules and $B$-modules, respectively.
The indecomposable injective right $\Lambda$-modules are precisely
$$(I, {\rm Hom}_{A}(M, I))_{\epsilon_{I}, 0}
\ \ \ \ \mbox{and}\ \ \ \
({\rm Hom}_{B}(N, J), J)_{0, \epsilon'_{J}}$$
where $I$ and $J$ range over indecomposable injective right $A$-modules and $B$-modules, respectively, and
$\epsilon_{I}: {\rm Hom}_{A}(M, I)\otimes_{B}M\lxr I$ and $\epsilon'_{J}:{\rm Hom}_{B}(N, J) \otimes_{A}N \lxr J$ are the evaluation maps.
\subsection{Some isomorphisms} \

\vskip 10pt

\begin{lem} \label{Extiso} \ {\rm ( [GP, Lemma 4.8])} \ Let $\Lambda=\begin{pmatrix}\begin{smallmatrix}
A & N \\
M & B \\
\end{smallmatrix}\end{pmatrix}$ be a Morita ring, and $\begin{pmatrix}\begin{smallmatrix}X'\\Y'\end{smallmatrix}\end{pmatrix}_{f', g'}$ a left $\Lambda$-module.

\vskip 5pt

{\rm (1)} \ If $M_A$ is projective, then for any left $A$-module $X$ there holds

$${\rm Ext}_{\Lambda}^i(\begin{pmatrix}\begin{smallmatrix} X \\ M\otimes_{A}X \end{smallmatrix}\end{pmatrix}_{\id_{M\otimes_AX}, 0}, \ \begin{pmatrix}\begin{smallmatrix}X'\\Y'\end{smallmatrix}\end{pmatrix}_{f', g'})\iso {\rm Ext}_{A}^{i}(X, X'), \ \forall \ i\ge 0.$$

\vskip 5pt

{\rm (2)} \ If $N_B$ is projective, then for any left $B$-module $Y$ there holds

$${\rm Ext}_{\Lambda}^{i}(\begin{pmatrix}\begin{smallmatrix} N\otimes_BY \\ Y \end{smallmatrix}\end{pmatrix}_{0, \id_{N\otimes_BY}}, \ \begin{pmatrix}\begin{smallmatrix}X'\\Y'\end{smallmatrix}\end{pmatrix}_{f', g'})\iso {\rm Ext}_{B}^{i}(Y, Y'), \ \forall \ i\ge 0.$$
\end{lem}

\vskip 10pt

\begin{lem}\label{baersum}
Let $\Lambda=\begin{pmatrix}\begin{smallmatrix}
A & N \\
M & B \\
\end{smallmatrix}\end{pmatrix}$ be a Morita ring,  $_{A}X$ an $A$-module and $_{B}Y$ a $B$-module. Then

\vskip 5pt

{\rm (1)} \  ${\rm Hom}_{A}(N\otimes_{B}Y, X) \iso {\rm Ext}_{\Lambda}^{1}(\begin{pmatrix}\begin{smallmatrix}0\\ Y\end{smallmatrix}\end{pmatrix}_{0, 0},\ \begin{pmatrix}\begin{smallmatrix}X\\0\end{smallmatrix}\end{pmatrix}_{0, 0})$, and this isomorphism is functorial
both in $X$ and $Y$. Explicitly,  this isomorphism $\Phi$ sends $g$ to
the equivalence class of the exact sequence

\[
0\lxr \begin{pmatrix}\begin{smallmatrix}X\\0\end{smallmatrix}\end{pmatrix}_{0, 0}\stackrel {\binom{1}{0}}
\lxr \begin{pmatrix}\begin{smallmatrix}X\\Y\end{smallmatrix}\end{pmatrix}_{0, g}\stackrel {\binom{0}{1}}\lxr \begin{pmatrix}\begin{smallmatrix}0\\Y\end{smallmatrix}\end{pmatrix}_{0, 0}\lxr 0.
\]

\vskip 5pt

{\rm (2)} \  ${\rm Hom}_{B}(M\otimes_{A}X, Y)\iso {\rm Ext}_{\Lambda}^{1}(\begin{pmatrix}\begin{smallmatrix}X\\ 0\end{smallmatrix}\end{pmatrix}_{0, 0},\ \begin{pmatrix}\begin{smallmatrix}0\\Y\end{smallmatrix}\end{pmatrix}_{0, 0})$,
and this isomorphism is functorial both in $X$ and $Y$.
Explicitly,  this isomorphism sends $f$ to
the equivalence class of the exact sequence

\[
0\lxr \begin{pmatrix}\begin{smallmatrix}0\\Y\end{smallmatrix}\end{pmatrix}_{0, 0}\stackrel {\binom{0}{1}}
\lxr \begin{pmatrix}\begin{smallmatrix}X\\Y\end{smallmatrix}\end{pmatrix}_{f, 0}\stackrel {\binom{1}{0}}\lxr \begin{pmatrix}\begin{smallmatrix}X\\ 0\end{smallmatrix}\end{pmatrix}_{0, 0}\lxr 0.
\]

\end{lem}
\begin{proof}\ (1) \ Any equivalence class in ${\rm Ext}_{\Lambda}^{1}(\begin{pmatrix}\begin{smallmatrix}0\\Y\end{smallmatrix}\end{pmatrix}_{0, 0},\ \begin{pmatrix}\begin{smallmatrix}X\\0\end{smallmatrix}\end{pmatrix}_{0, 0})$ has to be of the form

\[
0\lxr \begin{pmatrix}\begin{smallmatrix}X\\0\end{smallmatrix}\end{pmatrix}_{0, 0}\stackrel {\binom{1}{0}}
\lxr \begin{pmatrix}\begin{smallmatrix}X\\Y\end{smallmatrix}\end{pmatrix}_{f, g}\stackrel {\binom{0}{1}}\lxr \begin{pmatrix}\begin{smallmatrix}0\\Y\end{smallmatrix}\end{pmatrix}_{0, 0}\lxr 0.
\]
Since $\binom{1}{0}: \begin{pmatrix}\begin{smallmatrix}X\\0\end{smallmatrix}\end{pmatrix}_{0, 0} \longrightarrow \begin{pmatrix}\begin{smallmatrix}X\\Y\end{smallmatrix}\end{pmatrix}_{f, g}$ is a left $\Lambda$-homomorphism, it follows that $f = 0$, and hence $\Phi$ is a surjective map. Clearly, $\Phi$ is an injective map.
One directly verifies that $\Phi$ is functorial in $X$ and $Y$.

\vskip5pt

By the Baer sum one sees that $\Phi (g_1 + g_2) = \Phi (g_1) + \Phi (g_2)$.  In fact, one can see this via pushout

\[\xymatrix{0\ar[r] & {\begin{pmatrix}\begin{smallmatrix} X\oplus X \\ 0 \end{smallmatrix}\end{pmatrix}_{0, 0}}\ar[d]_{\binom{(1, 1)}{0}}
\ar[r]^-{\begin{pmatrix}\begin{smallmatrix} 1 \\ 0 \end{smallmatrix}\end{pmatrix}} & {\begin{pmatrix}\begin{smallmatrix} X\oplus X \\ Y\oplus Y \end{smallmatrix}\end{pmatrix}_{0, g_1\oplus g_2}}
\ar[d]_-{\binom{(1,1)}{1}}\ar[r]^-{\begin{pmatrix}\begin{smallmatrix} 0 \\ 1 \end{smallmatrix}\end{pmatrix}} &   {\begin{pmatrix}\begin{smallmatrix} 0 \\ Y\oplus Y \end{smallmatrix}\end{pmatrix}_{0, 0}} \ar[r] \ar@{=}[d]& 0
\\ 0 \ar[r]  & {\begin{pmatrix}\begin{smallmatrix} X \\ 0 \end{smallmatrix}\end{pmatrix}_{0,0}} \ar[r]^-{\binom{1}{0}} & {\begin{pmatrix}\begin{smallmatrix} X\\ Y\oplus Y \end{smallmatrix}\end{pmatrix}_{0, (g_1, g_2)}} \ar[r]^-{\binom{0}{1}} &   {\begin{pmatrix}\begin{smallmatrix} 0 \\ Y\oplus Y \end{smallmatrix}\end{pmatrix}_{0, 0}} \ar[r] & 0.}\]

\noindent and pullback

\[\xymatrix{0\ar[r] & {\begin{pmatrix}\begin{smallmatrix} X \\ 0 \end{smallmatrix}\end{pmatrix}_{0, 0}}\ar@{=}[d]
\ar[r]^-{\begin{pmatrix}\begin{smallmatrix} 1 \\ 0 \end{smallmatrix}\end{pmatrix}} & {\begin{pmatrix}\begin{smallmatrix} X \\ Y \end{smallmatrix}\end{pmatrix}_{0, g_1+ g_2}}
\ar[d]_-{\binom{1}{\binom{1}{1}}}\ar[r]^-{\begin{pmatrix}\begin{smallmatrix} 0 \\ 1 \end{smallmatrix}\end{pmatrix}} &   {\begin{pmatrix}\begin{smallmatrix} 0 \\ Y \end{smallmatrix}\end{pmatrix}_{0, 0}} \ar[r] \ar[d]^-{\binom{0}{\binom{1}{1}}}& 0
\\ 0 \ar[r]  & {\begin{pmatrix}\begin{smallmatrix} X \\ 0 \end{smallmatrix}\end{pmatrix}_{0,0}} \ar[r]^-{\binom{1}{0}} & {\begin{pmatrix}\begin{smallmatrix} X\\ Y\oplus Y \end{smallmatrix}\end{pmatrix}_{0, (g_1, g_2)}} \ar[r]^-{\binom{0}{1}} &   {\begin{pmatrix}\begin{smallmatrix} 0 \\ Y\oplus Y \end{smallmatrix}\end{pmatrix}_{0, 0}} \ar[r] & 0.}\]

\vskip5pt

The assertion $(2)$ can be similarly proved. \end{proof}

\begin{lem}\label{iso}
Let $\Lambda=\begin{pmatrix}\begin{smallmatrix}
A & N \\
M & B \\
\end{smallmatrix}\end{pmatrix}$ be a Morita ring with $N\otimes_{B}M =0$, $_{A}X$ an $A$-module and $_{B}Y$ a $B$-module. Assume that  $N_{B}$ is a projective module. Then

\vskip 5pt

{\rm (1)} \  For any injective $A$-module $_{A}I$, there holds \ ${\rm Ext}_{\Lambda}^{i}(\begin{pmatrix}\begin{smallmatrix}0\\Y\end{smallmatrix}\end{pmatrix}_{0, 0},\ \begin{pmatrix}\begin{smallmatrix}I\\0\end{smallmatrix}\end{pmatrix}_{0, 0})=0,\ \forall\ i\geq 2$.

\vskip 5pt

{\rm (2)} \  There is an isomorphism \ ${\rm Ext}_{\Lambda}^{i+1}(\begin{pmatrix}\begin{smallmatrix}0\\Y\end{smallmatrix}\end{pmatrix}_{0, 0},\ \begin{pmatrix}\begin{smallmatrix}X\\0\end{smallmatrix}\end{pmatrix}_{0, 0})
\iso {\rm Ext}_{A}^{i}(N\otimes_{B}Y, X),\ \forall\ i\geq 0,$ which  is functorial both in $X$ and $Y$.
\vskip5pt Moreover, if \ $0 \longrightarrow  X \longrightarrow I \longrightarrow I^1 \stackrel {d^1}\longrightarrow \cdots $ is an injective resolution of $X$, then the isomorphisms  make the following diagram commute, where $\partial$ refers to the connecting maps$:$

\[\xymatrix@C=12pt@R=20pt{
{\rm Ext}_{\Lambda}^{1}({\begin{pmatrix}\begin{smallmatrix}0\\Y\end{smallmatrix}\end{pmatrix}_{0, 0}}, {\begin{pmatrix}\begin{smallmatrix}\Ker d^1 \\0\end{smallmatrix}\end{pmatrix}_{0, 0}})\ar[r]^-{\partial} \ar[d]_-{\iso}
&   {\rm Ext}_\Lambda^2({\begin{pmatrix}\begin{smallmatrix}0\\Y\end{smallmatrix}\end{pmatrix}_{0, 0}}, {\begin{pmatrix}\begin{smallmatrix}X\\0\end{smallmatrix}\end{pmatrix}_{0, 0}})\ar[d]^{\iso}
\\ {\rm Hom}_A(N\otimes_B Y, \Ker d^1)\ar[r]^-{\partial}   &    {\rm Ext}_A^1(N\otimes_BY, X). }
\]

\end{lem}
\begin{proof}\ (1) \ By the exact sequence of left $\Lambda$-modules
\[
0\lxr \begin{pmatrix}\begin{smallmatrix}I\\0\end{smallmatrix}\end{pmatrix}_{0, 0}\lxr \begin{pmatrix}\begin{smallmatrix}I\\{\rm Hom}_{A}(N, I)\end{smallmatrix}\end{pmatrix}_{0, e_{I}}\lxr \begin{pmatrix}\begin{smallmatrix}0\\{\rm Hom}_{A}(N, I)\end{smallmatrix}\end{pmatrix}_{0, 0}\lxr 0
\]
one has an exact sequence for $i\ge 1$
\[
{\rm Ext}_{\Lambda}^{i}(\begin{pmatrix}\begin{smallmatrix}0\\Y\end{smallmatrix}\end{pmatrix}_{0, 0},\ \begin{pmatrix}\begin{smallmatrix}I\\{\rm Hom}_{A}(N, I)\end{smallmatrix}\end{pmatrix}_{0, e_{I}})\lxr {\rm Ext}_{\Lambda}^{i}(\begin{pmatrix}\begin{smallmatrix}0\\Y\end{smallmatrix}\end{pmatrix}_{0, 0},\ \begin{pmatrix}\begin{smallmatrix}0\\{\rm Hom}_{A}(N, I)\end{smallmatrix}\end{pmatrix}_{0, 0})
\]
\[
\lxr {\rm Ext}_{\Lambda}^{i+1}(\begin{pmatrix}\begin{smallmatrix}0\\Y\end{smallmatrix}\end{pmatrix}_{0, 0},\ \begin{pmatrix}\begin{smallmatrix}I\\0\end{smallmatrix}\end{pmatrix}_{0, 0})\lxr {\rm Ext}_{\Lambda}^{i+1}(\begin{pmatrix}\begin{smallmatrix}0\\Y\end{smallmatrix}\end{pmatrix}_{0, 0},\ \begin{pmatrix}\begin{smallmatrix}I\\{\rm Hom}_{A}(N, I)\end{smallmatrix}\end{pmatrix}_{0, e_{I}}).
\]
Since $\begin{pmatrix}\begin{smallmatrix}I\\{\rm Hom}_{A}(N, I)\end{smallmatrix}\end{pmatrix}_{0, e_{I}}$ is an injective left $\Lambda$-module,
it follows that
$${\rm Ext}_{\Lambda}^{i}(\begin{pmatrix}\begin{smallmatrix}0\\Y\end{smallmatrix}\end{pmatrix}_{0, 0},\ \begin{pmatrix}\begin{smallmatrix}0\\{\rm Hom}_{A}(N, I)\end{smallmatrix}\end{pmatrix}_{0, 0})
\cong {\rm Ext}_{\Lambda}^{i+1}(\begin{pmatrix}\begin{smallmatrix}0\\Y\end{smallmatrix}\end{pmatrix}_{0, 0}, \ \begin{pmatrix}\begin{smallmatrix}I\\0\end{smallmatrix}\end{pmatrix}_{0, 0}).$$
Since $N_{B}$ is projective, one sees from the well-known isomorphism
\[
{\rm Ext}_{B}^{n}(Y,\ {\rm Hom}_{A}(N, I))\iso {\rm Hom}_{A}({\rm Tor}_{n}^{B}(N, Y),\ I), \ \forall \ Y\in B\modu, \ \forall \ n\geq 1
\]
 that ${\rm Hom}_{A}(N, I)$ is an injective left $B$-module. Since $N\otimes_{B}M=0$, it follows that ${\rm Hom}_{B}(M, {\rm Hom}_{A}(N, I))=0$, and hence
\[
\begin{pmatrix}\begin{smallmatrix}0\\{\rm Hom}_{A}(N, I)\end{smallmatrix}\end{pmatrix}_{0, 0}=\begin{pmatrix}\begin{smallmatrix}{\rm Hom}_{B}(M, {\rm Hom}_{A}(N, I))\\{\rm Hom}_{A}(N, I)\end{smallmatrix}\end{pmatrix}_{0, 0}\]
is an injective left $\Lambda$-module. Therefore
\[
{\rm Ext}_{\Lambda}^{i+1}(\begin{pmatrix}\begin{smallmatrix}0\\Y\end{smallmatrix}\end{pmatrix}_{0, 0},\ \begin{pmatrix}\begin{smallmatrix}I\\0\end{smallmatrix}\end{pmatrix}_{0, 0})\iso {\rm Ext}_{\Lambda}^{i}(\begin{pmatrix}\begin{smallmatrix}0\\Y\end{smallmatrix}\end{pmatrix}_{0, 0},\ \begin{pmatrix}\begin{smallmatrix}0\\{\rm Hom}_{A}(N, I)\end{smallmatrix}\end{pmatrix}_{0, 0}) = 0, \ \forall \ i\geq 1. \]

\vskip 5pt

(2) \ If $i=0$ then the assertion is Lemma \ref{baersum}(1). For $i =1$, let $0\lxr X\lxr I^{0}\stackrel{d^{0}}{\lxr} I^{1}\stackrel{d^{1}}{\lxr} \cdots$
be an injective resolution of $X$. For $m\geq 0$, the exact sequence $0\lxr \Ker{d^{m}}\lxr I^{m}\lxr \Ker{d^{m+1}}\lxr 0$ induces an exact sequence of left $\Lambda$-modules:
\[
0\lxr \begin{pmatrix}\begin{smallmatrix}\Ker{d^{m}}\\0\end{smallmatrix}\end{pmatrix}_{0, 0}\lxr \begin{pmatrix}\begin{smallmatrix}I^{m}\\0\end{smallmatrix}\end{pmatrix}_{0, 0}\lxr \begin{pmatrix}\begin{smallmatrix}\Ker{d^{m+1}}\\0\end{smallmatrix}\end{pmatrix}_{0, 0}\lxr 0.
\]
By Lemma \ref{baersum}(1) and the assertion (1) one gets a commutative diagram with exact rows:
{\footnotesize
\[\xymatrix@C=12pt@R=20pt{
{\rm Ext}_{\Lambda}^{1}({\begin{pmatrix}\begin{smallmatrix}0\\Y\end{smallmatrix}\end{pmatrix}_{0, 0}}, {\begin{pmatrix}\begin{smallmatrix}I^{m}\\0\end{smallmatrix}\end{pmatrix}_{0, 0}})\ar[r] \ar[d]_-{\iso}   &   {\rm Ext}_{\Lambda}^{1}({\begin{pmatrix}\begin{smallmatrix}0\\Y\end{smallmatrix}\end{pmatrix}_{0, 0}}, {\begin{pmatrix}\begin{smallmatrix}\Ker{d^{m+1}}\\0\end{smallmatrix}\end{pmatrix}_{0, 0}})\ar[r]^-{\partial} \ar[d]_-{\iso}  &   {\rm Ext}_{\Lambda}^{2}({\begin{pmatrix}\begin{smallmatrix}0\\Y\end{smallmatrix}\end{pmatrix}_{0, 0}}, {\begin{pmatrix}\begin{smallmatrix}\Ker{d^{m}}\\0\end{smallmatrix}\end{pmatrix}_{0, 0}})\ar[r]\ar@{.>}[d]_-{\iso}  &  0 \\
{\rm Hom}_{A}(N\otimes_{B}Y, I^{m})\ar[r]^{}   &   {\rm Hom}_{A}(N\otimes_{B}Y, \Ker{d^{m+1}})\ar[r]^-{\partial}   &    {\rm Ext}_{A}^{1}(N\otimes_{B}Y, \Ker{d^{m}})\ar[r]  & 0}
\]
}
where both $\partial$'s are the connecting maps. Thus
\[
{\rm Ext}_{\Lambda}^{2}({\begin{pmatrix}\begin{smallmatrix}0\\Y\end{smallmatrix}\end{pmatrix}_{0, 0}},\ {\begin{pmatrix}\begin{smallmatrix}\Ker{d^{m}}\\0\end{smallmatrix}\end{pmatrix}_{0, 0}})\iso {\rm Ext}_{A}^{1}(N\otimes_{B}Y,\ \Ker{d^{m}}),\ \forall\ m\geq 0.
\]
Taking $m = 0$ one gets ${\rm Ext}_{\Lambda}^{2}({\begin{pmatrix}\begin{smallmatrix}0\\Y\end{smallmatrix}\end{pmatrix}_{0, 0}},\ {\begin{pmatrix}\begin{smallmatrix}X\\0\end{smallmatrix}\end{pmatrix}_{0, 0}})\iso {\rm Ext}_{A}^{1}(N\otimes_{B}Y,\ X)$
and the commutative diagram

\[\xymatrix@C=12pt@R=20pt{
{\rm Ext}_{\Lambda}^{1}({\begin{pmatrix}\begin{smallmatrix}0\\Y\end{smallmatrix}\end{pmatrix}_{0, 0}}, {\begin{pmatrix}\begin{smallmatrix}\Ker d^1 \\0\end{smallmatrix}\end{pmatrix}_{0, 0}})\ar[r]^-{\partial} \ar[d]_-{\iso}
&   {\rm Ext}_\Lambda^2({\begin{pmatrix}\begin{smallmatrix}0\\Y\end{smallmatrix}\end{pmatrix}_{0, 0}}, {\begin{pmatrix}\begin{smallmatrix}X\\0\end{smallmatrix}\end{pmatrix}_{0, 0}})\ar[d]^{\iso}
\\ {\rm Hom}_A(N\otimes_B Y, \Ker d^1)\ar[r]^-{\partial}   &    {\rm Ext}_A^1(N\otimes_BY, X)}
\]
This proves the assertion holds for $i=1$.

\vskip 5pt

For $i\geq 2$, from the exact sequence of $\Lambda$-modules
\[
{\rm Ext}_{\Lambda}^{i}({\begin{pmatrix}\begin{smallmatrix}0\\Y\end{smallmatrix}\end{pmatrix}_{0, 0}}, {\begin{pmatrix}\begin{smallmatrix}I^{m}\\0\end{smallmatrix}\end{pmatrix}_{0, 0}})\lxr {\rm Ext}_{\Lambda}^{i}({\begin{pmatrix}\begin{smallmatrix}0\\Y\end{smallmatrix}\end{pmatrix}_{0, 0}}, {\begin{pmatrix}\begin{smallmatrix}\Ker{d^{m+1}}\\0\end{smallmatrix}\end{pmatrix}_{0, 0}})
\]
\[
\lxr {\rm Ext}_{\Lambda}^{i+1}({\begin{pmatrix}\begin{smallmatrix}0\\Y\end{smallmatrix}\end{pmatrix}_{0, 0}}, {\begin{pmatrix}\begin{smallmatrix}\Ker{d^{m}}\\0\end{smallmatrix}\end{pmatrix}_{0, 0}})\lxr {\rm Ext}_{\Lambda}^{i+1}({\begin{pmatrix}\begin{smallmatrix}0\\Y\end{smallmatrix}\end{pmatrix}_{0, 0}}, {\begin{pmatrix}\begin{smallmatrix}I^{m}\\0\end{smallmatrix}\end{pmatrix}_{0, 0}})\]
one sees from (1) that
\[
{\rm Ext}_{\Lambda}^{i+1}({\begin{pmatrix}\begin{smallmatrix}0\\Y\end{smallmatrix}\end{pmatrix}_{0, 0}}, {\begin{pmatrix}\begin{smallmatrix}\Ker{d^{m}}\\0\end{smallmatrix}\end{pmatrix}_{0, 0}})\iso {\rm Ext}_{\Lambda}^{i}({\begin{pmatrix}\begin{smallmatrix}0\\Y\end{smallmatrix}\end{pmatrix}_{0, 0}}, {\begin{pmatrix}\begin{smallmatrix}\Ker{d^{m+1}}\\0\end{smallmatrix}\end{pmatrix}_{0, 0}}),\ \forall\ m\geq 0.
\]
It follows that for $i\geq 2$ there holds
\[
\begin{aligned}
{\rm Ext}_{\Lambda}^{i+1}({\begin{pmatrix}\begin{smallmatrix}0\\Y\end{smallmatrix}\end{pmatrix}_{0, 0}}, {\begin{pmatrix}\begin{smallmatrix}X\\0\end{smallmatrix}\end{pmatrix}_{0, 0}}) & \iso {\rm Ext}_{\Lambda}^{i}({\begin{pmatrix}\begin{smallmatrix}0\\Y\end{smallmatrix}\end{pmatrix}_{0, 0}}, {\begin{pmatrix}\begin{smallmatrix}\Ker{d^{1}}\\0\end{smallmatrix}\end{pmatrix}_{0, 0}})\\
&\iso \cdots \iso {\rm Ext}_{\Lambda}^{2}({\begin{pmatrix}\begin{smallmatrix}0\\Y\end{smallmatrix}\end{pmatrix}_{0, 0}}, {\begin{pmatrix}\begin{smallmatrix}\Ker{d^{i-1}}\\0\end{smallmatrix}\end{pmatrix}_{0, 0}})\\
&\iso {\rm Ext}_{A}^{1}(N\otimes_{B}Y,\ \Ker{d^{i-1}}).
\end{aligned}
\]
On the other side, by dimension shift one gets
\[
{\rm Ext}_{A}^{i}(N\otimes_{B}Y,\ X)\iso {\rm Ext}_{A}^{1}(N\otimes_{B}Y,\ \Ker{d^{i-1}}),\ \mbox{for}\ i\geq 1.
\]
Hence
\[
{\rm Ext}_{\Lambda}^{i+1}({\begin{pmatrix}\begin{smallmatrix}0\\Y\end{smallmatrix}\end{pmatrix}_{0, 0}}, {\begin{pmatrix}\begin{smallmatrix}X\\0\end{smallmatrix}\end{pmatrix}_{0, 0}})\iso {\rm Ext}_{A}^{i}(N\otimes_{B}Y,\ X),\ \ \mbox{for}\ i\geq 2.
\]
This isomorphism is functorial in $X$ and $Y$, since so is every step. \end{proof}

\vskip 5pt

Similarly, one can prove the following fact, by using Lemma \ref{baersum}(2).

\begin{lem}\label{iso2}
Let $\Lambda=\begin{pmatrix}\begin{smallmatrix}
A & N \\
M & B \\
\end{smallmatrix}\end{pmatrix}$ be a Morita ring with $M\otimes_{A}N=0$, $_{A}X$ an $A$-module and $_{B}Y$ a $B$-module. Assume that  $M_{A}$ is a projective module.

\vskip 5pt

{\rm (1)} \  For any injective $B$-module $_{B}J$,  there holds \ ${\rm Ext}_{\Lambda}^{i}(\begin{pmatrix}\begin{smallmatrix}X\\0\end{smallmatrix}\end{pmatrix}_{0, 0},\ \begin{pmatrix}\begin{smallmatrix}0\\J\end{smallmatrix}\end{pmatrix}_{0, 0})=0,\ \forall\ i\geq2$.

\vskip 5pt

{\rm (2)} \  There is an isomorphism \ ${\rm Ext}_{\Lambda}^{i+1}(\begin{pmatrix}\begin{smallmatrix}X\\0\end{smallmatrix}\end{pmatrix}_{0, 0},\ \begin{pmatrix}\begin{smallmatrix}0\\Y\end{smallmatrix}\end{pmatrix}_{0, 0})
\iso {\rm Ext}_{B}^{i}(M\otimes_{A}X, Y),\ \forall\ i\geq 0,$  which is functorial both in $X$ and $Y$.
\vskip5pt Moreover, if \ $0 \longrightarrow  Y \longrightarrow J \longrightarrow J^1 \stackrel {d^1}\longrightarrow \cdots $ is an injective resolution of $Y$, then the isomorphisms  make the following diagram commutes,
where $\partial$ refers to the connecting maps$:$

\[\xymatrix@C=12pt@R=20pt{
{\rm Ext}_{\Lambda}^{1}({\begin{pmatrix}\begin{smallmatrix}X\\0\end{smallmatrix}\end{pmatrix}_{0, 0}}, {\begin{pmatrix}\begin{smallmatrix}0\\ \Ker d^1 \end{smallmatrix}\end{pmatrix}_{0, 0}})\ar[r]^-{\partial} \ar[d]_-{\iso}
&   {\rm Ext}_\Lambda^2({\begin{pmatrix}\begin{smallmatrix}X\\ 0\end{smallmatrix}\end{pmatrix}_{0, 0}}, {\begin{pmatrix}\begin{smallmatrix}0 \\ Y\end{smallmatrix}\end{pmatrix}_{0, 0}})\ar[d]^{\iso}
\\ {\rm Hom}_B(M\otimes_A X, \Ker d^1)\ar[r]^-{\partial}   &    {\rm Ext}_B^1(M\otimes_A X, Y). }
\]\end{lem}

\vskip5pt

\subsection {The left perpendicular categories $^\perp \begin{pmatrix}\begin{smallmatrix}A\\ 0\end{smallmatrix}\end{pmatrix}_{0, 0}$ and \ $^\perp \begin{pmatrix}\begin{smallmatrix}0\\ B\end{smallmatrix}\end{pmatrix}_{0, 0}$}
For a left $\Lambda$-module $\begin{pmatrix}\begin{smallmatrix}X\\Y\end{smallmatrix}\end{pmatrix}_{f, g}$, by the decomposition of $f$:
\[
\xymatrix@C=0.6cm@R=0.5cm{0\ar[r] & \Ker f\ar[r]^-\sigma &
M\otimes_{A}X \ar[rr]^-{f}    &&  Y \ar[r]^-{\pi} \ar@{<-_{)}}[dl]^-{\iota} &   \cok{f} \ar[r]   &    0\\
&&& \im{f} \ar@{<<-}[ul]^-{p}  &&& \\
}
\]
one has an exact sequence of left $\Lambda$-modules
$$
0\lxr {\begin{pmatrix}\begin{smallmatrix} X \\ \im{f} \end{smallmatrix}\end{pmatrix}_{p, 0}}
\stackrel{\begin{pmatrix}\begin{smallmatrix} 1 \\ \iota \end{smallmatrix}\end{pmatrix}}\lxr {\begin{pmatrix}\begin{smallmatrix} X \\ Y \end{smallmatrix}\end{pmatrix}_{f, g}}
\stackrel{\begin{pmatrix}\begin{smallmatrix} 0 \\ \pi \end{smallmatrix}\end{pmatrix}}\lxr  {\begin{pmatrix}\begin{smallmatrix} 0 \\ \cok{f} \end{smallmatrix}\end{pmatrix}_{0, 0}}  \lxr 0.
$$

\vskip5pt

Similarly, one has (for simplicity we use the same notation $\sigma, \ p, \ \iota, \ \pi$)
\[
\xymatrix@C=0.6cm@R=0.5cm{0\ar[r] & \Ker g\ar[r]^-\sigma &
N\otimes_{B}Y \ar[rr]^-{g}    &&  X \ar[r]^-{\pi} \ar@{<-_{)}}[dl]^-{\iota} &   \cok{g} \ar[r]   &    0\\
&&& \im{g} \ar@{<<-}[ul]^-{p}  &&& \\
}
\]
and the exact sequence of left $\Lambda$-modules
$$
0\lxr {\begin{pmatrix}\begin{smallmatrix} \im{g} \\ Y \end{smallmatrix}\end{pmatrix}_{0, p}}
\stackrel{\begin{pmatrix}\begin{smallmatrix}  \iota \\ 1 \end{smallmatrix}\end{pmatrix}}\lxr {\begin{pmatrix}\begin{smallmatrix} X \\ Y \end{smallmatrix}\end{pmatrix}_{f, g}}
\stackrel{\begin{pmatrix}\begin{smallmatrix} \pi \\ 0 \end{smallmatrix}\end{pmatrix}}\lxr  {\begin{pmatrix}\begin{smallmatrix} \cok{g} \\ 0 \end{smallmatrix}\end{pmatrix}_{0, 0}}  \lxr 0.$$

\begin{lem}\label{im} Let $\Lambda=\begin{pmatrix}\begin{smallmatrix}
A & N \\
M & B \\
\end{smallmatrix}\end{pmatrix}$ be a Morita ring,  and $\begin{pmatrix}\begin{smallmatrix}X\\Y\end{smallmatrix}\end{pmatrix}_{f, g}$ a left $\Lambda$-module.
Assume that $M_{A}$ and $N_{B}$ are projective modules.

\vskip5pt

$(1)$ \  If $N\otimes_{B}M = 0$, then for any injective left $A$-module $I$ one has

$${\rm Ext}_{\Lambda}^{i}(\begin{pmatrix}\begin{smallmatrix}X\\ \im{f}\end{smallmatrix}\end{pmatrix}_{p, 0}, \begin{pmatrix}\begin{smallmatrix}I\\0\end{smallmatrix}\end{pmatrix}_{0, 0})=0, \ \ \forall \ i\ge 1.$$

\vskip5pt

$(2)$ \  If $N\otimes_{B}M = 0$, then for any module \ $_AU$ there is an isomorphism which is functorial in $U:$

$${\rm Ext}_{\Lambda}^{i}({\begin{pmatrix}\begin{smallmatrix} X \\ \im{f} \end{smallmatrix}\end{pmatrix}_{p, 0}},\ {\begin{pmatrix}\begin{smallmatrix} U \\ 0 \end{smallmatrix}\end{pmatrix}_{0, 0}})\iso {\rm Ext}_{A}^{i}(X, U), \ \ \forall \ i\ge 0.$$

\vskip5pt

$(3)$ \  If $M\otimes_{A}N=0$, then for any injective left $B$-module $J$, one has

$${\rm Ext}_{\Lambda}^{i}(\begin{pmatrix}\begin{smallmatrix}\im{g} \\ Y\end{smallmatrix}\end{pmatrix}_{0, p}, \begin{pmatrix}\begin{smallmatrix}0\\ J\end{smallmatrix}\end{pmatrix}_{0, 0})=0, \ \ \forall \ i\ge 1.$$

\vskip5pt

$(4)$ \  If $M\otimes_{A}N=0$, then for any module \ $_AV$ there is an isomorphism which is functorial in $V:$

$${\rm Ext}_{\Lambda}^{i}({\begin{pmatrix}\begin{smallmatrix} \im{g} \\ Y \end{smallmatrix}\end{pmatrix}_{0, p}},\ {\begin{pmatrix}\begin{smallmatrix} 0 \\ V \end{smallmatrix}\end{pmatrix}_{0, 0}})\iso {\rm Ext}_{A}^{i}(Y, V), \ \ \forall \ i\ge 0.$$
\end{lem}

\begin{proof} \ $(1)$ \ By Lemma \ref{Extiso}(1)  one has

$${\rm Ext}_{\Lambda}^{i}(\begin{pmatrix}\begin{smallmatrix}X\\ M\otimes_{A}X \end{smallmatrix}\end{pmatrix}_{\id_{M\otimes_{A}X}, 0}, \begin{pmatrix}\begin{smallmatrix}I\\0\end{smallmatrix}\end{pmatrix}_{0, 0})
\cong {\rm Ext}_{A}^{i}(X, I) =0, \ \forall \ i\ge 1.$$

\vskip5pt

\noindent Since $N_B$ is projective and $N\otimes_BM = 0$, by the exact sequence $0 \longrightarrow N\otimes_B \Ker f \stackrel{1\otimes \sigma} \longrightarrow N\otimes_B M\otimes_A X$  one sees that $N\otimes_B \Ker f = 0$. It follows from Lemma \ref{iso}(2) that

$${\rm Ext}_{\Lambda}^{i}({\begin{pmatrix}\begin{smallmatrix} 0 \\ {\Ker}f \end{smallmatrix}\end{pmatrix}_{0, 0}},\ {\begin{pmatrix}\begin{smallmatrix} I \\ 0 \end{smallmatrix}\end{pmatrix}_{0, 0}})
\iso {\rm Ext}_A^{i-1}(N\otimes_{B}{\Ker}f, I)=0,  \ \forall \ i\ge 1.$$

\vskip5pt

\noindent Applying ${\rm Hom}_\Lambda(-, \begin{pmatrix}\begin{smallmatrix}I\\0\end{smallmatrix}\end{pmatrix}_{0, 0})$ to the exact sequence

$$0\lxr {\begin{pmatrix}\begin{smallmatrix} 0 \\ {\Ker}f \end{smallmatrix}\end{pmatrix}_{0, 0}}
\stackrel{\begin{pmatrix}\begin{smallmatrix} 0 \\ \sigma \end{smallmatrix}\end{pmatrix}} \lxr  {\begin{pmatrix}\begin{smallmatrix} X \\ M\otimes_{A}X \end{smallmatrix}\end{pmatrix}_{\id_{M\otimes_{A}X}, 0}}
\stackrel{\begin{pmatrix}\begin{smallmatrix} 1 \\ p \end{smallmatrix}\end{pmatrix}}\lxr  {\begin{pmatrix}\begin{smallmatrix} X \\ \im{f} \end{smallmatrix}\end{pmatrix}_{p, 0}}  \lxr 0$$

\vskip5pt \noindent one sees that ${\rm Ext}_{\Lambda}^{i}(\begin{pmatrix}\begin{smallmatrix}X\\ \im{f}\end{smallmatrix}\end{pmatrix}_{p, 0}, \begin{pmatrix}\begin{smallmatrix}I\\0\end{smallmatrix}\end{pmatrix}_{0, 0})=0$ for $i\ge 2$, and
the following exact sequence

\begin{align*}0 = & {\rm Hom}_\Lambda(\begin{pmatrix}\begin{smallmatrix} 0 \\ {\Ker}f \end{smallmatrix}\end{pmatrix}_{0, 0}, \begin{pmatrix}\begin{smallmatrix}I\\0\end{smallmatrix}\end{pmatrix}) \longrightarrow
{\rm Ext}_{\Lambda}^{1}(\begin{pmatrix}\begin{smallmatrix}X\\ \im{f}\end{smallmatrix}\end{pmatrix}_{p, 0}, \begin{pmatrix}\begin{smallmatrix}I\\0\end{smallmatrix}\end{pmatrix}) \longrightarrow
{\rm Ext}_{\Lambda}^{1}(\begin{pmatrix}\begin{smallmatrix}X\\ \\ M\otimes_{A}X \end{smallmatrix}\end{pmatrix}_{\id_{M\otimes_{A}X}, 0}, \begin{pmatrix}\begin{smallmatrix}I\\0\end{smallmatrix}\end{pmatrix}) = 0.\end{align*}
Thus also ${\rm Ext}_{\Lambda}^{1}(\begin{pmatrix}\begin{smallmatrix}X\\ \im{f}\end{smallmatrix}\end{pmatrix}_{p, 0}, \begin{pmatrix}\begin{smallmatrix}I\\0\end{smallmatrix}\end{pmatrix}_{0,0}) =0$.

\vskip5pt

$(2)$ \ Applying ${\rm Hom}_{\Lambda}(-, {\begin{pmatrix}\begin{smallmatrix} U \\ 0 \end{smallmatrix}\end{pmatrix}_{0, 0}})$ to the exact sequence

$$
0\lxr {\begin{pmatrix}\begin{smallmatrix} 0 \\ {\Ker}f \end{smallmatrix}\end{pmatrix}_{0, 0}}
\stackrel{\begin{pmatrix}\begin{smallmatrix} 0 \\ \sigma \end{smallmatrix}\end{pmatrix}} \lxr  {\begin{pmatrix}\begin{smallmatrix} X \\ M\otimes_{A}X \end{smallmatrix}\end{pmatrix}_{\id_{M\otimes_{A}X}, 0}}
\stackrel{\begin{pmatrix}\begin{smallmatrix} 1 \\ p \end{smallmatrix}\end{pmatrix}}\lxr  {\begin{pmatrix}\begin{smallmatrix} X \\ \im{f} \end{smallmatrix}\end{pmatrix}_{p, 0}}  \lxr 0
$$
 one gets an exact sequence for $i\geq 0$

\[
{\rm Ext}_{\Lambda}^{i}({\begin{pmatrix}\begin{smallmatrix} 0 \\ {\Ker}f \end{smallmatrix}\end{pmatrix}_{0, 0}},\ {\begin{pmatrix}\begin{smallmatrix} U \\ 0 \end{smallmatrix}\end{pmatrix}_{0, 0}}
)\lxr {\rm Ext}_{\Lambda}^{i+1}({\begin{pmatrix}\begin{smallmatrix} X \\ \im{f} \end{smallmatrix}\end{pmatrix}_{p, 0}},\ {\begin{pmatrix}\begin{smallmatrix} U \\ 0 \end{smallmatrix}\end{pmatrix}_{0, 0}}) \lxr
\]
\[
{\rm Ext}_{\Lambda}^{i+1}({\begin{pmatrix}\begin{smallmatrix} X \\ M\otimes_{A}X \end{smallmatrix}\end{pmatrix}_{\id_{M\otimes_{A}X}, 0}},\ {\begin{pmatrix}\begin{smallmatrix} U \\ 0 \end{smallmatrix}\end{pmatrix}_{0, 0}})\lxr {\rm Ext}_{\Lambda}^{i+1}({\begin{pmatrix}\begin{smallmatrix} 0 \\ {\Ker}f \end{smallmatrix}\end{pmatrix}_{0, 0}},\ {\begin{pmatrix}\begin{smallmatrix} U \\ 0 \end{smallmatrix}\end{pmatrix}_{0, 0}}).
\]

\noindent Since $N\otimes_B \Ker f = 0$, it follows from Lemma \ref{iso}(2) that
$${\rm Ext}_{\Lambda}^{i}({\begin{pmatrix}\begin{smallmatrix} 0 \\ {\Ker}f \end{smallmatrix}\end{pmatrix}_{0, 0}},\ {\begin{pmatrix}\begin{smallmatrix} U \\ 0 \end{smallmatrix}\end{pmatrix}_{0, 0}})
\iso \begin{cases} {\rm Ext}_{\Lambda}^{i-1}(N\otimes_{B}{\Ker}f, U)=0, & \ \forall \ i\ge 1; \\ 0, & i = 0.\end{cases}$$
Hence, by Lemma \ref{Extiso}(1) one gets
\[
\begin{aligned}
{\rm Ext}_{\Lambda}^{i}({\begin{pmatrix}\begin{smallmatrix} X \\ \im{f} \end{smallmatrix}\end{pmatrix}_{p, 0}},\ {\begin{pmatrix}\begin{smallmatrix} U \\ 0 \end{smallmatrix}\end{pmatrix}_{0, 0}}) & \iso {\rm Ext}_{\Lambda}^{i}({\begin{pmatrix}\begin{smallmatrix} X \\ M\otimes_{A}X \end{smallmatrix}\end{pmatrix}_{\id_{M\otimes_{A}X}, 0}},\ {\begin{pmatrix}\begin{smallmatrix} U \\ 0 \end{smallmatrix}\end{pmatrix}_{0, 0}}) \iso {\rm Ext}_{A}^{i}(X, U)
\end{aligned}
\]
for all $i\geq 1$. This isomorphism is functorial in $U$, since every step involved is functorial in $U$.

\vskip5pt

The assertions $(3)$ and $(4)$ can be similarly proved, by using Lemmas \ref{Extiso}(2) and \ref{iso2}(2).
 \end{proof}

\vskip5pt

Furthermore, the isomorphisms in Lemma \ref{im}(2) and (4) are compatible with the connecting map, as the following lemma claims. This is important to get semi-Gorenstein-projective $\Lambda$-modules.

\vskip5pt

\begin{lem}\label{induction} Let $\Lambda=\begin{pmatrix}\begin{smallmatrix}
A & N \\
M & B \\
\end{smallmatrix}\end{pmatrix}$ be a Morita ring,  and $\begin{pmatrix}\begin{smallmatrix}X\\Y\end{smallmatrix}\end{pmatrix}_{f, g}$ a left $\Lambda$-module. Assume that $M_{A}$ and $N_{B}$ are projective modules. Then

\vskip5pt

$(1)$ \ If  $N\otimes_{B}M = 0$, then for any left $A$-module $U$, there is a commutative diagram for $i\ge 0,$ where $\partial$ is the connecting map$:$

\[\xymatrix{
{\rm Ext}_{\Lambda}^{i}({\begin{pmatrix}\begin{smallmatrix} X \\ \im{f} \end{smallmatrix}\end{pmatrix}_{p, 0}}, {\begin{pmatrix}\begin{smallmatrix} U \\ 0 \end{smallmatrix}\end{pmatrix}})\ar[r]^-{\partial} \ar[d]_-{\iso}   &   {\rm Ext}_{\Lambda}^{i+1}({\begin{pmatrix}\begin{smallmatrix} 0 \\ \cok{f} \end{smallmatrix}\end{pmatrix}_{0, 0}}, {\begin{pmatrix}\begin{smallmatrix} U \\ 0 \end{smallmatrix}\end{pmatrix}})\ar[d]^{\iso}\\
{\rm Ext}_{A}^{i}(X, U)\ar[r]^-{{\rm Ext}_\Lambda^{i}(g, U)}   &   {\rm Ext}_{A}^{i}(N\otimes_{B}Y, U).}
\]

\vskip5pt

$(2)$ \ If $M\otimes_{A}N=0$, then for any left $B$-module $V$, there is a commutative diagram for $i\ge 0,$ where $\partial$ is the connecting map$:$

\[\xymatrix{
{\rm Ext}_{\Lambda}^{i}({\begin{pmatrix}\begin{smallmatrix} \im g \\ Y \end{smallmatrix}\end{pmatrix}_{0, p}}, {\begin{pmatrix}\begin{smallmatrix} 0 \\ V \end{smallmatrix}\end{pmatrix}})\ar[r]^-{\partial} \ar[d]_-{\iso}   &   {\rm Ext}_{\Lambda}^{i+1}({\begin{pmatrix}\begin{smallmatrix} \cok{g} \\ 0 \end{smallmatrix}\end{pmatrix}_{0, 0}}, {\begin{pmatrix}\begin{smallmatrix} 0 \\ V \end{smallmatrix}\end{pmatrix}})\ar[d]^{\iso}\\
{\rm Ext}_{B}^{i}(Y, V)\ar[r]^-{{\rm Ext}_\Lambda^{i}(g, V)}   &   {\rm Ext}_{B}^{i}(M\otimes_{A}X, V).}
\]
\end{lem}

\begin{proof}  \ $(1)$ \ Use induction on $i\ge 0$. For $i = 0$, we need to prove that there is a commutative diagram

\[\xymatrix{
{\rm Hom}_{\Lambda}({\begin{pmatrix}\begin{smallmatrix} X \\ \im{f} \end{smallmatrix}\end{pmatrix}_{p, 0}}, {\begin{pmatrix}\begin{smallmatrix} U \\ 0 \end{smallmatrix}\end{pmatrix}})\ar[r]^-{\partial} \ar[d]_-{\iso}^-{\Psi}   &   {\rm Ext}_{\Lambda}^{1}({\begin{pmatrix}\begin{smallmatrix} 0 \\ \cok{f} \end{smallmatrix}\end{pmatrix}_{0, 0}}, {\begin{pmatrix}\begin{smallmatrix} U \\ 0 \end{smallmatrix}\end{pmatrix}})\ar[d]_-{\iso}^-{\Phi^{-1}}\\
{\rm Hom}_{A}(X, U)\ar[r]^-{{\rm Hom}_A(g, U)}   &   {\rm Hom}_{A}(N\otimes_{B}Y, U).}
\]

\noindent In fact, taking $\Psi$ to be the homomorphism of abelian groups given by $\begin{pmatrix}\begin{smallmatrix} a \\ 0 \end{smallmatrix}\end{pmatrix}\mapsto a$,  and taking
$\Phi$ to be the homomorphism of abelian groups given in Lemma \ref{baersum}(1), which sends $h\in {\rm Hom}_{A}(N\otimes_{B}Y, U)$ to the equivalence class of exact sequence
$$0 \longrightarrow \begin{pmatrix}\begin{smallmatrix} U \\ 0 \end{smallmatrix}\end{pmatrix}_{0,0}\longrightarrow \begin{pmatrix}\begin{smallmatrix} U \\ \cok f \end{smallmatrix}\end{pmatrix}_{0, h\circ ({\rm Id}_N\otimes \pi)^{-1}} \longrightarrow  \begin{pmatrix}\begin{smallmatrix} 0 \\ \cok{f} \end{smallmatrix}\end{pmatrix}_{0, 0} \longrightarrow 0$$ in ${\rm Ext}_{\Lambda}^{1}(\begin{pmatrix}\begin{smallmatrix} 0 \\ \cok{f} \end{smallmatrix}\end{pmatrix}_{0, 0}, \begin{pmatrix}\begin{smallmatrix} U \\ 0 \end{smallmatrix}\end{pmatrix}).$ (Note that, by the exact sequence $M\otimes_AX\stackrel f \longrightarrow Y \stackrel {\pi} \longrightarrow \cok f\longrightarrow 0$
and the assumption $N\otimes_BM = 0$, the map ${\rm Id}_N\otimes \pi:  N\otimes_{B}Y\longrightarrow N\otimes_{B}\cok f$ is an isomorphism.)
The connecting map $\partial$ is by definition given by the pushout

\[\xymatrix{0\ar[r] & {\begin{pmatrix}\begin{smallmatrix} X \\ \im{f} \end{smallmatrix}\end{pmatrix}_{p, 0}}\ar[d]_{\binom{a}{0}}\ar[r]^-{\begin{pmatrix}\begin{smallmatrix} 1 \\ \iota \end{smallmatrix}\end{pmatrix}} & {\begin{pmatrix}\begin{smallmatrix} X \\ Y \end{smallmatrix}\end{pmatrix}_{f, g}} \ar[d]_-{\binom{a}{\pi}}\ar[r]^-{\begin{pmatrix}\begin{smallmatrix} 0 \\ \pi \end{smallmatrix}\end{pmatrix}} &   {\begin{pmatrix}\begin{smallmatrix} 0 \\ \cok{f} \end{smallmatrix}\end{pmatrix}_{0, 0}} \ar[r] \ar@{=}[d]& 0
\\ 0 \ar[r]  & {\begin{pmatrix}\begin{smallmatrix} U \\ 0 \end{smallmatrix}\end{pmatrix}_{0,0}} \ar[r]^-{\binom{1}{0}} & {\begin{pmatrix}\begin{smallmatrix} U \\ \cok f \end{smallmatrix}\end{pmatrix}_{f', g'}} \ar[r]^-{\binom{0}{1}} &   {\begin{pmatrix}\begin{smallmatrix} 0 \\ \cok{f} \end{smallmatrix}\end{pmatrix}_{0, 0}} \ar[r] & 0.}\]

\vskip5pt

\noindent We need to determine the structure maps $f': M\otimes_A U \longrightarrow \cok f$ and $g': N\otimes_B\cok f \longrightarrow U.$ Since
$\binom{1}{0}$ is a left $\Lambda$-homomorphism,
$f' = 0$. Since
$\binom{a}{\pi}$ is a left $\Lambda$-homomorphism,
by the commutative diagram

\[
\xymatrix@C=0.5cm{
N\otimes_{B}Y\ar@{->}[d]_-{g}\ar@{->}[rr]^-{1\otimes \pi} && N\otimes_{B}\cok f\ar@{->}[d]^-{g^{\prime}}\\
X\ar@{->}[rr]^{a}     && U\\
}
\]

\noindent one has $g' = a\circ g\circ ({\rm Id}_N\otimes \pi)^{-1}.$ Thus

$$ \Phi^{-1} \partial (\begin{pmatrix}\begin{smallmatrix} a \\ 0 \end{smallmatrix}\end{pmatrix}) = a\circ g = {\rm Hom}_\Lambda(g, U)\Psi (\begin{pmatrix}\begin{smallmatrix} a \\ 0 \end{smallmatrix}\end{pmatrix}).$$

\vskip5pt
\noindent This proves the assertion for $i = 0$.

\vskip5pt

For $i = 1$, we need to prove that there is a commutative diagram

\[\xymatrix{
{\rm Ext}_{\Lambda}^1({\begin{pmatrix}\begin{smallmatrix} X \\ \im{f} \end{smallmatrix}\end{pmatrix}_{p, 0}}, {\begin{pmatrix}\begin{smallmatrix} U \\ 0 \end{smallmatrix}\end{pmatrix}})\ar[r]^-{\partial} \ar[d]_-{\iso}   &   {\rm Ext}_{\Lambda}^{2}({\begin{pmatrix}\begin{smallmatrix} 0 \\ \cok{f} \end{smallmatrix}\end{pmatrix}_{0, 0}}, {\begin{pmatrix}\begin{smallmatrix} U \\ 0 \end{smallmatrix}\end{pmatrix}})\ar[d]_-{\iso}\\
{\rm Ext}_{A}^1(X, U)\ar[r]^-{{\rm Ext}^1_\Lambda(g, U)}   &   {\rm Ext}^1_{A}(N\otimes_{B}Y, U).}
\]

\vskip5pt \noindent For this let  $0\lxr U\lxr I\stackrel{d^{0}}{\lxr} I^{1}\stackrel{d^{1}}{\lxr} \cdots$
be an injective resolution of $U$. Consider the exact sequences
$$0\lxr {\begin{pmatrix}\begin{smallmatrix} U \\ 0 \end{smallmatrix}\end{pmatrix}}
\lxr  {\begin{pmatrix}\begin{smallmatrix} I \\ 0 \end{smallmatrix}\end{pmatrix}}
\lxr  {\begin{pmatrix}\begin{smallmatrix} \Ker d^1 \\ 0 \end{smallmatrix}\end{pmatrix}}  \lxr 0
$$
and
$$
0\lxr {\begin{pmatrix}\begin{smallmatrix} X \\ \im{f} \end{smallmatrix}\end{pmatrix}_{p, 0}}
\stackrel{\begin{pmatrix}\begin{smallmatrix} 1 \\ \iota \end{smallmatrix}\end{pmatrix}}\lxr {\begin{pmatrix}\begin{smallmatrix} X \\ Y \end{smallmatrix}\end{pmatrix}_{f, g}}
\stackrel{\begin{pmatrix}\begin{smallmatrix} 0 \\ \pi \end{smallmatrix}\end{pmatrix}}\lxr  {\begin{pmatrix}\begin{smallmatrix} 0 \\ \cok{f} \end{smallmatrix}\end{pmatrix}_{0, 0}}  \lxr 0.
$$

\vskip5pt
\noindent We claim that each parallelogram and each square commute in the following diagram of  exact rows (to save the space, we omit ${\rm Hom}$ and write ${\rm Ext}$ as ${\rm E}$), where the isomorphisms follow from
Lemmas \ref{im}(2) and \ref{iso}(2):

\vskip10pt

{\tiny\[\xymatrix@C=0.00001cm{({\begin{pmatrix}\begin{smallmatrix} X \\ \im f \end{smallmatrix}\end{pmatrix}}, {\begin{pmatrix}\begin{smallmatrix} I \\ 0 \end{smallmatrix}\end{pmatrix}}) \ar[rd]^-{\cong}\ar[rr] \ar[dd] &&
({\begin{pmatrix}\begin{smallmatrix} X \\ \im f \end{smallmatrix}\end{pmatrix}}, {\begin{pmatrix}\begin{smallmatrix} \Ker d^1 \\ 0 \end{smallmatrix}\end{pmatrix}})\ar[rr] \ar[rd]^-{\cong} \ar[dd]
& &{\rm E}^{1}({\begin{pmatrix}\begin{smallmatrix} X \\ \im f \end{smallmatrix}\end{pmatrix}}, {\begin{pmatrix}\begin{smallmatrix} U \\ 0 \end{smallmatrix}\end{pmatrix}})\ar[rd]^-{\cong}\ar[rr] \ar[dd]
& & 0
\\ & (X, I) \ar[rr] \ar[dd] &&
(X, \Ker d^1)\ar[rr]  \ar[dd]
&& {\rm E}^{1}(X, U)\ar[r] \ar[dd]
&& 0
\\
{\rm E}^{1}({\begin{pmatrix}\begin{smallmatrix} 0 \\ \cok{f} \end{smallmatrix}\end{pmatrix}}, {\begin{pmatrix}\begin{smallmatrix} I \\ 0 \end{smallmatrix}\end{pmatrix}})\ar[rd]^-{\cong}\ar[rr]  &&
{\rm E}^{1}({\begin{pmatrix}\begin{smallmatrix} 0 \\ \cok{f} \end{smallmatrix}\end{pmatrix}}, {\begin{pmatrix}\begin{smallmatrix} \Ker d^1 \\ 0 \end{smallmatrix}\end{pmatrix}})\ar[rd]^-{\cong}\ar[rr]
&& {\rm E}^{2}({\begin{pmatrix}\begin{smallmatrix} 0 \\ \cok{f} \end{smallmatrix}\end{pmatrix}}, {\begin{pmatrix}\begin{smallmatrix} U \\ 0 \end{smallmatrix}\end{pmatrix}})\ar[rd]^-{\cong}\ar[rr]
&& 0
\\
& (N\otimes_BY, I) \ar[rr] &&
(N\otimes_BY, \Ker d^1)\ar[rr]
&& {\rm E}^{1}(N\otimes_BY, U)\ar[rr]
&& 0.
}\]}

\vskip5pt

\noindent (Note that ${\rm Ext}_{\Lambda}^{i}(\begin{pmatrix}\begin{smallmatrix}X\\ \im{f}\end{smallmatrix}\end{pmatrix}_{p, 0}, \begin{pmatrix}\begin{smallmatrix}I\\0\end{smallmatrix}\end{pmatrix}_{0, 0})=0, \ \ \forall \ i\ge 1,$ by Lemma \ref{im}(1);
and that ${\rm Ext}_{\Lambda}^{i}(\begin{pmatrix}\begin{smallmatrix}0\\Y\end{smallmatrix}\end{pmatrix}_{0, 0},\ \begin{pmatrix}\begin{smallmatrix}I\\0\end{smallmatrix}\end{pmatrix}_{0, 0})=0,\ \forall\ i\geq 2$, by Lemma \ref{iso}(1).)

\vskip5pt

In fact, by the case $i = 0$ one has the commutative parallelogram:

{\tiny\[\xymatrix@R=0.5cm@C=0.01cm{{\rm Hom}_\Lambda({\begin{pmatrix}\begin{smallmatrix} X \\ \im f \end{smallmatrix}\end{pmatrix}}, {\begin{pmatrix}\begin{smallmatrix} \Ker d^1 \\ 0 \end{smallmatrix}\end{pmatrix}}) \ar[rd]^-{\cong}\ar[dd]_-{\partial}
\\ & {\rm Hom}_A(X, \Ker d^1)\ar[dd]^-{{\rm Hom}_A(g, \Ker d^1)}
\\
{\rm Ext}^{1}_\Lambda({\begin{pmatrix}\begin{smallmatrix} 0 \\ \cok{f} \end{smallmatrix}\end{pmatrix}}, {\begin{pmatrix}\begin{smallmatrix} \Ker d^1 \\ 0 \end{smallmatrix}\end{pmatrix}})\ar[rd]^-{\cong}
\\
& {\rm Hom}_A(N\otimes_BY, \Ker d^1). }\]}

\noindent By the naturality of the connecting map one gets  the commutative squares

{\tiny\[\xymatrix@R=0.5cm@C=0.2cm {({\begin{pmatrix}\begin{smallmatrix} X \\ \im f \end{smallmatrix}\end{pmatrix}}, {\begin{pmatrix}\begin{smallmatrix} \Ker d^1 \\ 0 \end{smallmatrix}\end{pmatrix}})\ar[rr] \ar[d]
&& {\rm E}^{1}({\begin{pmatrix}\begin{smallmatrix} X \\ \im f \end{smallmatrix}\end{pmatrix}}, {\begin{pmatrix}\begin{smallmatrix} U \\ 0 \end{smallmatrix}\end{pmatrix}})\ar[d]
\\
{\rm E}^{1}({\begin{pmatrix}\begin{smallmatrix} 0 \\ \cok{f} \end{smallmatrix}\end{pmatrix}}, {\begin{pmatrix}\begin{smallmatrix} \Ker d^1 \\ 0 \end{smallmatrix}\end{pmatrix}})\ar[rr]
&& {\rm E}^{2}({\begin{pmatrix}\begin{smallmatrix} 0 \\ \cok{f} \end{smallmatrix}\end{pmatrix}}, {\begin{pmatrix}\begin{smallmatrix} U \\ 0 \end{smallmatrix}\end{pmatrix}})}
\ \ \ \ \ \ \ \ \ \ \
\xymatrix@R=0.5cm@C=0.2cm {(X, \Ker d^1)\ar[rr]\ar[d]
&& {\rm E}^{1}(X, U)\ar[d]
\\
(N\otimes_BY, \Ker d^1)\ar[rr]
&& {\rm E}^{1}(N\otimes_BY, U).}\]}

\vskip5pt

The commutativity of the parallelogram

{\tiny \[\xymatrix@C=0.3cm{{\rm Hom}_\Lambda({\begin{pmatrix}\begin{smallmatrix} X \\ \im f \end{smallmatrix}\end{pmatrix}}, {\begin{pmatrix}\begin{smallmatrix} I \\ 0 \end{smallmatrix}\end{pmatrix}})\ar[rr] \ar[rd]^-{\cong}
& & {\rm Hom}_\Lambda({\begin{pmatrix}\begin{smallmatrix} X \\ \im f \end{smallmatrix}\end{pmatrix}}, {\begin{pmatrix}\begin{smallmatrix} \Ker d^1 \\ 0 \end{smallmatrix}\end{pmatrix}})\ar[rd]^-{\cong}
\\ & {\rm Hom}_A(X, I)\ar[rr]
&& {\rm Hom}_A(X, \Ker d^1)}\]}
is guaranteed by Lemma \ref{im}(2), and this induces the commutativity of the parallelogram involving the cokernels:

{\tiny \[\xymatrix@C=0.3cm{{\rm Hom}_\Lambda({\begin{pmatrix}\begin{smallmatrix} X \\ \im f \end{smallmatrix}\end{pmatrix}}, {\begin{pmatrix}\begin{smallmatrix} \Ker d^1 \\ 0 \end{smallmatrix}\end{pmatrix}})\ar[rr]^{\partial} \ar[rd]^-{\cong}
& & {\rm Ext}_\Lambda^{1}({\begin{pmatrix}\begin{smallmatrix} X \\ \im f \end{smallmatrix}\end{pmatrix}}, {\begin{pmatrix}\begin{smallmatrix} U \\ 0 \end{smallmatrix}\end{pmatrix}})\ar[rd]^-{\cong}
\\ & {\rm Hom}_A(X, \Ker d^1)\ar[rr]^{\partial'}
&& {\rm Ext}_A^{1}(X, U).}\]}

\vskip5pt

The commutativity of the parallelogram

{\tiny \[\xymatrix@C=0.1cm{{\rm Ext}^{1}({\begin{pmatrix}\begin{smallmatrix} 0 \\ \cok{f} \end{smallmatrix}\end{pmatrix}}, {\begin{pmatrix}\begin{smallmatrix} \Ker d^1 \\ 0 \end{smallmatrix}\end{pmatrix}})\ar[rd]_-{\cong}\ar[rr]
&& {\rm Ext}^{2}({\begin{pmatrix}\begin{smallmatrix} 0 \\ \cok{f} \end{smallmatrix}\end{pmatrix}}, {\begin{pmatrix}\begin{smallmatrix} U \\ 0 \end{smallmatrix}\end{pmatrix}})\ar[rd]_-{\cong}
\\ & {\rm Hom}(N\otimes_BY, \Ker d^1)\ar[rr]
&& {\rm Ext}^{1}(N\otimes_BY, U)}\]}

\noindent follows from Lemma \ref{iso}(2).

\vskip5pt

Finally, all these five commutative diagrams induce the commutativity of the following parallelogram, by a standard argument

{\tiny\[\xymatrix@C=0.00001cm{{\rm Ext}^{1}_\Lambda({\begin{pmatrix}\begin{smallmatrix} X \\ \im f \end{smallmatrix}\end{pmatrix}}, {\begin{pmatrix}\begin{smallmatrix} U \\ 0 \end{smallmatrix}\end{pmatrix}})\ar[rd]^-{\cong}\ar[dd]_{\partial}
\\ & {\rm Ext}^{1}_A(X, U)\ar[dd]^-{{\rm Ext}_A^{1}(g, U)}
\\
{\rm Ext}^{2}_\Lambda({\begin{pmatrix}\begin{smallmatrix} 0 \\ \cok{f} \end{smallmatrix}\end{pmatrix}}, {\begin{pmatrix}\begin{smallmatrix} U \\ 0 \end{smallmatrix}\end{pmatrix}})\ar[rd]^-{\cong}
\\
& {\rm Ext}_A^{1}(N\otimes_BY, U)
}\]}

\noindent This is precisely the assertion for $i = 1$.

\vskip10pt

For $i\ge 2,$ let  $0\lxr U\lxr I\stackrel{d^{0}}{\lxr} I^{1}\stackrel{d^{1}}{\lxr} \cdots$
be an injective resolution of $U$. Consider the exact sequences
$$0\lxr {\begin{pmatrix}\begin{smallmatrix} U \\ 0 \end{smallmatrix}\end{pmatrix}}
\lxr  {\begin{pmatrix}\begin{smallmatrix} I \\ 0 \end{smallmatrix}\end{pmatrix}}
\lxr  {\begin{pmatrix}\begin{smallmatrix} \Ker d^1 \\ 0 \end{smallmatrix}\end{pmatrix}}  \lxr 0
$$
and
$$
0\lxr {\begin{pmatrix}\begin{smallmatrix} X \\ \im{f} \end{smallmatrix}\end{pmatrix}_{p, 0}}
\stackrel{\begin{pmatrix}\begin{smallmatrix} 1 \\ \iota \end{smallmatrix}\end{pmatrix}}\lxr {\begin{pmatrix}\begin{smallmatrix} X \\ Y \end{smallmatrix}\end{pmatrix}_{f, g}}
\stackrel{\begin{pmatrix}\begin{smallmatrix} 0 \\ \pi \end{smallmatrix}\end{pmatrix}}\lxr  {\begin{pmatrix}\begin{smallmatrix} 0 \\ \cok{f} \end{smallmatrix}\end{pmatrix}_{0, 0}}  \lxr 0.
$$

\vskip5pt
\noindent Then one gets the following commutative diagram with exact rows and columns

\vskip10pt

{\tiny\[\xymatrix@C = 0.3cm {
{\rm Ext}_{\Lambda}^{i-1}({\begin{pmatrix}\begin{smallmatrix} X \\ \im f \end{smallmatrix}\end{pmatrix}}, {\begin{pmatrix}\begin{smallmatrix} I \\ 0 \end{smallmatrix}\end{pmatrix}})\ar[r] \ar[d] &
{\rm Ext}_{\Lambda}^{i-1}({\begin{pmatrix}\begin{smallmatrix} X \\ \im f \end{smallmatrix}\end{pmatrix}}, {\begin{pmatrix}\begin{smallmatrix} \Ker d^1 \\ 0 \end{smallmatrix}\end{pmatrix}})\ar[r]^-\cong  \ar[d]^-{\partial}
& {\rm Ext}_{\Lambda}^{i}({\begin{pmatrix}\begin{smallmatrix} X \\ \im f \end{smallmatrix}\end{pmatrix}}, {\begin{pmatrix}\begin{smallmatrix} U \\ 0 \end{smallmatrix}\end{pmatrix}})\ar[r] \ar[d]^-{\partial}
& {\rm Ext}_{\Lambda}^{i}({\begin{pmatrix}\begin{smallmatrix} X \\ \im f \end{smallmatrix}\end{pmatrix}}, {\begin{pmatrix}\begin{smallmatrix} I \\ 0 \end{smallmatrix}\end{pmatrix}}) \ar[d]
\\
{\rm Ext}_{\Lambda}^{i}({\begin{pmatrix}\begin{smallmatrix} 0 \\ \cok{f} \end{smallmatrix}\end{pmatrix}}, {\begin{pmatrix}\begin{smallmatrix} I \\ 0 \end{smallmatrix}\end{pmatrix}})\ar[r]  &
{\rm Ext}_{\Lambda}^{i}({\begin{pmatrix}\begin{smallmatrix} 0 \\ \cok{f} \end{smallmatrix}\end{pmatrix}}, {\begin{pmatrix}\begin{smallmatrix} \Ker d^1 \\ 0 \end{smallmatrix}\end{pmatrix}})\ar[r]^-\cong
& {\rm Ext}_{\Lambda}^{i+1}({\begin{pmatrix}\begin{smallmatrix} 0 \\ \cok{f} \end{smallmatrix}\end{pmatrix}}, {\begin{pmatrix}\begin{smallmatrix} U \\ 0 \end{smallmatrix}\end{pmatrix}})\ar[r]
&{\rm Ext}_{\Lambda}^{i+1}({\begin{pmatrix}\begin{smallmatrix} 0 \\ \cok{f} \end{smallmatrix}\end{pmatrix}}, {\begin{pmatrix}\begin{smallmatrix} I \\ 0 \end{smallmatrix}\end{pmatrix}})}
\]}

\vskip10pt

\noindent  By Lemmas \ref{im}(1)  and  \ref{iso}(1) one has ${\rm Ext}_{\Lambda}^{i-1}({\begin{pmatrix}\begin{smallmatrix} X \\ \im f \end{smallmatrix}\end{pmatrix}}, {\begin{pmatrix}\begin{smallmatrix} I \\ 0 \end{smallmatrix}\end{pmatrix}}) = 0$
and ${\rm Ext}_{\Lambda}^{i}({\begin{pmatrix}\begin{smallmatrix} 0 \\ \cok{f} \end{smallmatrix}\end{pmatrix}}, {\begin{pmatrix}\begin{smallmatrix} I \\ 0 \end{smallmatrix}\end{pmatrix}}) =0$ for $i\ge 2$.
Then one has the commutative diagram (rotating the diagram)

\vskip5pt

\[\xymatrix{
{\rm Ext}_{\Lambda}^{i}({\begin{pmatrix}\begin{smallmatrix} X \\ \im{f} \end{smallmatrix}\end{pmatrix}_{p, 0}}, {\begin{pmatrix}\begin{smallmatrix} U \\ 0 \end{smallmatrix}\end{pmatrix}})\ar[r]^-{\partial}    &   {\rm Ext}_{\Lambda}^{i+1}({\begin{pmatrix}\begin{smallmatrix} 0 \\ \cok{f} \end{smallmatrix}\end{pmatrix}_{0, 0}}, {\begin{pmatrix}\begin{smallmatrix} U \\ 0 \end{smallmatrix}\end{pmatrix}})\\
{\rm Ext}_{\Lambda}^{i-1}({\begin{pmatrix}\begin{smallmatrix} X \\ \im{f} \end{smallmatrix}\end{pmatrix}_{p, 0}}, {\begin{pmatrix}\begin{smallmatrix} \Ker d^1 \\ 0 \end{smallmatrix}\end{pmatrix}})\ar[r]^-{\partial} \ar[u]^{\iso} &
 {\rm Ext}_{\Lambda}^{i}({\begin{pmatrix}\begin{smallmatrix} 0 \\ \cok{f} \end{smallmatrix}\end{pmatrix}}, {\begin{pmatrix}\begin{smallmatrix} \Ker d^1 \\ 0 \end{smallmatrix}\end{pmatrix}})\ar[u]^{\iso}}
\]

\vskip5pt
\noindent Then one get the following diagram

\vskip5pt

\[\xymatrix@C=2cm{
{\rm Ext}_{\Lambda}^{i}({\begin{pmatrix}\begin{smallmatrix} X \\ \im{f} \end{smallmatrix}\end{pmatrix}_{p, 0}}, {\begin{pmatrix}\begin{smallmatrix} U \\ 0 \end{smallmatrix}\end{pmatrix}})\ar[r]^-{\partial}
&   {\rm Ext}_{\Lambda}^{i+1}({\begin{pmatrix}\begin{smallmatrix} 0 \\ \cok{f} \end{smallmatrix}\end{pmatrix}_{0, 0}}, {\begin{pmatrix}\begin{smallmatrix} U \\ 0 \end{smallmatrix}\end{pmatrix}})\\
{\rm Ext}_{\Lambda}^{i-1}({\begin{pmatrix}\begin{smallmatrix} X \\ \im{f} \end{smallmatrix}\end{pmatrix}_{p, 0}}, {\begin{pmatrix}\begin{smallmatrix} \Ker d^1 \\ 0 \end{smallmatrix}\end{pmatrix}})\ar[r]^-{\partial} \ar[u]^{\iso} \ar[d]_{\iso}&
 {\rm Ext}_{\Lambda}^{i}({\begin{pmatrix}\begin{smallmatrix} 0 \\ \cok{f} \end{smallmatrix}\end{pmatrix}}, {\begin{pmatrix}\begin{smallmatrix}  \Ker d^1 \\ 0 \end{smallmatrix}\end{pmatrix}})\ar[u]^{\iso}\ar[d]_{\iso}
 \\
{\rm Ext}_{A}^{i-1}(X, \Ker d^1)\ar[r]^-{{\rm Ext}_{A}^{i-1}(g, \Ker d^1)}  \ar[d]_{\iso} &  {\rm Ext}_{A}^{i-1}(N\otimes_BY, \Ker d^1) \ar[d]_{\iso} \\
{\rm Ext}_{A}^{i}(X, U)\ar[r]^-{{\rm Ext}_{A}^{i}(g, U)}  &  {\rm Ext}_{A}^{i}(N\otimes_BY, U)}\]

\vskip5pt
\noindent The first square is already proved commutative; the second square commutes by induction; the third square also commutes,
since  both the vertical isomorphisms are obtained by dimension shift.
All together, the first and the fourth row give rise to the desired commutative diagram.

\vskip10pt The assertion (2) can be similarly proved, by using Lemma \ref{baersum}(2), Lemma \ref{im}(3) and (4), and Lemma \ref{iso2}.
\end{proof}

\begin{prop}\label{iff} Let $\Lambda=\begin{pmatrix}\begin{smallmatrix}
A & N \\
M & B \\
\end{smallmatrix}\end{pmatrix}$ be a Morita ring, and $\begin{pmatrix}\begin{smallmatrix}X\\Y\end{smallmatrix}\end{pmatrix}_{f, g}$ a left $\Lambda$-module.
Assume that $M_{A}$ and $N_{B}$ are projective modules. Then

\vskip 5pt

{\rm (1)} \   Assume that  $N\otimes_{B}M = 0$. Then $\begin{pmatrix}\begin{smallmatrix}X\\Y\end{smallmatrix}\end{pmatrix}_{f, g}\in \ {^{\perp}\begin{pmatrix}\begin{smallmatrix}A\\0\end{smallmatrix}\end{pmatrix}_{0, 0}}$
if and only if \ ${\rm Hom}_A(g, A):{\rm Hom}_{A}(X, A)\lxr {\rm Hom}_{A}(N\otimes_{B}Y, A)$ is an epimorphism and \ ${\rm Ext}_{A}^{i}(g, A): {\rm Ext}_{A}^{i}(X, A)\longrightarrow {\rm Ext}_{A}^{i}(N\otimes_{B}Y, A)$ is an isomorphism for $i\geq 1$.

\vskip 5pt

{\rm (2)} \  Assume that $M\otimes_{A}N=0$. Then  $\begin{pmatrix}\begin{smallmatrix}X\\Y\end{smallmatrix}\end{pmatrix}_{f, g}\in \ {^{\perp}\begin{pmatrix}\begin{smallmatrix}0\\B\end{smallmatrix}\end{pmatrix}_{0, 0}}$ if and only if
\ ${\rm Hom}_B(f, B):{\rm Hom}_{B}(Y, B)\lxr {\rm Hom}_B(M\otimes_AX, B)$ is an epimorphism and \ ${\rm Ext}_{B}^{i}(f, B):{\rm Ext}_{B}^{i}(Y, B)\iso {\rm Ext}_{B}^{i}(M\otimes_{A}X, B)$ is an isomorphism for $i\geq 1$.
\end{prop}
\begin{proof}\ (1) \ The exact sequence of left $B$-modules
\[
\xymatrix@C=0.6cm@R=0.5cm{
0\ar[r]  &   \Ker{f}\ar[r]^-{\sigma}   &    M\otimes_{A}X \ar[rr]^-{f}    &&  Y \ar[r]^-{\pi} \ar@{<-_{)}}[dl]_{\iota} &   \cok{f} \ar[r]   &    0\\
 &&& \im{f} \ar@{<<-}[ul]_{p}  &&& \\
}
\]
induces the exact sequence of left $\Lambda$-modules:

$$
0\lxr {\begin{pmatrix}\begin{smallmatrix} X \\ \im{f} \end{smallmatrix}\end{pmatrix}_{p, 0}}
\stackrel{\begin{pmatrix}\begin{smallmatrix} 1 \\ \iota \end{smallmatrix}\end{pmatrix}}\lxr {\begin{pmatrix}\begin{smallmatrix} X \\ Y \end{smallmatrix}\end{pmatrix}_{f, g}}
\stackrel{\begin{pmatrix}\begin{smallmatrix} 0 \\ \pi \end{smallmatrix}\end{pmatrix}}\lxr  {\begin{pmatrix}\begin{smallmatrix} 0 \\ \cok{f} \end{smallmatrix}\end{pmatrix}_{0, 0}}  \lxr 0.
$$

\vskip5pt

\noindent By Lemma \ref{im}(2) one has
$
{\rm Ext}_{\Lambda}^{i}({\begin{pmatrix}\begin{smallmatrix} X \\ \im{f} \end{smallmatrix}\end{pmatrix}_{p, 0}},\ {\begin{pmatrix}\begin{smallmatrix} A \\ 0 \end{smallmatrix}\end{pmatrix}_{0, 0}}) \iso {\rm Ext}_{A}^{i}(X, A)
$ for $i\geq 0$.

\vskip 5pt

Since $N\otimes_{B}M=0$,  $N\otimes_{B}{\cok}f\iso N\otimes_{B}Y$.
By Lemma \ref{iso}(2) one has
${\rm Ext}_{\Lambda}^{i+1}({\begin{pmatrix}\begin{smallmatrix} 0 \\ \cok{f} \end{smallmatrix}\end{pmatrix}_{0, 0}},\ {\begin{pmatrix}\begin{smallmatrix} A \\ 0 \end{smallmatrix}\end{pmatrix}_{0, 0}}) \iso {\rm Ext}_{A}^{i}(N\otimes_BY, A)
$ for $i\geq 0$. Applying ${\rm Hom}_{\Lambda}(-, {\begin{pmatrix}\begin{smallmatrix} A \\ 0 \end{smallmatrix}\end{pmatrix}_{0, 0}})$  we get the diagram with exact row:

{\footnotesize
\[\xymatrix@C=12pt@R=20pt{
{\rm Hom}_{\Lambda}({\begin{pmatrix}\begin{smallmatrix} X \\ \im{f} \end{smallmatrix}\end{pmatrix}_{p, 0}},\ {\begin{pmatrix}\begin{smallmatrix} A \\ 0 \end{smallmatrix}\end{pmatrix}_{0, 0}})\ar[r] \ar[d]_-{\iso}   &   {\rm Ext}_{\Lambda}^{1}({\begin{pmatrix}\begin{smallmatrix} 0 \\ \cok{f} \end{smallmatrix}\end{pmatrix}_{0, 0}},\ {\begin{pmatrix}\begin{smallmatrix} A \\ 0 \end{smallmatrix}\end{pmatrix}_{0, 0}})\ar[r] \ar[d]^{\iso}  &   {\rm Ext}_{\Lambda}^{1}({\begin{pmatrix}\begin{smallmatrix} X \\ Y \end{smallmatrix}\end{pmatrix}_{f, g}},\ {\begin{pmatrix}\begin{smallmatrix} A \\ 0 \end{smallmatrix}\end{pmatrix}_{0, 0}}) \\
{\rm Hom}_{A}(X, A)\ar[r]^-{{\rm Hom}(g, A)}   &   {\rm Hom}_{A}(N\otimes_{B}Y, A)   &  }
\]
}

\vskip5pt \noindent and the diagram with exact row for $i\geq1$

{\tiny
\[\xymatrix@C=10pt@R=20pt{
{\rm Ext}_{\Lambda}^{i}\left({\begin{pmatrix}\begin{smallmatrix} X \\ Y \end{smallmatrix}\end{pmatrix}_{f, g}}, {\begin{pmatrix}\begin{smallmatrix} A \\ 0 \end{smallmatrix}\end{pmatrix}}\right) \ar[r]  &  {\rm Ext}_{\Lambda}^{i}\left({\begin{pmatrix}\begin{smallmatrix} X \\ \im{f} \end{smallmatrix}\end{pmatrix}_{p, 0}}, {\begin{pmatrix}\begin{smallmatrix} A \\ 0 \end{smallmatrix}\end{pmatrix}}\right)\ar[r] \ar[d]_-{\iso}   &   {\rm Ext}_{\Lambda}^{i+1}\left({\begin{pmatrix}\begin{smallmatrix} 0 \\ \cok{f} \end{smallmatrix}\end{pmatrix}}, {\begin{pmatrix}\begin{smallmatrix} A \\ 0 \end{smallmatrix}\end{pmatrix}}\right)\ar[r] \ar[d]^{\iso}  &   {\rm Ext}_{\Lambda}^{i+1}\left({\begin{pmatrix}\begin{smallmatrix} X \\ Y \end{smallmatrix}\end{pmatrix}_{f, g}}, {\begin{pmatrix}\begin{smallmatrix} A \\ 0 \end{smallmatrix}\end{pmatrix}}\right) \\
& {\rm Ext}_{A}^{i}(X, A)\ar[r]^-{{\rm Ext}^{i}(g, A)}   &   {\rm Ext}_{A}^{i}(N\otimes_{B}Y, A).   &  }
\]}

\vskip5pt

\noindent The commutativity of the two squares follows from Lemma \ref{induction}(1). From these we see that $\begin{pmatrix}\begin{smallmatrix}X\\Y\end{smallmatrix}\end{pmatrix}_{f, g}\in {^{\perp}\begin{pmatrix}\begin{smallmatrix}A\\0\end{smallmatrix}\end{pmatrix}_{0, 0}}$ if and only if \ ${\rm Hom}(g, A):{\rm Hom}_{A}(X, A)\lxr {\rm Hom}_{A}(N\otimes_{B}Y, A)$ is an epimorphism and ${\rm Ext}^i_A(g, A): \ {\rm Ext}_{A}^{i}(X, A)\iso {\rm Ext}_{A}^{i}(N\otimes_{B}Y, A)$ is an isomorphism for $i\geq 1$. This completes the proof of (1).

\vskip5pt

The assertion (2) can be similarly proved, by using Lemmas \ref{im}(4), \ref{iso2}(2), and  \ref{induction}(2). \end{proof}

\subsection{Semi-Gorenstein-projective modules} Now we are in position to prove the main result of this section.

\begin{thm}\label{semi} \ Let $\Lambda=\begin{pmatrix}\begin{smallmatrix}
A & N \\
M & B \\
\end{smallmatrix}\end{pmatrix}$ be a Morita ring with $M\otimes_{A}N=0=N\otimes_{B}M$. Assume that \ $_BM, \ _AN$, \ $M_{A}$ and $N_{B}$ are projective modules. Then $\begin{pmatrix}\begin{smallmatrix}X\\Y\end{smallmatrix}\end{pmatrix}_{f, g}$ is a semi-Gorenstein-projective $\Lambda\mbox{-}$module if and only if the following conditions are satisfied$:$

\vskip 5pt

{\rm (1)} \  ${\rm Hom}_A(g, A):{\rm Hom}_{A}(X, A)\lxr {\rm Hom}_{A}(N\otimes_{B}Y, A)$ is an epimorphism$;$

\vskip 5pt

{\rm (2)} \ ${\rm Ext}_{A}^{i}(g, A): {\rm Ext}_{A}^{i}(X, A)\longrightarrow {\rm Ext}_{A}^{i}(N\otimes_{B}Y, A)$ is an isomorphism for $i\geq 1;$

\vskip 5pt

{\rm (3)} \  ${\rm Hom}_B(f, B):{\rm Hom}_{B}(Y, B)\lxr {\rm Hom}_B(M\otimes_AX, B)$ is an epimorphism$;$

\vskip 5pt

{\rm (4)} \ \  ${\rm Ext}_{B}^{i}(f, B):{\rm Ext}_{B}^{i}(Y, B)\longrightarrow {\rm Ext}_{B}^{i}(M\otimes_{A}X, B)$ is an isomorphism for $i\geq 1$.

\end{thm}
\begin{proof}\ Consider exact sequences of left $\Lambda$-modules:
\[
0\lxr \begin{pmatrix}\begin{smallmatrix}0 \\ M \end{smallmatrix}\end{pmatrix}_{0,0} \lxr \begin{pmatrix}\begin{smallmatrix}A \\ M \end{smallmatrix}\end{pmatrix}_{\id_{M}, 0}\lxr \begin{pmatrix}\begin{smallmatrix}A \\ 0 \end{smallmatrix}\end{pmatrix}_{0,0} \lxr 0
\]
and
\[
0\lxr \begin{pmatrix}\begin{smallmatrix} N \\ 0 \end{smallmatrix}\end{pmatrix}_{0,0}\lxr \begin{pmatrix}\begin{smallmatrix} N \\ B \end{smallmatrix}\end{pmatrix}_{0, \id_{N}}\lxr \begin{pmatrix}\begin{smallmatrix}0 \\ B \end{smallmatrix}\end{pmatrix}_{0,0}\lxr 0.
\]

\vskip 10pt

$\Longleftarrow:$  \ \ By Proposition \ref{iff} one has

$$\begin{pmatrix}\begin{smallmatrix}X\\Y\end{smallmatrix}\end{pmatrix}_{f, g}\in \ {^{\perp}\begin{pmatrix}\begin{smallmatrix}A\\0\end{smallmatrix}\end{pmatrix}_{0, 0}} \ \ \ \mbox{and}
\ \ \begin{pmatrix}\begin{smallmatrix}X\\Y\end{smallmatrix}\end{pmatrix}_{f, g}\in \  {^{\perp}\begin{pmatrix}\begin{smallmatrix}0\\B\end{smallmatrix}\end{pmatrix}_{0, 0}}.$$

\vskip5pt

\noindent Since $_{A}N$ and $_{B}M$ are projective, it follows that

$$\begin{pmatrix}\begin{smallmatrix}N\\0\end{smallmatrix}\end{pmatrix}_{0, 0}\in \ {\rm add} \begin{pmatrix}\begin{smallmatrix}A\\0\end{smallmatrix}\end{pmatrix}_{0, 0} \ \ \
\mbox{and} \ \ \begin{pmatrix}\begin{smallmatrix}0\\M\end{smallmatrix}\end{pmatrix}_{0, 0}\in \ {\rm add} \begin{pmatrix}\begin{smallmatrix}0\\B\end{smallmatrix}\end{pmatrix}_{0, 0}$$

\vskip5pt \noindent and hence $$\begin{pmatrix}\begin{smallmatrix}X\\Y\end{smallmatrix}\end{pmatrix}_{f, g}\in \ {^{\perp}\begin{pmatrix}\begin{smallmatrix}N\\0\end{smallmatrix}\end{pmatrix}_{0, 0}} \ \ \ \mbox{and} \ \ \begin{pmatrix}\begin{smallmatrix}X\\Y\end{smallmatrix}\end{pmatrix}_{f, g}\in \ {^{\perp}\begin{pmatrix}\begin{smallmatrix}0\\M\end{smallmatrix}\end{pmatrix}_{0, 0}}.$$

\vskip5pt \noindent Hence $\begin{pmatrix}\begin{smallmatrix}X\\Y\end{smallmatrix}\end{pmatrix}_{f, g}\in \ {^{\perp}\begin{pmatrix}\begin{smallmatrix}A\\M\end{smallmatrix}\end{pmatrix}_{\id_{M}, 0}} \ \ \mbox{and} \
\begin{pmatrix}\begin{smallmatrix}X\\Y\end{smallmatrix}\end{pmatrix}_{f, g}\in \ {^{\perp}\begin{pmatrix}\begin{smallmatrix}N\\B\end{smallmatrix}\end{pmatrix}_{0, \id_{N}}}.$
That is,  $\begin{pmatrix}\begin{smallmatrix}X\\Y\end{smallmatrix}\end{pmatrix}_{f, g}$ is a semi-Gorenstein-projective $\Lambda$-module.

\vskip 10pt

$\Longrightarrow:$ \ Conversely, if $\begin{pmatrix}\begin{smallmatrix}X\\Y\end{smallmatrix}\end{pmatrix}_{f, g}$ is a semi-Gorenstein-projective $\Lambda$-module,
i.e.,

$$\begin{pmatrix}\begin{smallmatrix}X\\Y\end{smallmatrix}\end{pmatrix}_{f, g}\in \ {^{\perp}\begin{pmatrix}\begin{smallmatrix}A\\M\end{smallmatrix}\end{pmatrix}_{\id_{M}, 0}} \ \ \
\mbox{and} \ \ \begin{pmatrix}\begin{smallmatrix}X\\Y\end{smallmatrix}\end{pmatrix}_{f, g}\in \ {^{\perp}\begin{pmatrix}\begin{smallmatrix}N\\B\end{smallmatrix}\end{pmatrix}_{0, \id_{N}}}.$$

\vskip5pt
\noindent Since $M\otimes_AN = 0 = N\otimes_BM$ and $_{A}N$ and $_{B}M$ are projective, one sees that

\[
\begin{pmatrix}\begin{smallmatrix}N\\0\end{smallmatrix}\end{pmatrix}_{0, 0}=\begin{pmatrix}\begin{smallmatrix}N\\M\otimes_{A}N\end{smallmatrix}\end{pmatrix}_{0, 0}\ \ \ \mbox{and}\ \ \
\begin{pmatrix}\begin{smallmatrix}0\\M\end{smallmatrix}\end{pmatrix}_{0, 0}=
\begin{pmatrix}\begin{smallmatrix}N\otimes_{B}M\\M\end{smallmatrix}\end{pmatrix}_{0, 0}
\]
\vskip5pt
\noindent are projective left $\Lambda$-modules. Thus

$$\begin{pmatrix}\begin{smallmatrix}X\\Y\end{smallmatrix}\end{pmatrix}_{f, g}\in \ {^{\perp}\begin{pmatrix}\begin{smallmatrix}N\\0\end{smallmatrix}\end{pmatrix}_{0, 0}} \ \ \
\mbox{and} \ \ \ \begin{pmatrix}\begin{smallmatrix}X\\Y\end{smallmatrix}\end{pmatrix}_{f, g}\in \ {^{\perp}\begin{pmatrix}\begin{smallmatrix}0\\M\end{smallmatrix}\end{pmatrix}_{0, 0}}$$

\vskip5pt
\noindent and hence

$$\begin{pmatrix}\begin{smallmatrix}X\\Y\end{smallmatrix}\end{pmatrix}_{f, g}\in \ {^{\perp}\begin{pmatrix}\begin{smallmatrix}A\\0\end{smallmatrix}\end{pmatrix}_{0, 0}} \ \ \
\mbox{and} \ \ \ \begin{pmatrix}\begin{smallmatrix}X\\Y\end{smallmatrix}\end{pmatrix}_{f, g}\in {^{\perp}\begin{pmatrix}\begin{smallmatrix}0\\B\end{smallmatrix}\end{pmatrix}_{0, 0}}.$$
\vskip5pt
\noindent Then the assertion follows from Proposition \ref{iff}. \end{proof}

Taking $M = 0$ or $N = 0$ in Theorem \ref{semi} one gets the following consequence, which is known in \cite[Theorem 1.1]{XZ}.

\begin{cor}\label{triangsemi} \ {\rm (\cite{XZ})} \ $(1)$ \ Let $\Lambda=\begin{pmatrix}\begin{smallmatrix}
A & N \\
0 & B \\
\end{smallmatrix}\end{pmatrix}$. Assume that \ $_AN$ and $N_{B}$ are projective modules.
Then $\begin{pmatrix}\begin{smallmatrix}X\\Y\end{smallmatrix}\end{pmatrix}_{0, g}\in \ ^\perp \Lambda$ if and only if
${\rm Hom}_A(g, A):{\rm Hom}_{A}(X, A)\lxr {\rm Hom}_{A}(N\otimes_{B}Y, A)$ is an epimorphism,
${\rm Ext}_{A}^{i}(g, A): {\rm Ext}_{A}^{i}(X, A)\longrightarrow {\rm Ext}_{A}^{i}(N\otimes_{B}Y, A)$ is an isomorphism for $i\geq 1,$ and $Y\in \ ^\perp B$.

\vskip5pt

$(2)$ \ Let $\Lambda=\begin{pmatrix}\begin{smallmatrix}
A & 0 \\
M & B \\
\end{smallmatrix}\end{pmatrix}$. Assume that \ $_BM$ and $M_{A}$ are projective modules.
Then $\begin{pmatrix}\begin{smallmatrix}X\\Y\end{smallmatrix}\end{pmatrix}_{f, 0}\in \ ^\perp \Lambda$
if and only if $X\in \ ^\perp A$,  ${\rm Hom}_B(f, B):{\rm Hom}_{B}(Y, B)\lxr {\rm Hom}_B(M\otimes_AX, B)$ is an epimorphism, and
${\rm Ext}_{B}^{i}(f, B):{\rm Ext}_{B}^{i}(Y, B)\longrightarrow {\rm Ext}_{B}^{i}(M\otimes_{A}X, B)$ is an isomorphism for $i\geq 1$.
\end{cor}

\section{\bf Weakly Gorensteinness of the Morita rings in Theorem \ref{semi}}

The aim of this section is to prove the following result, which answers the question when the algebras in Theorem \ref{semi} are left weakly Gorenstein.

\vskip5pt

\begin{thm}\label{weakGmorita} \ Let $\Lambda=\begin{pmatrix}\begin{smallmatrix}
A & N \\
M & B \\
\end{smallmatrix}\end{pmatrix}$  with $M\otimes_{A}N=0=N\otimes_{B}M$. Assume that \ $_BM, \ _AN$, \ $M_{A}$ and $N_{B}$ are projective modules.
Then $\Lambda$ is left weakly-Gorenstein if and only if  so are $A$ and $B$.
\end{thm}

\subsection{Gorenstein-projective modules over a class of Morita rings}

\begin{defn} \label{compatible} \ Let $A$ and $B$ be Artin algebras.

\vskip 5pt

$(1)$ \ \ A left module $_{A}L$ is ${\rm Hom}$-compatible, provided that ${\rm Hom}_{A}(P^{\bullet}, L)$ is acyclic for each complete projective resolution $P^{\bullet}$ in $A\mbox{-}{\rm mod}$.

\vskip 5pt

$(2)$ \ \  A right module $L_{B}$ is tensor-compatible, provided that $L\otimes_{B}Q^{\bullet}$ is acyclic for each complete projective resolution $P^{\bullet}$ in $B\mbox{-}{\rm mod}$.

\vskip 5pt

$(3)$ \ \ A bomodule $_{A}L_{B}$ is compatible, provided that the left module $_{A}L$ is ${\rm Hom}$-compatible and the right module $L_{B}$ is tensor-compatible.
\end{defn}

This notion of a compatible bimodule has been introduced in \cite [Definition 1.1]{Zh13}, to describe Gorenstein-projective modules over triangular matrix rings. See also \cite{GP}
and \cite {GX24}. Note that if ${\rm proj.dim} _AL<\infty$ and ${\rm proj.dim} L_B<\infty$, then $_{A}L_{B}$ is a compatible bimodule (\cite[Proposition 1.3]{Zh13}).

\vskip 5pt

A left module $\begin{pmatrix}\begin{smallmatrix}X \\ Y\end{smallmatrix}\end{pmatrix}_{f, g}$ over Morita ring $\Lambda=\begin{pmatrix}\begin{smallmatrix}
A & N \\
M & B \\
\end{smallmatrix}\end{pmatrix}$ will be said to be {\it monic}, provided that the structure maps
$f:M\otimes_{A}X\lxr Y$ and
$g:N\otimes_{B}Y\lxr X$  are injective maps.

\vskip5pt

\begin{thm}\label{Gproj} \ {\rm (\cite[Theorem 3.10]{GP},  \cite[Proposition 3.14]{GX24})} \ \  Let  $\Lambda=\begin{pmatrix}\begin{smallmatrix}
A & N \\
M & B \\
\end{smallmatrix}\end{pmatrix}$  with $M\otimes_{A}N=0=N\otimes_{B}M$. Assume that  $_{B}M_{A}$ and $_{A}N_{B}$ are compatible bimodules. Then

\vskip5pt

$(1)$ \ If $\begin{pmatrix}\begin{smallmatrix}X\\Y\end{smallmatrix}\end{pmatrix}_{f, g}$ is monic with $\cok f\in \Gproj(B)$ and $\cok g\in \Gp(A)$, then
$\begin{pmatrix}\begin{smallmatrix}X\\Y\end{smallmatrix}\end{pmatrix}_{f, g}\in \Gproj(\Lambda)$.

\vskip5pt

$(2)$ \ Conversely, assume that the right $\Lambda$-module $(M, 0)_{0,0}$ is tensor-compatible
and $\begin{pmatrix}\begin{smallmatrix}N\\0\end{smallmatrix}\end{pmatrix}_{0,0}$ is ${\rm Hom}$-compatible.
If $\begin{pmatrix}\begin{smallmatrix}X\\Y\end{smallmatrix}\end{pmatrix}_{f, g}\in \Gproj(\Lambda)$,
then it is monic with $\cok f\in \Gproj(B)$ and $\cok g\in \Gproj(A)$.
\end{thm}
\begin{proof} \ The assertion (1) is a special case of \cite[Theorem 3.10]{GP}, although the notion of a compatible bimodule used in \cite{GP} is
not precisely the same as in Definition \ref{compatible}(2), the proof of \cite[Theorem 3.10]{GP} only uses Definition \ref{compatible}(2).
The assertion (2) follows from  \cite[Proposition 3.14]{GX24}. (However the statement of Theorem \ref{Gproj} is not available either in \cite {GP} or in \cite{GX24}.) \end{proof}

We stress that if $M$ or $N$ is not a compatible bimodule then Theorem \ref{Gproj} is no longer true, even in the special case of triangular matrix algebra (see \cite [Remark 1.5]{Zh13}).

\begin{cor}\label{gpmorita}
\ Let $\Lambda=\begin{pmatrix}\begin{smallmatrix}
A & N \\
M & B \\
\end{smallmatrix}\end{pmatrix}$ with $M\otimes_{A}N=0=N\otimes_{B}M$. Assume that \ $_BM, \ _AN$, \ $M_{A}$ and $N_{B}$ are projective modules.
Then $\begin{pmatrix}\begin{smallmatrix}X\\Y\end{smallmatrix}\end{pmatrix}_{f, g}\in \Gproj(\Lambda)$ if and only if it is a monic $\Lambda\mbox{-}$module
with  $\cok f\in \Gproj(B)$ and $\cok g\in \Gproj(A)$.
\end{cor}
\begin{proof} \ Since $_BM, \ _AN$, \ $M_{A}$ and $N_{B}$ are projective modules,   $_{B}M_{A}$ and $_{A}N_{B}$ are compatible bimodules.
Since $M\otimes_{A}N=0=N\otimes_{B}M$, the right $\Lambda$-module $(M, 0)_{0,0} = (M, M\otimes_AN, 0, 0)$ is a projective module, and hence it is tensor-compatible.
Also, the left $\Lambda$-module $\begin{pmatrix}\begin{smallmatrix}N\\0\end{smallmatrix}\end{pmatrix}_{0,0} = \begin{pmatrix}\begin{smallmatrix} N\\ M\otimes_AN\end{smallmatrix}\end{pmatrix}_{0,0} $ is
a projective module, and hence it is ${\rm Hom}$-compatible. Thus,  the assertion follows from Theorem \ref{Gproj}.\end{proof}

An Artin  algebra $A$ is  GP-{\it free}, if any Gorenstein-projective $A$-module is projective; and  GP-{\it finite}, if $A$ admits only finitely many isomorphism classes of indecomposable Gorenstein-projective $A$-modules.
Since there is a duality $\Hom_A(-, A): \Gp(A) \longrightarrow \Gp(A^{\rm op})$, the notion of GP-free and GP-finite is left-right symmetric.
By Corollary \ref{gpmorita} one has the following consequences.

\vskip 5pt

\begin{cor}\label{finitepd} \ Let  $\Lambda=\begin{pmatrix}\begin{smallmatrix}
A & N \\
M & B \\
\end{smallmatrix}\end{pmatrix}$ with $M\otimes_{A}N=0=N\otimes_{B}M$. Assume that  $_{B}M_{A}$ and $_{A}N_{B}$ are compatible bimodules.

\vskip 5pt

{\rm (1)} \ \ If $\Lambda$ is {\rm GP}-free, then so are $A$ and $B$.

\vskip5pt

Conversely, assume that the right $\Lambda$-module $(M, 0)_{0,0}$ is tensor-compatible
and left $\Lambda$-module $\begin{pmatrix}\begin{smallmatrix}N\\0\end{smallmatrix}\end{pmatrix}_{0,0}$ is ${\rm Hom}$-compatible. If  $A$ and $B$ are {\rm GP}-free, then so is $\Lambda$.

\vskip 5pt

{\rm (2)} \ \ If $\Lambda$ is {\rm GP}-finite, then so are $A$ and $B$.

\vskip 5pt

{\rm (3)} \ \ Assume that $(M, 0)$ is tensor-compatible. If $\gdim{B}<\infty$ and $\pd{_{A}N}<\infty$,
then all the indecomposable Gorenstein-projective left $\Lambda$-modules are precisely
$\begin{pmatrix}\begin{smallmatrix}G\\M\otimes_{A}G\end{smallmatrix}\end{pmatrix}_{\id_{M\otimes_{A}G}, 0}, \ \ \begin{pmatrix}\begin{smallmatrix} N\otimes_BQ \\ Q\end{smallmatrix}\end{pmatrix}_{0, \id_{N\otimes_BQ}}$,
where $G$ runs over all the indecomposable Gorenstein-projective left $A$-modules, and
$Q$ runs over all the indecomposable projective left $B$-modules.

\vskip 5pt

{\rm (4)} \ \ Assume that $(M, 0)_{0,0}$ is tensor-compatible
and $\begin{pmatrix}\begin{smallmatrix}N\\0\end{smallmatrix}\end{pmatrix}_{0,0}$ is ${\rm Hom}$-compatible. If $\gdim{A}<\infty$ and $\pd{_{B}M}<\infty$, then all the indecomposable Gorenstein-projective left $\Lambda$-modules are precisely
$\begin{pmatrix}\begin{smallmatrix}P\\M\otimes_AP\end{smallmatrix}\end{pmatrix}_{\id_{M\otimes_AP}, 0},   \begin{pmatrix}\begin{smallmatrix} N\otimes_BE \\ E\end{smallmatrix}\end{pmatrix}_{0, \id_{N\otimes_BE}},$
where $P$ runs over all the indecomposable projective left $A$-modules, and
$E$ runs over all the indecomposable Gorenstein-projective left $B$-modules.
\end{cor}
\begin{proof}  \ (1) \ \ Assume that $\Lambda$ is GP-free. Let $X\in \Gp(A)$.
Then $\begin{pmatrix}\begin{smallmatrix}X\\M\otimes_{A}X\end{smallmatrix}\end{pmatrix}_{\id_{M\otimes_{A}X}, 0}\in \Gp(\Lambda)$, by Theorem \ref{Gproj}(1).
Thus, by assumption $\begin{pmatrix}\begin{smallmatrix}X\\M\otimes_{A}X\end{smallmatrix}\end{pmatrix}_{\id_{M\otimes_{A}X}, 0}$ is a projective $\Lambda$-module.
Using Lemma \ref{projectives} one easily deduces that $X$  is a projective $A$-module. So $A$ is GP-free. Similarly, one can see $B$ is GP-free, by using
$\begin{pmatrix}\begin{smallmatrix}N\otimes_{B}Y\\Y\end{smallmatrix}\end{pmatrix}_{0, \id_{N\otimes_{B}Y}}$.

\vskip 5pt

Conversely, assume that $(M, 0)_{0,0}$ is tensor-compatible
and  $\begin{pmatrix}\begin{smallmatrix}N\\0\end{smallmatrix}\end{pmatrix}_{0,0}$ is ${\rm Hom}$-compatible,
and that $A$ and $B$ are GP-free. Let  $\begin{pmatrix}\begin{smallmatrix}X\\Y\end{smallmatrix}\end{pmatrix}_{f, g}\in \Gproj(\Lambda)$. Then by Theorem \ref{Gproj}(2)
one has an exact sequence of $B$-modules
$$0\lxr M\otimes_{A}X\stackrel {f}{\lxr} Y\stackrel{\pi_{f}}{\lxr} \cok f\lxr 0$$
with $\cok f\in \Gproj(B)$, and an exact sequence of $A$-modules
$$0\lxr N\otimes_{B} Y\stackrel{g}{\lxr} X\stackrel{\pi_{g}}{\lxr} \cok g\lxr 0$$
with $\cok g\in \Gproj(A)$. Since $B$ is GP-free,  $\cok{f}=Q$ is a projective left $B$-module. Thus $Y\cong Q\oplus (M\otimes_{A}X)$.
Similarly, since $A$ is GP-free, $\cok{g}=P$ is a projective left $A$-module. Thus $X \cong  P\oplus (N\otimes_{B}Y)$.
Since $M\otimes_A N = 0 = N\otimes_BM$, it follows that
$$X \cong  P\oplus (N\otimes_{B}Q), \ \ \ \ Y\cong Q\oplus (M\otimes_{A}P).$$
Therefore
$$\begin{aligned}
\begin{pmatrix}\begin{smallmatrix}X\\Y\end{smallmatrix}\end{pmatrix}_{f, g} \iso \begin{pmatrix}\begin{smallmatrix}P\\M\otimes_{A}P\end{smallmatrix}\end{pmatrix}_{\id_{M\otimes_{A}P}, 0}\oplus \begin{pmatrix}\begin{smallmatrix}N\otimes_{B}Q\\Q\end{smallmatrix}\end{pmatrix}_{0, \id_{N\otimes_{B}Q}}
\end{aligned}$$
is a projective $\Lambda$-module. Thus $\Lambda$ is GP-free.

\vskip 5pt

(2) \ \ If $X$ is an indecomposable Gorenstein-projective $A$-module,
then $\begin{pmatrix}\begin{smallmatrix}X\\M\otimes_{A}X\end{smallmatrix}\end{pmatrix}_{\id_{M\otimes_{A}X}, 0}\in \Gp(\Lambda)$ by Theorem \ref{Gproj}(1).
Also, since ${\rm End}_\Lambda(\begin{pmatrix}\begin{smallmatrix}X\\M\otimes_{A}X\end{smallmatrix}\end{pmatrix}_{\id_{M\otimes_{A}X}, 0}) \cong {\rm End}_A(X)$ is local, $\begin{pmatrix}\begin{smallmatrix}X\\M\otimes_{A}X\end{smallmatrix}\end{pmatrix}_{\id_{M\otimes_{A}X}, 0}$ is indecomposable as a $\Lambda$-module.
Thus, if $\Lambda$ is GP-finite, then $A$ is GP-finite.

\vskip5pt

Similarly, if $Y\in \Gproj(B)$ is indecomposable, then $\begin{pmatrix}\begin{smallmatrix}N\otimes_{B}Y\\Y\end{smallmatrix}\end{pmatrix}_{0, \id_{N\otimes_{B}Y}}\in\Gp(B)$ is indecomposable.
Thus if $\Lambda$ is GP-finite, then $B$ is GP-finite.

\vskip 5pt

(3) \ \ Let $\begin{pmatrix}\begin{smallmatrix}X\\Y\end{smallmatrix}\end{pmatrix}_{f, g}\in \Gproj(\Lambda)$. Since $M\otimes_AN = 0$,
$\binom{N}{0} = \binom{N}{M\otimes_AN}$ is a projective $\Lambda$-module, and hence it is Hom-compatible. By Theorem \ref{Gproj}(2) one has an exact sequence of $B$-modules
$$0\lxr M\otimes_{A}X\stackrel {f}{\lxr} Y\stackrel{\pi_{f}}{\lxr} \cok f\lxr 0$$
with $\cok f\in \Gproj(B)$, and an exact sequence of $A$-modules
$$0\lxr N\otimes_{B} Y\stackrel{g}{\lxr} X\stackrel{\pi_{g}}{\lxr} \cok g\lxr 0$$
with $\cok g\in \Gproj(A)$.
Since $\gdim{B}<\infty$,  $\cok{f}=Q$ is a projective left $B$-module. Thus $Y\cong Q\oplus (M\otimes_{A}X)$.
Since $M\otimes_A N = 0 = N\otimes_BM$, it follows that $M\otimes_A X \cong M\otimes_A \cok g$ and $N\otimes_BY\cong N\otimes_B Q$, and hence $Y \cong Q\oplus (M\otimes_{A}\cok g)$.

\vskip 5pt

Since $\pd{_{A}N}<\infty$ and $_BQ$ is projective, one has $\pd_A(N\otimes_{B}Q) <\infty.$
Since $\cok g\in \Gproj(A)$, it follows that  ${\rm Ext}_{A}^{1}(\cok{g}, N\otimes_{B}Q)=0$, and hence
${\rm Ext}_{A}^{1}(\cok{g}, N\otimes_{B}Y)=0$.
Thus
$X \cong \cok{g}\oplus (N\otimes_{B}Y) \cong \cok{g}\oplus (N\otimes_{B}Q)$. Therefore
$$\begin{pmatrix}\begin{smallmatrix}X\\Y\end{smallmatrix}\end{pmatrix}_{f, g}  \iso \begin{pmatrix}\begin{smallmatrix}\cok{g}\oplus (N\otimes_{B}Q)\\(M\otimes_{A}\cok g)\oplus Q\end{smallmatrix}\end{pmatrix}_{\begin{pmatrix}\begin{smallmatrix}
1 & 0\\
0 & 0\\
\end{smallmatrix}\end{pmatrix}, \begin{pmatrix}\begin{smallmatrix}
0 & 0\\
0 & 1\\
\end{smallmatrix}\end{pmatrix}} \iso \begin{pmatrix}\begin{smallmatrix}\cok{g}\\M\otimes_{A}\cok g\end{smallmatrix}\end{pmatrix}_{\id_{M\otimes_{A}\cok g}, 0}\oplus \begin{pmatrix}\begin{smallmatrix}N\otimes_{B}Q\\Q\end{smallmatrix}\end{pmatrix}_{0, \id_{N\otimes_{B}Q}}.
$$
From this one sees the assertion $(3)$.

\vskip 5pt

The assertion (4) can be similarly proved.
\end{proof}

\subsection {A property of left weakly Gorenstein algebras}

\begin{thm}\label{wgmono}
Let $A$ be a left weakly Gorenstein algebra. If $\phi: X\longrightarrow Y$ is a left $A$-homomorphism such that $\phi^*=\Hom_A(\phi,A)$ is an epimorphism and
$\Ext^i_A(\phi,A)$ is an isomorphism for $i\geq 1$, then $\phi$ is a monomorphism and $\cok\phi\in \Gp(A)$.
\end{thm}
\begin{proof} \ Let $0\longrightarrow L\longrightarrow P\stackrel{\pi}\longrightarrow Y\longrightarrow 0$ be an exact sequence, where $\pi$ is the projective cover.
Consider the pull-back
$$\xymatrix{0\ar[r]&L\ar[r]\ar@{=}[d]&E\ar[r]\ar[d]_-{\psi}&X\ar[r]\ar[d]^\phi&0.\\
0\ar[r]&L\ar[r]&P\ar[r]^\pi&Y\ar[r]&0 }$$

\vskip5pt

\noindent Applying $\Hom_A(-,A)$  one gets a commutative diagram with exact rows (we write ${\rm E}$ for ${\rm Ext})$:

$$\xymatrix{Y^*\ar[r]\ar@{->>}[d]_{\phi^*}&P^*\ar[r]\ar[d]^{\psi^*}&L^*\ar[r]\ar@{=}[d]& {\rm E}^1(Y, A)\ar[r]\ar[d]_\cong^{{\rm E}^1(\phi,A)}&0\ar[r]\ar[d] &
{\rm E}^1(L, A)\ar[r]\ar@{=}[d] & \cdots \\
X^*\ar[r]&E^*\ar[r]&L^*\ar[r]&{\rm E}^1(X,A)\ar[r]^-0&{\rm E}^1(E, A)\ar[r] & {\rm E}^1(L, A)\ar[r] & \cdots }$$
where by assumption $\phi^*$  is an epimorphism, and the following commutative diagram with exact rows for any $t\ge 1:$

$$\xymatrix@C=0.4cm{\cdots \ar[r] & {\rm E}^t(Y,A)\ar[r]\ar[d]_\cong^-{{\rm E}^{t}(\phi,A)}&0\ar[r]\ar[d]&{\rm E}^t(L,A)\ar[r]^-\cong \ar@{=}[d]& {\rm E}^{t+1}(Y,A)\ar[r]\ar[d]_\cong^{{\rm E}^{t+1}(\phi,A)}& 0\ar[r]\ar[d]&{\rm E}^{t+1}(L,A) \ar@{=}[d]\ar[r]& \cdots \\
\cdots \ar[r] & {\rm E}^t(X,A)\ar[r]^-0& {\rm E}^t(E,A)\ar[r]^-0&{\rm E}^t(L,A)\ar[r]^-\cong & {\rm E}^{t+1}(X,A)\ar[r]^-0 & {\rm E}^{t+1}(E,A)\ar[r]^-0 & {\rm E}^{t+1}(L,A) \ar[r]& \cdots }$$
where by assumption $\Ext^t_A(\phi,A)$ is an isomorphism for $t\ge 1$.

\vskip5pt

Hence, it is easy to see that $\psi^*$ is an epimorphism and $E\in\leftidx{^\perp}A=\Gp(A)$, via the analysis of exactness.
Since $\Hom_A(-,A):  \Gp(A) \longrightarrow \Gp(A^{\rm op})$ is a duality, it follows that $\psi$ is a monomorphism.
Thus, $\phi$ is a monomorphism.

\vskip5pt

So one has an exact sequence $0\longrightarrow E\stackrel \psi\longrightarrow P\longrightarrow \cok \psi \longrightarrow 0$
with epimorphism $\psi^*$ and $E\in \Gp(A)$. It follows that $\cok \psi\in\leftidx{^\perp}A=\Gp(A)$, and hence $\cok\phi\cong\cok\psi$ is a Gorenstein-projective $A$-module.
\end{proof}

\subsection{Proof of Theorem \ref{weakGmorita}}

$\Longleftarrow$: \ Assume that $A$ and $B$ are left weakly-Gorenstein. Let $\begin{pmatrix}\begin{smallmatrix}X\\Y\end{smallmatrix}\end{pmatrix}_{f, g}\in \ ^\perp \Lambda$.
By Theorem \ref{semi}, the following conditions are satisfied:
\vskip5pt

(1) \  ${\rm Hom}_A(g, A):{\rm Hom}_{A}(X, A)\lxr {\rm Hom}_{A}(N\otimes_{B}Y, A)$ is an epimorphism$;$

(2) \ ${\rm Ext}_{A}^{i}(g, A): {\rm Ext}_{A}^{i}(X, A)\longrightarrow {\rm Ext}_{A}^{i}(N\otimes_{B}Y, A)$ is an isomorphism for $i\geq 1;$

(3) \ ${\rm Hom}_B(f, B):{\rm Hom}_{B}(Y, B)\lxr {\rm Hom}_B(M\otimes_AX, B)$ is an epimorphism$;$

(4) \  ${\rm Ext}_{B}^{i}(f, B):{\rm Ext}_{B}^{i}(Y, B)\longrightarrow {\rm Ext}_{B}^{i}(M\otimes_{A}X, B)$ is an isomorphism for $i\geq 1$.

\vskip5pt

\noindent It follows from Theorem \ref{wgmono}  that $f$ and $g$ are monomorphisms, $\cok f\in\Gp(B)$ and $\cok g\in\Gp(A)$. Thus, by Corollary \ref{gpmorita}, $\begin{pmatrix}\begin{smallmatrix}X\\Y\end{smallmatrix}\end{pmatrix}_{f, g}\in\Gp(\Lambda)$.
By definition $\Lambda$ is left weakly Gorenstein.

\vskip 5pt

$\Longrightarrow:$ \ Conversely, assume that $\Lambda$ is left weakly-Gorenstein. Let $X\in\ ^\perp A$. It is easy to see that $\begin{pmatrix}\begin{smallmatrix}X\\ M\otimes X\end{smallmatrix}\end{pmatrix}_{1, 0}$ satisfies all the four conditions in Theorem \ref{semi}, and hence it is a semi-Gorenstein projective $\Lambda$-module. By the assumption that $\Lambda$ is left weakly Gorenstein, one infers that $\begin{pmatrix}\begin{smallmatrix}X\\ M\otimes X\end{smallmatrix}\end{pmatrix}_{1, 0}\in\Gp(\Lambda)$. Hence, by Corollary \ref{gpmorita}, $\cok (0=N\otimes_B M\otimes_A X\lxr X)=X$ is a Gorenstein-projective $A$-module. Thus, $A$ is left weakly Gorenstein. Similarly, one can show that $B$ is left weakly Gorenstein. \hfill $\square$

\subsection{Application to $T_n(A)$} Let $A$ be an Artin algebra. Put  $T_n(A) =
\left(\begin{smallmatrix}
A&A&\cdots&A&A\\
0&A&\cdots&A&A\\
\vdots&\vdots&\ddots&\vdots&\vdots\\
0&0&\cdots&A&A\\ 0&0&\cdots&0&A
\end{smallmatrix}\right)$ to be the $n\times n$ upper triangular matrix algebra, $n\ge 2$. Then
$T_n(A) = \left(\begin{smallmatrix}
T_{n-1}(A)&N\\
0&A
\end{smallmatrix}\right)$ $T_n(A)$ is a Morita ring, where $N = \left(\begin{smallmatrix}
A\\
\vdots\\
A
\end{smallmatrix}\right)_{(n-1)\times 1}.$ Note that $_{T_{n-1}(A)}N$ and $N_A$ are projective modules.

\vskip5pt

A left $T_n(A)$-module can be identified with $\left(\begin{smallmatrix}
X_n\\
\vdots\\
X_1
\end{smallmatrix}\right)_{(\phi_i)},$
where $X_i\in A$-mod for all $i$, and $\phi_i: \ X_i\longrightarrow
X_{i+1}$ are $A$-maps for $1\le i\le n-1$.   Using induction on $n$, and Corollary \ref{triangsemi}(1),  Corollary \ref{gpmorita}, and Theorem \ref{weakGmorita}, respectively, one gets the following consequence.

 \vskip10pt

\begin{cor} \label{Tn(A)} \ $(1)$ \ {\rm (\cite[Corolary 4.2]{XZ})} \ \ $\left(\begin{smallmatrix}
X_n\\
\vdots\\
X_1
\end{smallmatrix}\right)_{(\phi_i)}\in \ ^\perp T_n(A)$ if and only if
\ $\phi^*_i: \ \Hom_A(X_{i+1}, A) \longrightarrow \Hom_A(X_i, A)$
are epimorphisms for $1\le i\le n-1$, and $X_i\in \ ^\perp A$ for
all $i$.

\vskip5pt

$(2)$ \ $\left(\begin{smallmatrix}
X_n\\
\vdots\\
X_1
\end{smallmatrix}\right)_{(\phi_i)}\in \Gp(T_n(A))$ if
and only if $X_i\in \Gp(A)$ for all
$i$, $\phi_i: X_i \longrightarrow X_{i+1}$ are monomorphisms, and
$\cok\phi_i\in \Gp(A)$ for $1\le i\le
n-1$.

\vskip5pt

$(3)$ \  $T_n(A)$ is a left weakly Gorenstein algebra if and only if so is $A$.\end{cor}

For the lower triangular matrix algebras one has the corresponding results.

\subsection{Examples}

\begin{exam}
\label{ex1}\ Let $k$ be a field and $A=B$ the path $k$-algebra of the quiver $3 \longrightarrow 2 \longrightarrow 1$.
We write the connection of paths from the right to the left.
Take $M=Ae_{1}\otimes_{k}e_{3}A$ and $N=Ae_{2}\otimes_{k}e_{3}A$.
Then $_{A}M,\ M_{A},\ _{A}N,\ N_{A}$ are projective modules, and  $M\otimes_{A}N=0=N\otimes_{A}M$.
By {\rm Corollary \ref{finitepd}(1)},  $\Lambda=\begin{pmatrix}\begin{smallmatrix}
A & N \\
M & A \\
\end{smallmatrix}\end{pmatrix}$  is {\rm GP}-free, thus, there are $6$ indecomposable Gorenstein-projective left $\Lambda$-modules, they are precisely the indecomposable projective left $\Lambda$-modules$:$
\[
\xymatrix@C=0.1cm@R=0.2cm{
{\begin{pmatrix}\begin{smallmatrix}Ae_1\\0\end{smallmatrix}\end{pmatrix}_{0, 0}}, \ \  {\begin{pmatrix}\begin{smallmatrix}Ae_2\\0\end{smallmatrix}\end{pmatrix}_{0, 0}},
\ \ {\begin{pmatrix}\begin{smallmatrix}Ae_3\\ Ae_1\end{smallmatrix}\end{pmatrix}_{\id_{Ae_{1}}, 0}},
\ \ {\begin{pmatrix}\begin{smallmatrix}0\\ Ae_1\end{smallmatrix}\end{pmatrix}_{\id_{Ae_{1}}, 0}},
\ \ {\begin{pmatrix}\begin{smallmatrix}0\\ Ae_2\end{smallmatrix}\end{pmatrix}_{\id_{Ae_{1}}, 0}},
\ \ {\begin{pmatrix}\begin{smallmatrix}Ae_2\\ Ae_3\end{smallmatrix}\end{pmatrix}_{0, \id_{Ae_2}}}.}
\]
\end{exam}

\vskip 5pt

\begin{exam}
\label{ex2}\ Let $kQ$ the path algebra of the quiver
\[
\xymatrix@C=0.9cm@R=0.2cm{1\ar[rr]^-\alpha & & 2 \ar[ld]^{\beta} \\
& 3 \ar[lu]^{\gamma} }
\]
and $A=B=kQ/J^{2}$, where $J$ is the ideal of $kQ$ generated by all the arrows.
Then $A$ is a self-injective  Nakayama algebra. Thus $A$ is of finite representation type and $A$ is weakly Gorenstein,  and  $$\Gproj(A)= \ ^\perp A = A\mbox{-}{\rm mod} = \{Ae_1, Ae_2, Ae_3, S_1, S_2, S_3\}.$$ Take $M=N=Ae_{1}\otimes_{k}e_{3}A$.
Then  $_{A}M$ and $M_{A}$ are projective modules, and $M\otimes_{A}N=0=N\otimes_{A}M$. Note that
$$M\otimes_A Ae_{1}=0, \ \ M\otimes_A Ae_{2} = Ae_1  = M\otimes_A Ae_3,  \ \ M\otimes_A S_1 = 0 = M\otimes_A S_2, \ \ M\otimes_A S_3 = Ae_1$$
as left $A$-modules.

\vskip 5pt

By {\rm Theorem \ref{weakGmorita}},  $\Lambda=\begin{pmatrix}\begin{smallmatrix}
A & M \\
M & A \\
\end{smallmatrix}\end{pmatrix}$ is a left weakly Gorenstein algebra. By {\rm Corollary \ref{gpmorita}},  there are $12$ indecomposable Gorenstein-projective left $\Lambda$-modules$:$
\[
\xymatrix@C=0.1cm@R=0.2cm{
{\begin{pmatrix}\begin{smallmatrix}0 \\ Ae_1\end{smallmatrix}\end{pmatrix}_{0, 0}},   & {\begin{pmatrix}\begin{smallmatrix}0\\S_1\end{smallmatrix}\end{pmatrix}_{0, 0}},
&  {\begin{pmatrix}\begin{smallmatrix}0\\ S_2\end{smallmatrix}\end{pmatrix}_{0, 0}},
&  {\begin{pmatrix}\begin{smallmatrix}Ae_1\\ 0 \end{smallmatrix}\end{pmatrix}_{0, 0}},
&  {\begin{pmatrix}\begin{smallmatrix}Ae_1\\ Ae_2\end{smallmatrix}\end{pmatrix}_{0, \id_{Ae_{1}}}},
&  {\begin{pmatrix}\begin{smallmatrix}Ae_1\\ Ae_3\end{smallmatrix}\end{pmatrix}_{0, \id_{Ae_1}}},
\\ {\begin{pmatrix}\begin{smallmatrix}Ae_1\\ S_3 \end{smallmatrix}\end{pmatrix}_{0, \id_{Ae_1}}},
&  {\begin{pmatrix}\begin{smallmatrix}Ae_2\\ Ae_1\end{smallmatrix}\end{pmatrix}_{\id_{Ae_1}, 0}},
&  {\begin{pmatrix}\begin{smallmatrix}Ae_3\\ Ae_1\end{smallmatrix}\end{pmatrix}_{\id_{Ae_1}, 0}},
&  {\begin{pmatrix}\begin{smallmatrix}S_1\\ 0\end{smallmatrix}\end{pmatrix}_{0, 0}},
&  {\begin{pmatrix}\begin{smallmatrix}S_2\\ 0\end{smallmatrix}\end{pmatrix}_{0, 0}},
&  {\begin{pmatrix}\begin{smallmatrix}S_3\\ Ae_1\end{smallmatrix}\end{pmatrix}_{\id_{Ae_1}, 0}}}
\]
and they are precisely all the indecomposable semi-Gorenstein-projective left $\Lambda$-modules.\end{exam}

\begin{exam} \label{thirdexam} \  Let $A$ be an Artin algebra such that there is a double semi-Gorenstein-projective $A$-module $X$ $($i.e., both $X$ and  $X^* = \Hom_A(X, A)$ are semi-Gorenstein-projective$)$, and that $X$ is not torsionless.
The existence of such an algebra $A$ is guaranteed by {\rm \cite {RZ2}}.  By {\rm Theorem \ref{weakGmorita}},  $\Lambda=\begin{pmatrix}\begin{smallmatrix}
A & A \\
0 & A \\
\end{smallmatrix}\end{pmatrix}$ is not left wekly Gorenstein. In fact, let $\varphi: X\longrightarrow P$ be a left $({\rm add} A)$-approximation of $X$.
By {\rm Theorem \ref{semi}} and {\rm Corollary \ref{gpmorita}},  $$\begin{pmatrix}\begin{smallmatrix}P\\ X\end{smallmatrix}\end{pmatrix}_{0, \varphi}, \ \ \ \begin{pmatrix}\begin{smallmatrix}X\\ 0\end{smallmatrix}\end{pmatrix}_{0, 0}, \ \ \ \begin{pmatrix}\begin{smallmatrix}X\\ X\end{smallmatrix}\end{pmatrix}_{0, 1}$$
are semi-Gorenstein-projective $\Lambda$-modules, but they are not Gorenstein-projective $\Lambda$-modules.
Moreover, all these three modules are double semi-Gorenstein-projective $\Lambda$-modules, which are not monic and not torsionless {\rm (\cite[Theorem 1.7, Lemma 5.6]{Z})}.

\end{exam}

\begin{exam} \ Let $A$ be a finite-dimensional algebra over field $k$.  Consider the matrix algebra
$$\Lambda=\begin{pmatrix}\begin{smallmatrix} A&A&0&A\\ 0&A&0&0 \\ 0&A&A&A \\ 0&0&0&A \end{smallmatrix}\end{pmatrix}\cong A\otimes kQ
= A\otimes k\left(\begin{smallmatrix}1 &\longleftarrow &2\\ \uparrow &&\downarrow \\ 4&\longrightarrow &3\end{smallmatrix}\right) $$
 Notice that $\Lambda$ can also be viewed as a Morita algebra $\Lambda=\begin{pmatrix}\begin{smallmatrix} B & M \\ M & B \end{smallmatrix}\end{pmatrix}$,
 where $B=T_2(A)=\begin{pmatrix}\begin{smallmatrix} A & A \\ 0 & A \end{smallmatrix}\end{pmatrix}\cong A\otimes k(2\lxr 1)$
 and the bimodule $M=\begin{pmatrix}\begin{smallmatrix} 0 & A \\ 0 & 0\end{smallmatrix} \end{pmatrix} =B(1\otimes e_1)\otimes (1\otimes e_2)B$ is projective as a left and right $B$-module, satisfying $M\otimes_BM=0$.

\vskip5pt

Thus, one can conclude from either {\rm Theorem \ref{thm:weak_G}} or {\rm Theorem  \ref{weakGmorita}} that $\Lambda$ is left weakly Gorenstein if and only if so is $A$.
Also, $\Lambda$ is Gorenstein if and only if $A$ is Gorenstein. This follows from  {\rm \cite[Proposition 2.2]{AR2}} or  {\rm \cite[Corollary 4.10, 4.15]{GP}}.

\vskip5pt

Assume that $A$ is not left weakly Gorenstein. We are going to explicitly construct some semi-Goresntein-projective $\Lambda$-modules which are not Gorenstein-projective.
The following equivalence allows us to interpret a $\Lambda$-module in two different ways.

\vskip5pt

A $B$-module is denoted by $\begin{pmatrix}\begin{smallmatrix} X\\Y\end{smallmatrix}\end{pmatrix}_{f}$, where $f:Y\longrightarrow X$ is an $A$-homomorphism.
A $\Lambda$-module is $\left( \begin{pmatrix}\begin{smallmatrix} X_1\\Y_1\end{smallmatrix}\end{pmatrix}_{f_1},
\begin{pmatrix}\begin{smallmatrix} X_2\\ Y_2\end{smallmatrix}\end{pmatrix}_{f_2},  \varphi_1, \varphi_2 \right)$,
where $$\varphi_1 \in \Hom_B(M\otimes_B \begin{pmatrix}\begin{smallmatrix} X_1\\Y_1\end{smallmatrix}\end{pmatrix}_{f_1},\begin{pmatrix}\begin{smallmatrix} X_2\\ Y_2\end{smallmatrix}\end{pmatrix}_{f_2})\cong \Hom_B(\begin{pmatrix}\begin{smallmatrix} Y_1\\0\end{smallmatrix}\end{pmatrix}_{0}, \begin{pmatrix}\begin{smallmatrix} X_2\\ Y_2\end{smallmatrix}\end{pmatrix}_{f_2})\cong \Hom_A(Y_1,X_2)$$
$$\varphi_2 \in \Hom_B(M\otimes_B \begin{pmatrix}\begin{smallmatrix} X_2\\Y_2\end{smallmatrix}\end{pmatrix}_{f_2},\begin{pmatrix}\begin{smallmatrix} X_1\\ Y_1\end{smallmatrix}\end{pmatrix}_{f_1})\cong \Hom_B(\begin{pmatrix}\begin{smallmatrix} Y_2\\0\end{smallmatrix}\end{pmatrix}_{0}, \begin{pmatrix}\begin{smallmatrix} X_1\\ Y_1\end{smallmatrix}\end{pmatrix}_{f_1})\cong \Hom_A(Y_2,X_1).$$

\vskip5pt
\noindent Let $g_1\in\Hom_A(Y_2,X_1)$ and $g_2\in\Hom_A(Y_1,X_2)$ be the image of $\varphi_1$ and $\varphi_2$ under the above isomorphism, respectively. Then
$$
\left( \begin{pmatrix}\begin{smallmatrix} X_1\\Y_1\end{smallmatrix}\end{pmatrix}_{f_1}, \begin{pmatrix}\begin{smallmatrix} X_2\\ Y_2\end{smallmatrix}\end{pmatrix}_{f_2},  \varphi_1, \varphi_2 \right) \ \ \mapsto \ \
\begin{smallmatrix} X_1 &\stackrel{f_1}\longleftarrow & Y_1\\ g_1\uparrow &&  \downarrow g_2 \\ Y_2 & \stackrel{f_2}\longrightarrow & X_2 \end{smallmatrix}
$$
yields an equivalence $\Lambda\modu\cong \rep(Q, A)$.

\vskip5pt

Let $A$ be a finite-dimensional $k$-algebra such that there is semi-Gorenstein-projective $A$-module $X$ and that $X$ is not torsionless.
The existence of such an algebra $A$ is guaranteed by {\rm \cite {RZ2}}.
Let $\alpha: X\longrightarrow P$ be a left $({\rm add}A)$-approximation of $X$.
As mentioned in {\rm Example \ref{thirdexam}},  $B=T_2(A)$ is not left wekly Gorenstein, and  $$\begin{pmatrix}\begin{smallmatrix}P\\ X\end{smallmatrix}\end{pmatrix}_{ \alpha}, \ \ \ \begin{pmatrix}\begin{smallmatrix}X\\ 0\end{smallmatrix}\end{pmatrix}_{0}, \ \ \ \begin{pmatrix}\begin{smallmatrix}X\\ X\end{smallmatrix}\end{pmatrix}_{1}$$
are semi-Gorenstein-projective $B$-modules, but they are not Gorenstein-projective $B$-modules.

\vskip5pt

Consider the following representation in $\rep(A,Q):$
$$G=\begin{smallmatrix} P &\stackrel{\alpha}\longleftarrow & X\\ \uparrow &&  \downarrow \mbox{\tiny $\alpha$}  \\ 0 &  \longrightarrow & P \end{smallmatrix}
$$
We claim that $G$ is a semi-Gorenstein-projective $\Lambda$-module which is not Goresntein-projective.

\vskip5pt

In fact, under the above equivalence, $G=\left(\begin{pmatrix}\begin{smallmatrix} P\\ X\end{smallmatrix}\end{pmatrix}_{\alpha}, \begin{pmatrix}\begin{smallmatrix} P\\ 0\end{smallmatrix}\end{pmatrix}_0,  \ \varphi_1, \ \varphi_2 \right)$, where
\begin{align*} \varphi_1=\begin{pmatrix}\begin{smallmatrix} \alpha\\0\end{smallmatrix}\end{pmatrix}: \ & \begin{pmatrix}\begin{smallmatrix} X\\0\end{smallmatrix}\end{pmatrix}_{0}\lxr \begin{pmatrix}\begin{smallmatrix} P\\0\end{smallmatrix}\end{pmatrix}_{0}
\\
\varphi_2=0 : \ & \begin{pmatrix}\begin{smallmatrix} 0\\0\end{smallmatrix}\end{pmatrix}_{0}\lxr \begin{pmatrix}\begin{smallmatrix} P\\X\end{smallmatrix}\end{pmatrix}_{\alpha}.\end{align*}
Since $\alpha: X\longrightarrow P$ is a left $({\rm add}A)$-approximation, it follows that $\Hom_A(\alpha, A): \Hom_A(P, A) \longrightarrow \Hom_A(X, A)$
is an epimorphism, and hence
$$\Hom_B(\varphi_1,B) = \begin{pmatrix}\begin{smallmatrix} \Hom(\alpha, A) &0 \\ 0& \Hom(\alpha, A)\end{smallmatrix}\end{pmatrix}: \Hom(P, A)\oplus \Hom(P, A) \longrightarrow \Hom(X, A)\oplus \Hom(X, A)$$
is an epimorphism for $i\ge 1$. It is clear that $\Hom(\varphi_2,B)$ is an epimorphism. Since both $\begin{pmatrix}\begin{smallmatrix} P\\ X\end{smallmatrix}\end{pmatrix}_{\alpha}$ and $\begin{pmatrix}\begin{smallmatrix} P\\ 0\end{smallmatrix}\end{pmatrix}_0$ are semi-Gorenstein-projective $B$-modules, it follows that $\Ext^i_B(\varphi_1,B)=0=\Ext^i_B(\varphi_2,B)$ are isomorphisms. Thus according to {\rm Theorem \ref{semi}}, $G$ is a semi-Goresntein-projective $\Lambda$-module. However, since $\varphi_1$ is not a monomorphism, $G$ is not monic and hence not a Goresntein-projective $\Lambda$-module.
\end{exam}

\section{\bf Appendix: Cartan-Eleinberg Isomorphism}
 On page 209 and 205 of \cite{CE56},  H. Cartan and S. Eilenberg gave the following
 isomorphism for computing $\Ext$-groups of tensor product of modules, which is a powerful tool in studying modules over tensor algebras.

 \begin{thm} \ {\rm (\cite[p.209, p.205]{CE56})} \ Let $A$ and $B$ be left Noetherian algebras over a semi-simple commutative ring $k$,
 $M$ a finitely generated left $A$-module,  and $X$ a finitely generated left $B$-module. Then for any left $A$-module $N$ and left $B$-module $Y$,  there is a functorial isomorphism$:$
$$ \Phi: \ \Ext_A(M,N)\otimes_k\Ext_B(X,Y)\longrightarrow \Ext_{A\otimes_kB}(M\otimes_k X, N\otimes_k Y).$$
 \end{thm}

In this appendix, we provide a direct and elementary proof of the Cartan-Eleinberg isomorphism
in the following context, based on the  K\"unneth formula for the homology of the tensor product of two complexes. In the following $\otimes = \otimes_k$.

  \begin{thm} \label{CEisofd}  {\rm (\cite[p.209, p.205]{CE56})} \ If $A$ and $B$ are finite-dimensional algebras over a field $k$.
  Let $M$ and $N$ be finitely generated left $A$-modules, and $X$ and $Y$ finitely generated left $B$-modules. Then

  \vskip5pt

$(1)$ \ There is a functorial isomorphism$:$
$$ \Phi:  \ \Hom_A(M,N)\otimes\Hom_B(X,Y)\longrightarrow \Hom_{A\otimes B}(M\otimes X, N\otimes Y),$$
where $\Phi(\sum\limits_i f_i\otimes g_i)(m\otimes x)=\sum\limits_i f_i(m)\otimes g_i(x)$.

\vskip5pt

$(2)$ \ There is a functorial isomorphism for $n\geq 0:$
$$\Phi: \ \bigoplus\limits_{i=0}^n\Ext^i_A(M,N)\otimes \Ext^{n-i}_B(X,Y)\longrightarrow \Ext^n_{A\otimes B}(M\otimes  X, N\otimes  Y).$$
\end{thm}

Note that for our application in this paper,
we need the concrete form of the isomorphism $\Phi$ for the case of $n =0$, and even in this case of $n = 0$, the inverse of $\Phi$ seems to be not available.

\subsection{The isomorphism between $\Hom$-spaces} \ We will first prove Theorem \ref{CEisofd}(1). In a sketch, this will be done in steps:
\begin{enumerate}
\item \ $\Phi$ is injective.
\item \ $\Phi$ is functorial in each variable.
\item \ $\Phi$ is an isomorphism for free modules $M, X$.
\item  \ $\Phi$ is an isomorphism for any $M, X$.
\end{enumerate}

\vskip5pt

\begin{lem} \label{basiciso} {\rm (\cite[p.116]{ASS})} \ Let $A$ be a finite-dimensional $k$-algebra,  $M$ a finite-dimensional left $A$-module, and $P$ a finite-dimensional projective left $A$-module.   Then there is a functorial isomorphism
$\varphi: \ P^*\mathop\otimes_A M\lxr \Hom_A(P,M)$
defined by $\varphi(f\otimes m)(p)=f(p)m$, where $P^*=\Hom_A(P,A)$.
\end{lem}

If $A=k$, then this is precisely the well-known isomorphism between finite-dimensional vector spaces.

\vskip5pt

If a finite-dimensional vector spaces $V$ has a basis $\{v_i\}_{i\in I}$, then the dual basis of $V^\vee=\Hom_k(V,k)$ will be denoted by $\{v_i^\vee\}_{i\in I}$, where $v_i^\vee(v_j)=\delta_{ij}$.

\begin{lem} \label{spaceiso}
Let $M,N,X,Y$ be finite-dimensional $k$-linear spaces. Then there is an isomorphism
$$ \Phi_k:  \Hom_k(M,N)\otimes\Hom_k(X,Y)\longrightarrow \Hom_k(M\otimes X, N\otimes Y),$$
where $\Phi_k(\sum\limits_i f_i\otimes g_i)(m\otimes x)=\sum\limits_i f_i(m)\otimes g_i(x)$.
\end{lem}
\begin{proof}
Let $\{m_i\}_{i\in I}$, $\{n_j\}_{j\in J}$, $\{x_s\}_{s\in S}$, $\{y_t\}_{t\in T}$ be basis of $M, N, X, Y$ respectively. By Lemma \ref{basiciso}, the space $$\Hom_k(M,N)\otimes \Hom_k(X,Y)\cong (M^\vee\otimes N)\otimes (X^\vee\otimes Y)$$ has a $k$-basis $\{m_i^\vee\otimes n_j\otimes x^\vee_s\otimes y_t\}_{(i,j,s,t)\in I\times J\times S\times T}$; while the space $$ \Hom_k(M\otimes X, N\otimes Y)\cong (M\otimes X)^\vee\otimes(N\otimes Y)\cong M^\vee\otimes X^\vee\otimes N\otimes Y$$ has a basis $\{m_i^\vee\otimes x^\vee_s\otimes n_j\otimes y_t\}_{(i,j,s,t)\in I\times J\times S\times T}$. It is straightforward to check that $\Phi_k$ sends $m_i^\vee\otimes n_j\otimes x^\vee_s\otimes y_t$ to $m_i^\vee\otimes x^\vee_s\otimes n_j\otimes y_t$. \end{proof}

\begin{cor}\label{cor:Phi_inj}
The map $\Phi:\Hom_A(M,N)\otimes\Hom_B(X,Y)\longrightarrow \Hom_{A\otimes B}(M\otimes X, N\otimes Y)$ is injective.
\end{cor}
\begin{proof}
Denote by $\ell_A:\Hom_A(M,N)\longrightarrow \Hom_k(M,N)$ (respectively $\ell_B, \ell_{A\otimes B}$) the inclusion of $k$-subspaces. In the commutative diagram below
$$
\xymatrix{\Hom_A(M,N)\otimes \Hom_B(X,Y)\ar[d]_-{\ell_A\otimes \ell_B} \ar[r]^\Phi& \Hom_{A\otimes B}(M\otimes X, N\otimes Y)\ar[d]^{\ell_{A\otimes B}}\\
\Hom_k(M,N)\otimes \Hom_k(X,Y) \ar[r]^{\Phi_k}& \Hom_{k}(M\otimes X, N\otimes Y)}
$$
by Lemma \ref{spaceiso}, both $\Phi_k$ and $\ell_A\otimes \ell_B$ are injective, thus $\Phi$ is injective as well.
\end{proof}

The following fact is clear.

\begin{lem}\label{nat} \ The map $\Phi$ is natural in each variable.
\end{lem}

\begin{lem}\label{lem:free_dual}
Let $M$ be a left free $A$-module and $X$ a left free $B$-module. Then there is a right $(A\otimes B)$-module isomorphism $\Hom_A(M,A)\otimes \Hom_B(X,B)\cong \Hom_{A\otimes B}(M\otimes X, A\otimes B)$.
\end{lem}
\begin{proof} \ Without loss of generality one may assume $M = A^m$ and $X = B^n$. Then $\Hom_A(A^m,A)\otimes\Hom_B(B^n,B)\cong A^m\otimes B^n\cong (A\otimes B)^{mn}$, and $\Hom_{A\otimes B}(A^m\otimes B^n, A\otimes B)\cong \Hom_{A\otimes B}((A\otimes B)^{mn}, A\otimes B)\cong (A\otimes B)^{mn}$.
\end{proof}

\begin{lem}\label{lem:Phi_free}
When $M$ and $X$ are free modules,  $\Phi$ is a bijection.
\end{lem}
\begin{proof}
 Since we have shown that $\Phi$ is injective in Corollary \ref{cor:Phi_inj}, to see it is bijective, it suffices to show   $$\dim_k(\Hom_A(M,N)\otimes\Hom_B(X,Y))\geq \dim_k\Hom_{A\otimes B}(M\otimes X, N\otimes Y).$$
 In fact, since $M$ and $X$ are free, by Lemma \ref{basiciso} there are isomorphisms $$\Hom_A(M,N)\cong \Hom_A(M,A)\otimes_A N, \ \ \ \ \Hom_B(X,Y)\cong\Hom_B(X,B)\otimes_B Y$$
 as well as
\begin{align*}
\Hom_{A\otimes B}(M\otimes X, N\otimes Y) & \cong  \Hom_{A\otimes B}(M\otimes X, A\otimes B)\otimes_{A\otimes B} (N\otimes Y)\\
 & \cong  (\Hom_A(M,A)\otimes \Hom_B(X,B)) \otimes_{A\otimes B} (N\otimes Y)\end{align*}
where the second isomorphism is due to Lemma \ref{lem:free_dual}.
Notice that the map $$\Xi: (\Hom_A(M,A)\otimes_A N)\otimes (\Hom_B(X,B)\otimes_B Y) \longrightarrow(\Hom_A(M,A)\otimes \Hom_B(X,B)) \otimes_{A\otimes B} (N\otimes Y)$$
given by $\Xi(f\otimes_A n\otimes g\otimes_B y)=(f\otimes g)\otimes_{A\otimes B}(n\otimes y)$ is clearly a surjection. This completes the proof.
\end{proof}

{\bf Proof of Theorem \ref{CEisofd}(1)} \ First, assume that $X$ is a finite-dimensional free $B$-module. Let $A_1\longrightarrow A_0\longrightarrow M\longrightarrow 0$ be a free presentation of $M$. In the following diagram with exact rows:
$$
\xymatrix@C=12pt@R=20pt{0\ar[r]& \Hom_A(M,N)\otimes \Hom_B(X,Y)\ar[r]\ar[d]_\Phi &\Hom_A(A_0,N)\otimes \Hom_B(X,Y)\ar[r]\ar[d]_\Phi^\wr &\Hom_A(A_1,N)\otimes \Hom_B(X,Y)\ar[d]_\Phi^\wr\\
0\ar[r]&\Hom_{A\otimes B}(M\otimes X, N\otimes Y)\ar[r]& \Hom_{A\otimes B}(A_0\otimes X, N\otimes Y)\ar[r]&\Hom_{A\otimes B}(A_1\otimes X, N\otimes Y)}
$$
the two squares commute, due to Lemma \ref{nat}.  Since the second and third vertical maps are bijections, by Lemma \ref{lem:Phi_free}, so is the first one.

\vskip5pt

Now, for an arbitrary finite-dimensional $B$-module $X$, take a free presentation $B_1\longrightarrow B_0\longrightarrow X\longrightarrow 0$.
In the following diagram with exact rows:
$$
\xymatrix@C=12pt@R=20pt{0\ar[r]& \Hom_A(M,N)\otimes \Hom_B(X,Y)\ar[r]\ar[d]_\Phi &\Hom_A(M,N)\otimes \Hom_B(B_0,Y)\ar[r]\ar[d]_\Phi^\wr &\Hom_A(M,N)\otimes \Hom_B(B_1,Y)\ar[d]_\Phi^\wr\\
0\ar[r]&\Hom_{A\otimes B}(M\otimes X, N\otimes Y)\ar[r]& \Hom_{A\otimes B}(M\otimes B_0, N\otimes Y)\ar[r]&\Hom_{A\otimes B}(M\otimes B_1, N\otimes Y)}
$$
Since we have shown the second and third vertical maps are bijections, so is the first one.   \hfill \qed

\subsection {The isomorphism between $\Ext$-groups}
From Theorem \ref{CEisofd}(1), we can deduce the Catan-Eilenberg isomorphisms for higher $\Ext$-groups. The key ingredient is the  K\"unneth formula, which we will recall first.

\vskip5pt

Let $(X^\bullet, d)$ be a complex of right $A$-modules and $(Y^\bullet,\partial)$ a complex of left $A$-modules. Then $(X^\bullet\mathop\otimes_A Y^\bullet, \delta)$ is a complex of abelian groups defined by $(X^\bullet\mathop\otimes_A Y^\bullet)_n=\bigoplus\limits_{i+j=n}(X_i\mathop\otimes_A Y_j)$,  and
$$\delta_n(x_i\otimes_A y_j)=dx_i\otimes_A y_j+(-1)^ix_i\otimes_A \partial y_j,  \ \forall \ x_i\in X_i, \  y_j\in Y_j.$$

\begin{thm}\label{thm:Kunneth} {\rm (\cite[Theorem 3.6.3]{Wei})} \
Let $(X^\bullet, d)$ and $(Y^\bullet,\partial)$ be  right $A$-module and left $A$-module complexes, respectively.
If $X_n$ and $\Ima d_n$ are flat $A$-modules for all $n$, then there is an exact sequence$:$
$$
0\lxr \bigoplus\limits_{i+j=n}H_i(X^\bullet)\otimes_A H_j(Y^\bullet)\lxr H_n(X^\bullet\otimes_A Y^\bullet)\lxr \bigoplus\limits_{i+j=n-1}{\rm Tor}^1_A(H_i(X^\bullet), H_j(Y^\bullet))\lxr 0.
$$
\end{thm}

We are interested in a special form of the K\"unneth formula:
\begin{thm}\label{thm:kunneth}
Let $(X^\bullet, d)$ and $(Y^\bullet,\partial)$ be  right $A$-module and left $A$-module complexes, respectively. If $\Ker d_n$ and $H_n(X^\bullet)$ are projective for all $n$, then there is an isomorphism of abelian groups$:$
$$
H_n(X^\bullet\otimes_A Y^\bullet)\cong\bigoplus\limits_{i+j=n}H_i(X^\bullet)\otimes_A H_j(Y^\bullet).
$$
\end{thm}
\begin{proof} \ By the exact sequences $0\lxr \Ima d_{n-1}\lxr\Ker d_n\lxr H_n(X^\bullet)\lxr 0$ and
$0\lxr \Ker d_{n-1}\lxr X_n\lxr \Ima d_{n-1} \lxr 0$ one sees that $\Ima d_{n-1}$ and $X^n$ are projective modules for all $n$. So Theorem \ref{thm:Kunneth} applies.
\end{proof}

As an immediate corollary, we obtain a formula for projective resolutions of the tensors of modules.

\begin{cor} \ Let $A$ and $B$ be finite-dimensional $k$-algebras. Suppose that $(P^\bullet,d)$ is a projective resolution of a left $A$-module $X$ and $(Q^\bullet,\partial)$ is a projective resolution of a left $B$-module $Y$. Then $P^\bullet\otimes_k Q^\bullet$ is a projective resolution of the left $(A\otimes B)$-module $X\otimes Y$.
\end{cor}
\begin{proof} \ Consider $(P^\bullet,d)$ as a complex of right $k$-modules and $(Q^\bullet,\partial)$ as a complex of left $k$-modules. Then $P^\bullet\otimes_k Q^\bullet$ is a complex satisfying $H_n(P^\bullet\otimes_k Q^\bullet)\cong\bigoplus\limits_{i+j=n}H_i(P^\bullet)\otimes_k H_j(Q^\bullet)=0$ for $n>0$, and  $H_0(P^\bullet\otimes_k Q^\bullet)\cong X\otimes Y$ as left $(A\otimes B)$-modules.\end{proof}

{\bf Proof of Theorem \ref{CEisofd}(2)}  \ Let $(P^\bullet,d)$ and $(Q^\bullet,\partial)$ be projective resolutions of $M$ and $X$, respectively. Then $P^\bullet\otimes_k Q^\bullet$ is a projective resolution of $M\otimes X$. Applying $\Hom_{A\otimes B}(-,N\otimes Y)$ and using Theorem \ref{CEisofd}(1), we obtain a commutative diagram:
$$
\begin{tikzpicture}[baseline= (a).base]
\node[scale=0.6] (a) at (0,0){
\xymatrix{0\ar[r] &\Hom_{A\otimes B}(P_0\otimes Q_0, N\otimes Y)\ar[r]\ar[d]_\Phi^\wr&\Hom_{A\otimes B}(P_1\otimes Q_0\oplus P_0\otimes Q_1, N\otimes Y)\ar[r]\ar[d]_(0.4){\Phi}^(0.4){\wr}&\cdots\ar[r]& \Hom_{A\otimes B}(\bigoplus\limits_{i+j=n}P_i\otimes Q_j, N\otimes Y)\ar[r]\ar[d]_\Phi^\wr&\cdots\\
0\ar[r] &\Hom_A(P_0,N)\otimes \Hom_B(Q_0,Y)\ar[r]& \begin{smallmatrix}{\begin{matrix}\Hom_A(P_1,N)\otimes \Hom_B(Q_0,Y)\\ \oplus \\ \Hom_A(P_0,N)\otimes \Hom_B(Q_1,Y)\end{matrix}}\end{smallmatrix}\ar[r]&\cdots\ar[r]& \bigoplus\limits_{i+j=n}\Hom_A(P_i, N)\otimes \Hom_B( Q_j,Y)\ar[r]&\cdots
}
};
\end{tikzpicture}
$$
That is, there is an isomorphism of complexes $\Hom_{A\otimes B}(P^\bullet\otimes_k Q^\bullet, N\otimes Y)\cong\Hom_A(P^\bullet,N)\otimes \Hom_B(Q^\bullet,Y).$
Thus
\begin{eqnarray*}
  \Ext^n_{A\otimes B}(M\otimes X, N\otimes Y)&=& H_n\Hom_{A\otimes B}(P^\bullet\otimes_k Q^\bullet, N\otimes Y)\\
  &\cong& H_n(\Hom_A(P^\bullet,N)\otimes \Hom_B(Q^\bullet,Y))\\
  &\cong& \bigoplus\limits_{i+j=n} H_i\Hom_A(P^\bullet,N)\otimes H_j\Hom_B(Q^\bullet,Y)\\
  &=&\bigoplus\limits_{i+j=n} \Ext^i_A(M,N)\otimes \Ext^j_B(X,Y). \\
\end{eqnarray*}
This completes the proof. \hfill \qed

\end{document}